\documentclass[10pt,a4paper]{article}

\usepackage[english]{babel}
\usepackage[T1]{fontenc}

\usepackage{xcolor}
\usepackage[font=small,labelfont=sc]{subcaption}
\usepackage{amsmath}
\usepackage{amssymb}
\usepackage{amsthm}
\usepackage{float}
  \theoremstyle{plain}
  \newtheorem{theorem}{Theorem}
  \newtheorem{lemma}{Lemma}
  \newtheorem{proposition}[theorem]{Proposition}
  \newtheorem{corollary}{Corollary}
  \newtheorem{assumption}{Assumption}
  \newtheorem{definition}{Definition}[section]
  \newtheorem{example}{Example}[section]
  \theoremstyle{remark}
  \newtheorem*{remark}{Remark}
\usepackage{mathrsfs}
\usepackage{mathtools}
\usepackage[left=2.7cm,right=2.7cm,top=3.0cm,bottom=3.0cm]{geometry}
\usepackage[round]{natbib}
\usepackage[colorlinks=true, allcolors=blue, hypertexnames=false]{hyperref}
\usepackage{booktabs}
\usepackage{algorithmicx}
\usepackage{algorithm}
\usepackage{placeins}
\usepackage{algpseudocode}
\usepackage{amsmath}

\allowdisplaybreaks[4]
\begin{document}

\title{\bf Policy Evaluation and Policy Gradient for Continuous-Time Markov Regime-Switching Systems}

\author{Minghui Zhang
\thanks{School of Mathematics, Southeast University, Nanjing, China. E-mail: minghuibest@163.com}
\and
Xun Li\thanks{Department of Applied Mathematics, The Hong Kong Polytechnic University, Kowloon, Hong Kong, China. E-mail: li.xun@polyu.edu.hk}
\and
Xin Zhang\thanks{
School of Mathematics, Southeast University, Nanjing, China. E-mail: xzhangseu@seu.edu.cn}
}

\date{\today}

\maketitle

\begin{abstract}
We study policy evaluation (PE) and policy gradient (PG) for continuous-time stochastic control with Markov regime-switching. The main challenges are the coexistence of diffusion and regime jumps, the coupling of regime-dependent value functions, and the construction of model-free learning rules from the coupled generators. We establish a martingale characterization of PE that incorporates both Brownian and switching martingales, and extend martingale loss and orthogonality methods to this setting. We further derive a coupled PG equation and a trajectory-based gradient representation that, under an action-independent switching generator, avoids explicit knowledge of the system coefficients and transition rates. These results lead to offline and online actor-critic algorithms. We also establish convergence and study the finite-grid approximation required for implementation. Numerical experiments for a regime-switching linear-quadratic problem illustrate the proposed methods.

\end{abstract}

\section{Introduction}

Reinforcement learning provides a data-driven approach to stochastic control when the underlying dynamics are unknown or only partially known. A typical RL procedure consists of exploration, policy evaluation, and policy improvement. Policy evaluation estimates the value function under a given policy, while policy gradient methods improve a parameterized policy using the gradient of its performance criterion. Their combination leads naturally to actor-critic methods, in which the critic evaluates the current policy and the actor updates it. Most classical RL theory is developed for discrete-time Markov decision processes. Many stochastic control problems, however, evolve continuously in time and are more naturally modeled by stochastic differential equations.

A systematic continuous-time RL theory has developed rapidly in recent years. \citet{wang2020reinforcement} introduced an entropy-regularized exploratory control framework in which randomized policies are incorporated through stochastic relaxed controls. \citet{jia2022policyevaluation} subsequently established a martingale framework for PE. They showed that the value function can be characterized through the martingale property of the value process plus accumulated rewards. This perspective leads to martingale loss and martingale orthogonality methods, and also explains why the mean-square temporal-difference error is generally not an appropriate PE objective for stochastic diffusions. \citet{jia2022policygradient} further derived a continuous-time PG representation and showed that PG can be viewed as an auxiliary PE problem. This representation forms the basis of offline and online actor-critic algorithms.

The standard diffusion framework does not capture abrupt changes in the operating environment. Such changes arise naturally in financial markets, engineering systems, and economic models. A common approach is to introduce a finite-state continuous-time Markov chain that modulates the coefficients of the diffusion process. The resulting Markov regime-switching model contains both continuous fluctuations and discrete regime transitions. Its value function becomes regime dependent, and the corresponding dynamic programming equations form a weakly coupled system through the generator of the Markov chain. This hybrid structure creates additional difficulties for RL because the learning problem must account for both diffusion uncertainty and regime-transition risk.

Recent studies on continuous-time RL with regime-switching can be broadly divided according to their main objectives. One line focuses on financial control under partially observed regimes. \citet{wu2024reinforcement} study continuous-time mean–variance portfolio selection when the market regime is unobservable. They combine Wonham filtering with an entropy-regularized exploratory formulation and develop an actor-critic algorithm based on semi-analytical representations of the value function and policy. Thus, the main difficulty there is the partial-information structure: the hidden regime is first converted into a belief state, after which PE and policy updates are performed in the filtered state space. More recently, \citet{gao2026dividend} adopt a related idea for dividend optimization under a partially observed regime-switching diffusion. Their method again relies on belief-state filtering and on the special semi-analytical structure of the underlying control problem to parameterize the critic and actor.

A second line concentrates on exploratory control and policy improvement for specific regime-switching portfolio models. \citet{chen2025exploratory} study exploratory mean–variance portfolio optimization with regime-switching. They derive a policy improvement theorem and propose an orthogonality-condition learning method based on the martingale property of the optimal value process. Their analysis also points out difficulties associated with applying standard TD learning directly in this setting. \citet{xing2026exploratory} extend exploratory mean–variance optimization to regime-switching markets with constant-elasticity-of-variance dynamics. They first exploit the analytical structure of the HJB system and then construct a reinforcement learning scheme with linear approximations for the value function and policy. These studies demonstrate that regime information can be incorporated effectively into continuous-time exploratory control. Their analyses, however, are closely tied to the structure of particular portfolio models and their explicit or semi-explicit HJB solutions.

A different research direction treats the regime itself as part of the control. \citet{huang2025switching}Huang, Li, Yu, and Zhou (2025, revised 2026) study continuous-time RL for optimal switching among multiple regimes. In their formulation, the agent randomizes both the switching time and the destination regime by controlling the generator matrix of a finite-state Markov chain. They establish well-posedness of the exploratory HJB system, prove policy improvement and policy-iteration convergence, and construct a model-free algorithm using a martingale characterization of PE. This setting is conceptually different from the present one. In optimal switching, regime transitions are decision variables. Here, the Markov chain represents an exogenous random environment, while the control acts on the continuous state dynamics.

Regime-switching RL has also begun to appear in stochastic differential games. \citet{hu2026stackelberg} consider entropy-regularized linear-quadratic Stackelberg differential games under Markovian regime-switching. Their formulation leads to exploratory weakly coupled Hamilton–Jacobi–Bellman–Isaacs equations, and neural networks are used to approximate regime-dependent value functions in high-dimensional settings. This line emphasizes hierarchical equilibrium computation rather than a general PE–PG theory for a fixed exogenous regime process.

Closely related from a technical perspective is the recent extension of continuous-time RL from diffusions to jump-diffusions. \citet{gao2026jumpdiffusions} formulate exploratory control for general jump-diffusion dynamics and develop martingale-based RL methods in the presence of jumps. A key observation is that, although jumps alter the Hamiltonian and the exploratory HJB equation, certain sample-based PE and q-learning algorithms retain the same form because jump effects are already contained in observed value increments. They also establish a connection between the theoretical exploratory process and discretely sampled randomized controls. Regime-switching diffusions are structurally different from general jump-diffusions because the jump component is a finite-state Markov chain and directly couples a vector of regime-dependent value functions. Nevertheless, this work highlights an important principle for the present study: jump terms may be essential in the theoretical generator while being absorbed implicitly by observed increments in model-free learning algorithms.

These recent developments show that continuous-time RL with regime-switching is becoming an active research direction. At the same time, the existing literature is largely centered on particular control problems, hidden-regime filtering, optimal switching, or structured HJB solutions. What remains less developed is a general martingale-based theory that treats PE and PG jointly for a controlled diffusion evolving in an exogenous Markovian regime environment. In particular, several issues do not follow directly from the pure-diffusion theory.

The main difficulty arises from the hybrid nature of the regime-switching system. Under a fixed policy, the value process is no longer continuous because regime transitions introduce jumps in addition to the Brownian fluctuations of the state process. Its semimartingale decomposition therefore contains both continuous and discontinuous martingale components. This alters the predictable quadratic variation and introduces an additional contribution generated by regime changes. As a result, the connections among temporal-difference errors, quadratic variation, martingale loss, and martingale orthogonality cannot be obtained by a direct extension of the pure-diffusion theory. They must be re-established in a framework that explicitly accounts for the jump structure induced by the Markov chain.
This additional regime dependence also propagates into the policy-gradient analysis. The value functions associated with different regimes are coupled through the switching generator, and differentiating the policy-evaluation equations with respect to the policy parameters preserves this cross-regime coupling. Hence, the regime-switching terms are indispensable in the theoretical derivation. At the same time, a model-free implementation should avoid explicit estimation of the switching generator. This creates a nontrivial gap between the analytical representation and the learning rule. When the switching mechanism is independent of both the action and the policy parameters, the action-independent generator terms can be removed from the score-weighted representation, while the effects of realized regime jumps are already embedded in the observed increments of the value process. Establishing this separation is essential for deriving a trajectory-based PG representation that remains model-free despite the presence of regime-switching.
A further difficulty appears when the continuous-time theory is translated into implementable learning algorithms. The relaxed process used for the analysis is an averaged object, whereas the data available to the learner consist of state, regime, and randomized action trajectories. The randomized implementation must preserve the environmental martingale structure and the prescribed switching compensators, and regime-dependent predictable quantities must be evaluated with respect to the pre-jump regime. Moreover, finite-grid sampling of actions does not automatically reproduce the occupation law of the continuous-time relaxed model. Therefore, a separate consistency argument is required to justify the passage from the continuous-time PE and PG representations to discretized actor-critic algorithms.

Motivated by these issues, this paper develops a unified theory of PE and PG for continuous-time Markov regime-switching stochastic control. We consider a controlled diffusion whose drift, diffusion, and reward may depend on a finite-state exogenous Markov chain. The control is generated by a parameterized stochastic policy, while the switching generator is assumed to be independent of the control and policy parameters. Our analysis is carried out in continuous-time, and time discretization is introduced only at the implementation stage.

The contributions of this paper follow the progression from policy evaluation to policy improvement and, ultimately, to implementable actor-critic learning. We begin by establishing a martingale characterization of PE for stochastic policies under Markov regime-switching. The resulting formulation incorporates both Brownian and regime-transition martingales, which allows the martingale loss and martingale orthogonality approaches to be extended to the hybrid diffusion-jump setting. This analysis also identifies the additional contribution of regime transitions to the predictable quadratic variation and clarifies how the PE theory differs from its pure-diffusion counterpart. Building on this foundation, we derive the policy-gradient equation for the coupled regime-dependent value functions and obtain a probabilistic representation of the gradient. By combining the score identity with the semimartingale decomposition of the value process, the generator-based expression is transformed into a trajectory-based representation that avoids explicit evaluation of the drift, diffusion coefficients, and switching generator. These PE and PG results are then integrated into a continuous-time actor-critic framework. The critic is constructed from the PE martingale conditions, whereas the actor is driven by the trajectory-based PG representation in the offline setting and by first-order martingale orthogonality conditions in the online setting. To connect the continuous-time theory with practical implementation, we further analyze the finite-grid approximation and specify the predictability and the Assumption \ref{assum:occupation-matching} under which discretely sampled randomized controls are consistent with the underlying continuous-time formulation.

The remainder of the paper is organized as follows. Section \ref{sec:problem-formulation} presents the problem formulation and preliminaries. Section \ref{sec:PE} develops policy evaluation under Markov regime-switching. Section \ref{sec:PG} establishes the corresponding policy gradient theory. Section \ref{sec:Actor-Critic-Algorithms} constructs offline and online actor--critic algorithms based on the PE and PG results. Section \ref{sec:Applications-and-Numerical-Experiments} applies the proposed method to a regime-switching linear--quadratic control problem and reports numerical results. Section \ref{sec:Conclusions} concludes the paper.

\section{Problem formulation and Preliminaries}\label{sec:problem-formulation}

For a fixed time horizon $T>0$, let $(\Omega, \mathcal{F}, \mathbb{F}, \mathbb{P})$ be a filtered probability space with the filtration $\mathbb{F}:=\{\mathcal{F}_t\}_{t \in[0, T]}$ generated by a standard $n$-dimensional Brownian motion $W=\{W_t\}_{t \in[0, T]}$ and a continuous-time irreducible Markov chain $\alpha=\{\alpha_t\}_{t \in[0, T]}$ with a finite state space $\mathcal{M}=\{1, 2, \ldots, L\}$ such that $W_t$ and $\alpha_t$ are independent of each other. We assume the generator of the Markov chain $\alpha$ is $Q = (q_{ij})_{L \times L}$. Throughout the basic formulation, the regime process is assumed to be exogenous, meaning that its transition intensities do not depend on the control action.  For $i, j \in \mathcal{M}, i \neq j$, let $N_{i j}(t)$ denote the number of transitions of the Markov chain $\alpha$ from state $i$ to state $j$ up to time $t$. Define the compensated transition-counting process by
$\widetilde{N}_{i j}(t):=N_{i j}(t)-\int_0^t \mathbf{1}_{\{\alpha_{s-}=i\}} q_{i j} d s
$.
Then, for each $i \neq j, \widetilde{N}_{i j}$ is an $(\mathbb{F}, \mathbb{P})$-local martingale. %The total number of jumps into state $j$ can equivalently be written as $\Phi_j(t)=\sum_{i \neq j} N_{i j}(t)$, with compensated process $\widetilde{\Phi}_j(t)=\sum_{i \neq j} \widetilde{N}_{i j}(t)$.
%Let $\Phi_j(t)$ be the number of jumps into state $j$ up to time $t$ and set $\widetilde{\Phi}_j(t):= \Phi_j(t)-\int_0^t \sum_{i\neq j}^L q_{ij} 1_{\{\alpha_{s-=i}\}}(s) ds$. Then for each $j\in \mathcal{M},$ the process $\widetilde{\Phi}_j(t)$ is an $(\mathbb{F}, \mathbb{P})$-martingale.
For $i, j \in \mathcal{M}$ and $t \geq s \geq 0$, let $p_{i j}(t, s)$ denote the transition probability $P(\alpha_t=j | \alpha_s=i)$.

We consider the following stochastic differential equation (SDE)
\begin{equation}\label{eq:abstractmodel}
d X_s^{a}=b(s, X_s^{a}, \alpha_s, a_s) d s+\sigma(s, X_s^{a}, \alpha_s, a_s) d W_s,\quad s \in[t, T],
\end{equation}
where $a=\{a_s\}_{s\in [t, T]}$ is an $\mathbb{F}$-adapted process taking values on $\mathcal{A}$, representing the agent's action or control at time $s$. Here, $b:[0, T] \times \mathbb{R}^d \times \mathcal{M}\times \mathcal{A} \mapsto \mathbb{R}^d$, $\sigma:[0, T] \times \mathbb{R}^d \times \mathcal{M} \times \mathcal{A} \mapsto \mathbb{R}^{d \times n}$ are assumed to be measurable functions.
For each initial time-state-regime triple $(t, x, i) \in [0, T) \times \mathbb{R}^d \times \mathcal{M}$, the agent is interested in the stochastic control problem of finding the optimal control $a=\{a_s\}_{s\in [0,T]}$ that maximizes the following expected discounted total reward:
\begin{equation}\label{eq:objective}
\mathbb{E}\left[\int_t^T e^{-\beta(s-t)} r(s, X_s^{a}, \alpha_s, a_s) d s+ e^{-\beta(T-t)}h(X_T^{a}) \mid X_t^{a}=x, \alpha_t=i\right],
\end{equation}
where $r$ is \textit{the running reward function}, $h$ is \textit{the terminal reward function} applied at the end of the period $T$, and $\beta  \geq 0$ is a constant discount factor.

First, we define the test function class
\begin{align*}
\mathscr{D} :=& \Big\{ f:[0,T]\times\mathbb R^d\times\mathcal{M}\to\mathbb R:\;  f(\cdot,\cdot,i) \in C([0,T]\times\mathbb R^d) \cap C^{1,2}([0,T)\times\mathbb R^d),
i\in\mathcal{M}, \text{ and }\\
&\text{ there exists a compact set } K_f\subset\mathbb R^d
\text{ such that } f(t,x,i)=0 \text{ for all } x\notin K_f,\; (t,i)\in[0,T]\times\mathcal{M} \Big\}.
\end{align*}
Thus, every $f\in\mathscr{D}$ is continuously differentiable in the time and twice continuously differentiable in the state variable on $[0,T)\times\mathbb R^d$, is continuous up to the terminal time $T$, and has compact support in the state variable uniformly over $(t,i)\in[0,T]\times\mathcal{M}$.

For any action $a \in \mathcal{A}$ and test function $f(\cdot, \cdot, i) \in \mathscr D$, the extended infinitesimal generator $\mathcal{L}^a$ associated with the regime-switching diffusion $(s, X_s, \alpha_s)$ governed by \eqref{eq:abstractmodel} is defined by:
\begin{equation}\label{eq:the-controlled-extended-generator}
\begin{aligned}
\mathcal{L}^a f(t, x, i):= & \frac{\partial f}{\partial t}(t, x, i)+b(t, x, i, a)^{\top} \frac{\partial f}{\partial x}(t, x, i) \\
& +\frac{1}{2} \operatorname{tr}\left[\sigma(t, x, i, a) \sigma(t, x, i, a)^{\top} \frac{\partial^2 f}{\partial x^2}(t, x, i)\right]+\sum_{j=1}^L q_{i j}f(t, x, j),
\end{aligned}
\end{equation}
for $(t, x, i) \in[0, T) \times \mathbb{R}^d \times \mathcal{M}$. Here,
$\frac{\partial f}{\partial x}(t, x, i) \in \mathbb{R}^d$
is the gradient,
$\frac{\partial^2 f}{\partial x^2}(t, x, i) \in \mathbb{R}^{d \times d}$
is the Hessian. To ensure that the above stochastic control problem %\eqref{eq:abstractmodel}–\eqref{eq:objective}
is well-posed, we impose the following assumptions:
\begin{assumption}\label{assumption}
The following conditions for the state dynamics and reward functions hold true:
\begin{itemize}
\item[(i)] $b, \sigma, r, h$ are all continuous functions in their respective arguments;
\item[(ii)] $b, \sigma$ are uniformly Lipschitz continuous in $x$, i.e., for $\varphi \in\{b, \sigma\}$, there exists a constant $C>0$ such that
\begin{equation*}
|\varphi(t, x, i, a)-\varphi\left(t, x^{\prime}, i, a\right)| \leq C|x-x^{\prime}|, \;\; \forall(t, i, a) \in[0, T] \times \mathcal{M} \times \mathcal{A}, \;\; \forall x, x^{\prime} \in \mathbb{R}^d ;
\end{equation*}
\item[(iii)] $b, \sigma$ have linear growth in $x$, i.e., for $\varphi \in\{b, \sigma\}$, there exists a constant $C>0$ such that
\begin{equation*}
|\varphi(t, x, i, a)| \leq C(1+|x|), \quad \forall(t, x, i, a) \in[0, T] \times \mathcal{M} \times \mathbb{R}^d \times \mathcal{A};
\end{equation*}
\item[(iv)] $r$ and $h$ have polynomial growth in $(x, i, a)$ and $x$ respectively, i.e., there exists constants $C>0$ and $c\geq 1$ such that
\begin{equation*}
|r(t, x, i, a)| \leq C(1+|x|^{c}+|a|^{c}), \; |h(x)| \leq C(1+|x|^{c}),\; \forall(t, x, i, a) \in[0, T] \times \mathcal{M} \times \mathbb{R}^d \times \mathcal{A}.
\end{equation*}
\end{itemize}
\end{assumption}

We now introduce the RL formulation of the problem to be studied in this paper. In what follows, we write $\mathcal{P}(\mathcal{A})$ for the set of probability densities on the action space $\mathcal{A}$. Similar to the approach in \citet{jia2022policyaa,jia2022policymm,jia2023q}, we model exploration by sampling the action or control $a$ according to a stochastic policy $\pi = \{\pi(a|t, X_t, \alpha_t)\}_{t \geq 0}$, where $\pi(\cdot \mid t, x, i) \in \mathcal{P}(\mathcal{A})$. Let $(E, \mathcal{B}(E))$ be a Borel space and introduce an auxiliary probability space $\left(\Omega^{\varsigma}, \mathcal{F}^{\varsigma}, \mathbb{P}^{\varsigma}\right)$, together with an $E$-valued random variable $\varsigma$ and a measurable mapping $\phi:[0, T] \times \mathbb{R}^d \times \mathcal{M} \times E \rightarrow \mathcal{A}$
such that, for every $(t, x, i)$,
$\phi(t, x, i, \varsigma) \sim \pi(\cdot \mid t, x, i) $. Let $\mathbb{N}_0=\mathbb{N} \cup\{0\}$ and let $\left(\Omega^{\varsigma_{k}}, \mathcal{F}^{\varsigma_{k}}, \mathbb{P}^{\varsigma_{k}}, \varsigma_{k}\right)_{k \in \mathbb{N}_0}$ be independent copies. We extend the original space to
\begin{equation}\label{eq:product-probability-space}
\left(\Omega^{\prime}, \mathcal{F}^{\prime}, \mathbb{Q}\right) := \left(\Omega \times \prod_{k=0}^{\infty} \Omega^{\varsigma_{k}}, \mathcal{F} \otimes \bigotimes_{k=0}^{\infty} \mathcal{F}^{\varsigma_{k}}, \mathbb{P} \otimes \bigotimes_{k=0}^{\infty} \mathbb{P}^{\varsigma_{k}}\right).
\end{equation}
Consider the time grid $G_{t: T}:=\left\{t=s_0<s_1<\cdots<s_{n_0}=T\right\}$.
Let $\mathcal{F}_s^{\varsigma} := \sigma(\{\varsigma_k : s_k \leq s\})$ be the $\sigma$-algebra generated by the exploration noise up to time $s$, and define the augmented filtration $\mathbb{F}^{\prime}=(\mathcal{F}_s \vee \mathcal{F}_s^{\varsigma})_{s \in[t, T]}$. At each decision time $s_k$, the randomized action is generated by $
a_{s_k}=\phi\left(s_k, X_{s_k}, \alpha_{s_k}, \varsigma_k\right)$,
so that, conditionally on the current state and regime, $a_{s_k} \sim \pi\left(\cdot \mid s_k, X_{s_k}, \alpha_{s_k}\right)$.
Moreover, $a_{s_k}$ is $\mathcal{F}_{s_k}^{\prime}$-measurable. For the finite-grid implementation, the sampled action is held fixed over each decision interval, $a_s=a_{s_k}$, $s \in[s_k, s_{k+1})$.
Since the exploration variables $\{\varsigma_k\}_{k \geq 0}$ are independent of the original stochastic system, enlarging the probability space in this manner does not alter the environmental dynamics. In particular, $W$ remains an $n$-dimensional Brownian motion with respect to the enlarged filtration $\mathbb{F}^{\prime}=(\mathcal{F}_s^{\prime})_{s \in[t, T]}$, while $\alpha$ remains a continuous-time Markov chain with the same generator $Q$. Thus, the product-space construction introduces the additional randomness required to implement the stochastic policy without changing the underlying Brownian or regime-switching dynamics.

The agent samples actions according to the policy $\pi$ exclusively at the grid points and holds them constant over the adjacent intervals. To rigorously establish the exploratory state process under discrete-time sampling, we first provide its piecewise definition in integral form on the time grid. Specifically, for all $k=0, 1, \ldots, n_0-1$ and $s \in\left[s_k, s_{k+1}\right]$, the state process with exploratory actions, denoted by $X_s^\pi$, is strictly defined by the following integral equation:
\begin{equation}\label{eq:X_s_pi}
X_s^\pi = X_{s_k}^\pi + \int_{s_k}^s b\left(\ell, X_{\ell}^\pi, \alpha_{\ell}, a_{s_k}\right) d \ell + \int_{s_k}^s \sigma\left(\ell, X_{\ell}^\pi, \alpha_{\ell}, a_{s_k}\right) d W_{\ell},
\end{equation}
with the initial conditions $X_{s_0}^\pi = x$ and $\alpha_{s_0} = i$, and the sampled action at the grid points given by $a_{s_k} = \phi\left(s_k, X_{s_k}^\pi, \alpha_{s_k}, \varsigma_k\right)$. Equivalently, define the piecewise constant control process
$a_s^\pi=\sum_{k=0}^{n_0-1} a_{s_k} \mathbf{1}_{(s_k, s_{k+1}]}(s)$, $s \in(t, T]$,
with $a_t^\pi=a_{s_0}$. Since each $a_{s_k}$ is $\mathcal{F}_{s_k}^{\prime}$-measurable, $a^\pi$ is an $\mathbb{F}^{\prime}$-predictable step process. The state dynamics over the entire horizon can therefore be written compactly as
\begin{equation}\label{eq:dX_s_pi}
d X_s^\pi =  b(s, X_s^\pi, \alpha_s, a_s^\pi) d s + \sigma(s, X_s^\pi, \alpha_s, a_s^\pi) d W_s.
\end{equation}
And the objective functional is defined by
\begin{equation}\label{eq:value-fun(X_s_pi)}
J^0(t, x, i; \pi)=\mathbb{E}\Bigg[\int_t^T e^{-\beta(s-t)} r(s, X_s^{\pi}, \alpha_s, \pi(a | s, X_s^{\pi}, \alpha_s)) d s+e^{-\beta(T-t)} h(X_T^{\pi}) | X_t^{\pi}=x,\alpha_t=i\bigg].
\end{equation}

The grid-based sampled-and-held process in \eqref{eq:dX_s_pi} provides a realization of the randomized policy, but it is not the appropriate continuous-time Markov representation for deriving the exploratory HJB equation directly from the dynamic programming principle. Indeed, between two consecutive decision times, the state dynamics depend on the action sampled at the previous grid point. Consequently, the pair $(X_s^\pi,\alpha_s)$ alone does not encode the complete local dynamics of the finite-grid implementation. To obtain a Markovian representation suitable for continuous-time dynamic programming, we therefore introduce the relaxed state process obtained by averaging the local characteristics under the stochastic policy.

For a given policy $\pi$, define the relaxed drift and covariance by
\begin{equation}\label{eq:the-relaxed-drift-covariance}
\tilde{b}(t, x, i, \pi):=\int_{\mathcal{A}} b(t, x, i, a) \pi(d a \mid t, x, i),\quad
\tilde{c}(t, x, i, \pi):=\int_{\mathcal{A}} \sigma(t, x, i, a) \sigma(t, x, i, a)^{\top} \pi(d a \mid t, x, i) .
\end{equation}
The covariance, rather than the diffusion coefficient itself, is averaged because the second-order term in the infinitesimal generator is determined by the instantaneous quadratic covariance $\sigma \sigma^{\top}$. Let $\widetilde{\sigma}(t, x, i, \pi)$ be a measurable matrix-valued function satisfying
$\widetilde{\sigma}(t, x, i, \pi) \widetilde{\sigma}(t, x, i, \pi)^{\top}=\tilde{c}(t, x, i, \pi)$.
A canonical choice is the symmetric nonnegative square root, $\widetilde{\sigma}(t, x, i, \pi)=[\tilde{c}(t, x, i, \pi)]^{1 / 2}$.
Thus, \textit{the exploratory state dynamics} with $\pi$ is then defined, for $s \in[t, T]$,
\begin{equation}\label{eq:state-Exp}
d \widetilde{X}_s^\pi=\widetilde{b}(s, \widetilde{X}_s^\pi,\alpha_s,\pi(\cdot| s, \widetilde{X}_s^\pi, \alpha_s)) d s+\widetilde{\sigma}(s, \widetilde{X}_s^\pi,\alpha_s,\pi(\cdot| s, \widetilde{X}_s^\pi, \alpha_s)) d \widetilde{W}_s,
\end{equation}
with $\tilde{X}_t=x$, $\alpha_t=i$. Here $\widetilde{W}$ denotes a Brownian motion on the probability space supporting the relaxed state process. The use of the pre-jump regime $\alpha_{s-}$ ensures predictability of the regime-dependent stochastic integrand. %Since the regime-switching process is exogenous in the present formulation, its generator $Q$ is unaffected by policy randomization.

\begin{lemma}[Well-posedness of the relaxed dynamics]
Under Assumption \ref{assumption} and for any admissible policy $\pi$, \eqref{eq:state-Exp} admit a solution for every $(t, x, i) \in[0, T] \times \mathbb{R}^d \times \mathcal{M}$. Its law is denoted by $\mathbb{P}_{t, x, i}^\pi$.
If, in addition, $\widetilde{b}$ and $\widetilde{c}$ satisfy the standard Lipschitz and growth conditions, the solution is unique in law.
\end{lemma}

To promote exploration while retaining the original reward criterion, we augment the instantaneous reward by a policy-dependent regularization term. Let
$p:[0, T] \times \mathbb{R}^d \times \mathcal{M} \times \mathcal{A} \times \mathcal{P}(\mathcal{A}) \rightarrow \mathbb{R}$
be a regularizer, and let $\gamma \geq 0$ denote its weighting parameter. For a stochastic policy $\pi$, define the regularized instantaneous reward by
$r(t, x, i, a)+\gamma p(t, x, i, a, \pi(\cdot \mid t, x, i)) $.
Averaging over the randomized action yields the relaxed running reward
$$
\mathcal{R}(t, x, i, \pi):=\int_{\mathcal{A}}[r(t, x, i, a)+\gamma p(t, x, i, a, \pi(\cdot \mid t, x, i))] \pi(d a \mid t, x, i) .
$$
Accordingly, for an admissible policy $\pi$, its value function is defined as
\begin{equation}\label{eq:value-fun(MC)}
J(t, x, i; \pi)=\mathbb{E}\bigg[\int_t^T e^{-\beta(s-t)} \mathcal{R}(s, \tilde{X}_s^{\pi}, \alpha_s, \pi(a | s, \tilde{X}_s^{\pi}, \alpha_s)) d s +e^{-\beta(T-t)} h(\tilde{X}_T^{\pi}) \big| \tilde{X}_t^{\pi}=x,\alpha_t=i\bigg].
\end{equation}
In the above, $\gamma \geq 0$ a weighting parameter on exploration, also known as the temperature parameter. When $\gamma=0$, the criterion reduces to the standard expected discounted reward. A commonly used choice is the entropy regularizer
$p(t, x, a, \pi(\cdot))=-\log \pi(a)$.
Then the extended infinitesimal generator of the relaxed process is given by
\begin{align}\label{eq:the-extended-infinitesimal-generator}
\widetilde{\mathcal{L}}^\pi f(t, x, i)
:=& \frac{\partial f}{\partial t}(t, x, i)+\tilde{b}(t, x, i, \pi)^{\top} \frac{\partial f}{\partial x}(t, x, i)  +\frac{1}{2} \operatorname{tr}\left[\tilde{c}(t, x, i, \pi) \frac{\partial^2 f}{\partial x^2}(t, x, i)\right]  +\sum_{j=1}^L q_{i j}f(t, x, j) \notag\\
= & \frac{\partial f}{\partial t}(t, x, i)  +\int_{\mathcal{A}} b(t, x, i, a)^{\top} \frac{\partial f}{\partial x}(t, x, i) \pi(a \mid t, x, i) d a \notag\\
& +\frac{1}{2} \int_{\mathcal{A}} \operatorname{tr}\left[\sigma(t, x, i, a) \sigma(t, x, i, a)^{\top} \frac{\partial^2 f}{\partial x^2}(t, x, i)\right] \pi(a \mid t, x, i) d a +\sum_{j=1}^L q_{i j}f(t, x, j)\notag\\
=&\int_{\mathcal{A}} \mathcal{L}^a f(t, x, i) \pi(a \mid t, x, i) d a .
\end{align}
The task of RL is to find the optimal policy $\pi^*$ and its value function $J(t,x,i;\pi^*)$ such that
%\textit{The optimal value function} is then
\begin{equation}\label{eq:opt-value-fun}
V(t, x, i):=J(t,x,i;\pi^*)=\sup_{\pi \in \Pi} J(t, x, i; \pi),
\end{equation}
where $\Pi$ stands for the set of admissible (stochastic) policies on $\mathcal{A}$ and $V(t, x, i)$ is called \textit{the optimal value function}. We next give the precise definition of an admissible (feedback) policy.

\begin{definition}\label{def:admissible}
A policy $\pi=\pi(\cdot | \cdot, \cdot, \cdot)$ is called admissible if
\begin{itemize}
\item[(i)] $\pi(\cdot | t, x, i) \in \mathcal{P}(\mathcal{A}), supp \,\pi(\cdot | t, x, i)=\mathcal{A}$ for every $(t, x, i) \in[0, T] \times \mathbb{R}^d \times \mathcal{M}$ and $\pi(a | t, x, i):(t, x, i, a) \in[0, T] \times \mathbb{R}^d \times \mathcal{M}\times \mathcal{A} \mapsto \mathbb{R}$ is measurable;
\item[(ii)] the SDE \eqref{eq:state-Exp} has a unique solution for initial $(t, x, i)\in [0,T]\times\mathbb{R}^d \times \mathcal{M}$;
\item[(iii)] $\boldsymbol{\pi}(a | t, x, i)$ is continuous in $(t, x)$ and uniformly Lipschitz continuous in $x$ in the total variation distance, i.e., for any $i\in \mathcal{M}, \lim_{(t^{\prime}, x^{\prime}) \rightarrow (t, x)}\int_{\mathcal{A}}|\pi(a | t, x, i)-\pi(a | t^{\prime}, x^{\prime}, i)| d a =0$, and there is a constant $C>0$ independent of $(t, x, i)$ such that
\begin{equation*}
\int_{\mathcal{A}}|\pi(a | t, x, i)-\pi(a | t, x^{\prime}, i)| d a \leqslant C|x-x^{\prime}|, \quad \forall x, x^{\prime} \in \mathbb{R}^d, \quad i\in \mathcal{M};
\end{equation*}
\item[(iv)] for any $i\in \mathcal{M}$, $\int_{\mathcal{A}}|r(t, x, i, a)+\gamma p(t, x, a, \boldsymbol{\pi}(\cdot \mid t, x, i))| \boldsymbol{\pi}(a \mid t, x, i) \mathrm{d} a \leq C\left(1+|x|^\mu\right), \forall(t, x)$ where $C>0$ and $\mu \geq 1$ are constants;

\end{itemize}
\end{definition}

Under the above conditions, we obtain the following result.

\begin{lemma}\label{lem:growth-condition}
Let Assumptions \ref{assumption} hold and $\boldsymbol{\pi}$ be a given admissible policy. Consider $(t, x) \in[0, T] \times \mathbb{R}^n, \pi \in \Pi_t$ and $\left(\Omega^{\prime}, \mathcal{F}^{\prime}, \mathbb{F}^{\prime}, \mathbb{Q}\right)$ which is the probability space specified in \eqref{eq:product-probability-space}. Then, we have
\begin{itemize}
\item[(i)] for any grid $G_{t: T}$, the $\operatorname{SDE}$ \eqref{eq:dX_s_pi} with $X_t^\pi=x, \alpha_t=i$ has a unique strong solution. Moreover, for all $k \geq 2$, there exists a constant $C_k>0$ such that
\begin{equation*}
\mathbb{E}[|X_s^\pi|^k|X_t^\pi=x,\alpha_t=i] \leq C_k(1+|x|^k) e^{C_k s},
\end{equation*}
for any grid $G_{t: T}$ and $s \in[t, T]$.
\item[(ii)] the $S D E$ \eqref{eq:state-Exp} admits a unique strong solution that
satisfies the following growth condition
\begin{equation*}
    \mathbb{E}\left[\max _{t \leq s \leq T}|\tilde{X}_s^\pi|^n | \tilde{X}_t^\pi=x, \alpha_t=i\right] \leq C\left(1+|x|^n\right), \quad n\geq 2,
\end{equation*}
for some constant $C=C(n,i)$. Finally, the expected payoff \eqref{eq:value-fun(MC)} is finite.
\end{itemize}
\end{lemma}

\begin{proof}
\begin{itemize}
\item[(i)] We first construct the sequence of stopping times corresponding to the discontinuity points of the coefficients. Let $\mathcal{T}_\delta := \{s_0, s_1, \ldots, s_{n_0}\}$ denote the deterministic time grid for action sampling. Let $\mathcal{T}_\alpha := \{\tau_m\}_{m \geq 1}$ be the sequence of jump times of the Markov chain $\alpha$, defined recursively by $\tau_m := \inf \{t > \tau_{m-1} : \alpha_t \neq \alpha_{\tau_{m-1}}\}$ with $\tau_0 = 0$. Let $\mathcal{T} := (\mathcal{T}_\delta \cup \mathcal{T}_\alpha) \cap [0, T]$ be the merged set of stopping times. Since the number of grid points is finite ($|\mathcal{T}_\delta| = n_0 + 1 < \infty$) and the finite-state Markov chain is non-explosive within a finite horizon (i.e., $\mathbb{P}(\lim_{m \to \infty} \tau_m \leq T) = 0$), the cardinality of $\mathcal{T}$ is almost surely finite. We order the elements of $\mathcal{T}$ in an ascending sequence of $\mathbb{F}^{\prime}$-stopping times:
$$
0 = \rho_0 < \rho_1 < \rho_2 < \dots < \rho_{N} = T, \,\, \text{where }\, \mathbb{P}(N < \infty) = 1.
$$
We construct the global strong solution using mathematical induction on the interval index $m \in \{0, 1, \ldots, N-1\}$. For $m=0$. On the stochastic interval $[\rho_0, \rho_1]$, both the control action $a_0^\pi = a_{s_0}$ and the regime $\alpha_0 = i$ remain constant and are $\mathcal{F}_{\rho_0}^{\prime}$-measurable. Equation \eqref{eq:dX_s_pi} restricted to this interval is equivalent to a standard SDE with parameters fixed given the initial filtration $\mathcal{F}_{\rho_0}^{\prime}$:
$$
X_s = x + \int_0^s b(\ell, X_\ell, i, a_0^\pi) d\ell + \int_0^s \sigma(\ell, X_\ell, i, a_0^\pi) dW_\ell.
$$
It follows from Assumption \ref{assumption}-(i) and (iii), there exists a unique continuous strong solution $X^{(0)}$ adapted to $\mathcal{F}_s^{\prime}$ such that $\mathbb{E}[\sup_{s \in [0, \rho_1]} |X_s^{(0)}|^2| X_t^\pi=x, \alpha_t=i] < \infty$.

For the inductive step, consider $m \in \{1, \ldots, N-1\}$. Suppose a unique continuous strong solution $X^\pi$ has been well-established on the interval $[0, \rho_m]$. On the subsequent stochastic interval $[\rho_m, \rho_{m+1}]$, we consider the initial value problem:
\begin{equation}\label{eq:lem_X_s}
X_s = X_{\rho_m}^\pi + \int_{\rho_m}^s b(\ell, X_\ell, \alpha_{\rho_m}, a_{\delta(\rho_m)}) d\ell + \int_{\rho_m}^s \sigma(\ell, X_\ell, \alpha_{\rho_m}, a_{\delta(\rho_m)}) dW_\ell,
\end{equation}
where the initial condition $X_{\rho_m}^\pi$, the initial regime $\alpha_{\rho_m}$, and the sampled action $a_{\delta(\rho_m)}$ are all $\mathcal{F}_{\rho_m}^{\prime}$-measurable random variables. Crucially, these variables are completely independent of the future Brownian increments $(W_s - W_{\rho_m})_{s \geq \rho_m}$. By the classical Itô existence and uniqueness theorem, conditional on $\mathcal{F}_{\rho_m}^{\prime}$, the SDE \eqref{eq:lem_X_s} admits a unique strong solution $X^{(m)}$ on $[\rho_m, \rho_{m+1}]$. By invoking the pasting lemma, we can construct the global solution over the entire horizon $[0, T]$ as:
$$
X_s^\pi := \sum_{m=0}^{N-1} X_s^{(m)} \mathbf{1}_{[\rho_m, \rho_{m+1})}(s) + X_T^{(N-1)} \mathbf{1}_{\{T\}}(s).
$$
Since the initial condition for each sub-interval matches the terminal value of the preceding one (i.e., $X_{\rho_m}^{(m)} = X_{\rho_m}^{(m-1)}$ a.s.), the concatenated path $s \mapsto X_s^\pi$ preserves almost sure continuity. Given that the total number of stopping times is finite almost surely ($\mathbb{P}(N < \infty) = 1$), this inductive pasting procedure terminates in a finite number of steps with probability 1, strictly establishing the global existence and uniqueness of the continuous strong solution on $[0, T]$.

To establish the moment estimate, we integrate the global compact SDE \eqref{eq:dX_s_pi} over $[t, s]$. Taking the conditional expectation of the $k$-th power ($k \geq 2$) and applying the elementary inequality $|a+b+c|^k \leq 3^{k-1}(|a|^k + |b|^k + |c|^k)$, we separate the initial state, drift, and diffusion components. Subsequently, we bound the drift term via Hölder's inequality, and the diffusion term by sequentially applying the Burkholder-Davis-Gundy (BDG) inequality and Hölder's inequality with respect to the Lebesgue measure. Consolidating these standard estimates directly yields
\begin{align*}
&\mathbb{E}\left[\left|X_s^\pi\right|^k \mid X_t^\pi=x, \alpha_t=i\right]\\
\leq & 3^{k-1}|x|^k + 3^{k-1}\mathbb{E}\left[\left|\int_t^s b(\ell, X_\ell^\pi, \alpha_\ell, a_\ell^\pi) d\ell\right|^k\mid X_t^\pi=x, \alpha_t=i\right] \\
&+ 3^{k-1}\mathbb{E}\left[\left|\int_t^s \sigma(\ell, X_\ell^\pi, \alpha_\ell, a_\ell^\pi) dW_\ell\right|^k\mid X_t^\pi=x, \alpha_t=i\right]\\
\leq & \ 3^{k-1}|x|^k + 3^{k-1}(s-t)^{k-1}\mathbb{E}\left[\int_t^s \left|b(\ell, X_\ell^\pi, \alpha_\ell, a_\ell^\pi)\right|^k d\ell \mid X_t^\pi=x, \alpha_t=i\right] \\
& + 3^{k-1} C_{BDG} (s-t)^{\frac{k}{2}-1}\mathbb{E}\left[\int_t^s \left|\sigma(\ell, X_\ell^\pi, \alpha_\ell, a_\ell^\pi)\right|^k d\ell \mid X_t^\pi=x, \alpha_t=i\right],
\end{align*}
where $C_{BDG} > 0$ is a universal constant depending only on $k$. It follows from Assumption \ref{assumption}-(iii) that there exists a generic constant $\bar{C} > 0$ such that
\begin{align*}
1 + \mathbb{E}\left[\left|X_s^\pi\right|^k\mid X_t^\pi=x, \alpha_t=i\right] \leq \bar{C}\left(1 + |x|^k\right) + \bar{C} \int_t^s \left(1 + \mathbb{E}\left[\left|X_\ell^\pi\right|^k\mid X_t^\pi=x, \alpha_t=i\right]\right) d\ell.
\end{align*}
By applying Grönwall's inequality, we deduce that
\begin{equation*}
\mathbb{E}[|X_s^\pi|^k \mid X_t^\pi=x, \alpha_t=i] \leq C_k(1 + |x|^k) e^{C_k (s-t)}.
\end{equation*}

\item[(ii)] For any $x, y \in \mathbb{R}^n$, by adding and subtracting $\int_{\mathcal{A}} b(s, y, i, a) \pi(a \mid s, x, i) da$, we have
\begin{equation*}
\begin{aligned}
|\widetilde{b}(s, x, i, \pi) - \widetilde{b}(s, y, i, \pi)| \leq & \int_{\mathcal{A}} |b(s, x, i, a) - b(s, y, i, a)| \pi(a \mid s, x, i) da \\
& + \int_{\mathcal{A}} |b(s, y, i, a)| \cdot |\pi(a \mid s, x, i) - \pi(a \mid s, y, i)| da.
\end{aligned}
\end{equation*}
By Assumption \ref{assumption}-(ii) and (iii), and Definition \ref{def:admissible}-(iii), the second term on the right-hand side is bounded by $K(1+|y|) L_\pi \mu(\mathcal{A}) |x-y|$, where $L_\pi$ is the Lipschitz constant of $\pi$ and $\mu(\mathcal{A})$ is the finite Lebesgue measure of the compact set $\mathcal{A}$. Consequently, for any $x, y$ residing in a ball of radius $R$ (i.e., $|x| \vee |y| \leq R$), there exists a constant $C_R > 0$ depending on $R$ such that
$$
|\widetilde{b}(s, x, i, \pi) - \widetilde{b}(s, y, i, \pi)| \leq C_R|x-y|,
$$
which implies that $\widetilde{b}(s, x, i, \pi)$ is locally Lipschitz continuous. Meanwhile, it is straightforward to verify that $\widetilde{b}$ preserves the global linear growth condition, since $|\widetilde{b}(s, x, i, \pi)| \leq \int_{\mathcal{A}} K(1+|x|) \pi(a \mid s, x, i) da = K(1+|x|)$. Similarly, due to the strict ellipticity of $\Sigma$ and the analytic properties of the matrix square root, the relaxed diffusion coefficient $\widetilde{\sigma}(s, x, i, \pi)$ also satisfies both the local Lipschitz condition and the global linear growth condition. According to the standard theory of SDEs, the combination of local Lipschitz continuity and global linear growth of the coefficients is sufficient to prevent finite-time explosion, thereby guaranteeing the existence and uniqueness of the strong solution for the relaxed SDE \eqref{eq:state-Exp}.

We now verify the growth condition for the relaxed process $\tilde{X}_s^\pi$. For $n \geq 2$, taking the supremum over $u \in [t, s]$ for the integral form of \eqref{eq:state-Exp} and raising it to the power of $n$, we obtain
\begin{equation*}
\max_{t \leq u \leq s}|\tilde{X}_u^\pi|^n \leq 3^{n-1}|x|^n + 3^{n-1} \left(\int_t^s \left|\tilde{b}(\ell, \tilde{X}_\ell^\pi, \alpha_\ell; \pi_\ell)\right| d\ell\right)^n + 3^{n-1} \max_{t \leq u \leq s}\left|\int_t^u \tilde{\sigma}(\ell, \tilde{X}_\ell^\pi, \alpha_\ell; \pi_\ell) dW_\ell\right|^n.
\end{equation*}
Taking the conditional expectation, applying Hölder's inequality for the drift, and using the BDG inequality for the maximal functional of the continuous martingale term, we have
\begin{align*}
&\mathbb{E}\left[\max_{t \leq u \leq s}|\tilde{X}_u^\pi|^n \mid \tilde{X}_t^\pi=x, \alpha_t=i\right] \\
\leq & 3^{n-1}|x|^n + 3^{n-1}(s-t)^{n-1} \int_t^s \mathbb{E}\left[|\tilde{b}(\ell, \tilde{X}_\ell^\pi, \alpha_\ell; \pi_\ell)|^n\mid \tilde{X}_t^\pi=x, \alpha_t=i\right] d\ell \\
& + 3^{n-1} K_n (s-t)^{\frac{n}{2}-1} \int_t^s \mathbb{E}\left[|\tilde{\sigma}(\ell, \tilde{X}_\ell^\pi, \alpha_\ell; \pi_\ell)|^n\mid \tilde{X}_t^\pi=x, \alpha_t=i\right] d\ell.
\end{align*}
Combining with Assumption \ref{assumption}-(iii) yields
\begin{align*}
\mathbb{E}\left[\max_{t \leq u \leq s}|\tilde{X}_u^\pi|^n \mid \tilde{X}_t^\pi=x, \alpha_t=i\right] \leq C_1(1 + |x|^n) + C_2 \int_t^s \left( 1 + \mathbb{E}\left[\max_{t \leq v \leq \ell}|\tilde{X}_v^\pi|^n \mid \tilde{X}_t^\pi=x, \alpha_t=i\right] \right) d\ell,
\end{align*}
for some constants $C_1, C_2 > 0$. Grönwall's inequality implies
\begin{equation*}
\mathbb{E}\left[\max_{t \leq u \leq s}|\tilde{X}_u^\pi|^n\mid \tilde{X}_t^\pi=x, \alpha_t=i\right] \leq C_1\left(1 + |x|^n\right) e^{C_2 (s-t)}.
\end{equation*}
Evaluating at $s = T$, and absorbing the finite exponential term $e^{C_2 T}$ into a new constant $C = C(n, i)$, we obtain the desired bound:
\begin{equation*}
\mathbb{E}\left[\max_{t \leq s \leq T}|\tilde{X}_s^\pi|^n \mid \tilde{X}_t^\pi=x, \alpha_t=i\right] \leq C\left(1+|x|^n\right).
\end{equation*}
Finally, we establish the finiteness of the expected regularized payoff \eqref{eq:value-fun(MC)}. By Assumption \ref{assumption}, the reward functions $r$ and $h$ satisfy the polynomial growth condition. Since the action space $\mathcal{A}$ is compact, the continuous dependency of the rewards on $a$ is uniformly bounded, implying that their growth is strictly governed by the spatial state variable $x$. Given that the state process $\tilde{X}_s^\pi$ admits bounded moments of arbitrary order $n \geq 2$, the expectation of the running and terminal rewards is finite. Furthermore, the compactness of $\mathcal{A}$ ensures that the Tsallis entropy rate is uniformly bounded from above. For instance, when $p=1$ (Shannon entropy), the function $z \mapsto -z \ln z$ is bounded by $1/e$ for $z \geq 0$, leading to $\int_{\mathcal{A}} -\pi \ln \pi da \leq \mu(\mathcal{A})/e < \infty$, where $\mu(\mathcal{A})$ is the finite Lebesgue measure of $\mathcal{A}$. A similar uniform upper bound holds for $p>1$. Consequently, the expected entropy regularizer is well-defined and finite, ensuring that the overall expected payoff is finite for any admissible policy.
\end{itemize}
\end{proof}
\begin{theorem}
Under Assumptions \ref{assumption} and Definition \ref{def:admissible}, let the maximum step size of the time grid be $\Delta := \max_k (s_{k+1} - s_k)$. Denote the solution generated by the discrete-time sampling SDE \eqref{eq:dX_s_pi} as $X^{\Delta}$. As the step size $\Delta \to 0$, the sequence $X^{\Delta}$ converges weakly to the relaxed state process $\widetilde{X}^\pi$ in the path space $D([0,T]; \mathbb{R}^n)$ equipped with the Skorokhod topology.
\end{theorem}
\begin{proof}
By virtue of the BDG inequality, the Cauchy-Schwarz inequality, and Assumptions \ref{assumption}-(iii), we have, for any $0 \leq s < t \leq T$:
\begin{align*}
\mathbb{E}\left[ \sup_{s \leq u \leq t} |X_u^{\Delta} - X_s^{\Delta}|^4 \right] &\leq C \mathbb{E} \left[ \left( \int_s^t |b(\ell, X_\ell^{\Delta}, \alpha_\ell, a_\ell^\pi)| d\ell \right)^4 + \left( \int_s^t |\sigma(\ell, X_\ell^{\Delta}, \alpha_\ell, a_\ell^\pi)|^2 d\ell \right)^2 \right] \\
&\leq C' \mathbb{E} \left[ \int_s^t (1 + |X_\ell^{\Delta}|^4) d\ell \right] (t-s)^2.
\end{align*}
Recall from Lemma \ref{lem:growth-condition} that the moments $\mathbb{E}[\sup_{t \in [0,T]} |X_t^{\Delta}|^4]$ are uniformly bounded independent of the grid step size $\Delta$. Substituting this uniform bound into the right-hand side yields
\begin{equation*}
\mathbb{E}\left[ \sup_{s \leq u \leq t} |X_u^{\Delta} - X_s^{\Delta}|^4 \right] \leq \tilde{C} (t-s)^3 \leq \hat{C} (t-s)^2,
\end{equation*}
for some generic constants $\tilde{C}, \hat{C} > 0$. According to the Kolmogorov-Chentsov criterion, this implies that the sequence of probability distributions of $X^{\Delta}$ is tight in $D([0,T]; \mathbb{R}^n)$. By Prohorov's theorem, there exists at least one weakly convergent subsequence; let its limit process be denoted by $X^*$.

For any test function $f \in C_b^2(\mathbb{R}^n \times \mathcal{M})$, let
$\mathcal{G}_{s_k}:=\mathcal{F}_{s_k} \vee \sigma(\varsigma_0, \ldots, \varsigma_{k-1})$
denote the information available immediately before the exploration variable $\varsigma_k$ is sampled at time $s_k$. Since $\varsigma_k$ is independent of $\mathcal{G}_{s_k}$,
$\mathbb{P}(a_{s_k} \in d a \mid \mathcal{G}_{s_k})=\pi(d a \mid s_k, X_{s_k}^{\Delta}, \alpha_{s_k})$. Therefore, by the identity \eqref{eq:the-extended-infinitesimal-generator}, we have the exact conditional relation
$$
\mathbb{E}\left[\mathcal{L}^{a_{s_k}} f\left(s_k, X_{s_k}^{\Delta}, \alpha_{s_k}\right) \mid \mathcal{G}_{s_k}\right]=\tilde{\mathcal{L}}^\pi f\left(s_k, X_{s_k}^{\Delta}, \alpha_{s_k}\right) .
$$
This identity holds at the grid point. Over the interval $(s_k, s_{k+1}]$, the sampled action $a_{s_k}$ is held fixed, whereas the state and regime continue to evolve. Hence the difference between the controlled and relaxed compensators can be written as
$$
R_t^{\Delta}:=\sum_k \int_{s_k}^{s_{k+1} \wedge t}\left[\mathcal{L}^{a_{s_k}} f\left(s, X_s^{\Delta}, \alpha_{s-}\right)-\tilde{\mathcal{L}}^\pi f\left(s, X_s^{\Delta}, \alpha_{s-}\right)\right] d s.
$$
For each interval, insert the corresponding quantities evaluated at the left endpoint to obtain
\begin{align*}
R_t^{\Delta}= & \sum_k \int_{s_k}^{s_{k+1} \wedge t}\left[\mathcal{L}^{a_{s_k}} f\left(s, X_s^{\Delta}, \alpha_{s-}\right)-\mathcal{L}^{a_{s_k}} f\left(s_k, X_{s_k}^{\Delta}, \alpha_{s_k}\right)\right] d s \\
& +\sum_k\left(s_{k+1} \wedge t-s_k\right)\left[\mathcal{L}^{a_{s_k}} f\left(s_k, X_{s_k}^{\Delta}, \alpha_{s_k}\right)-\tilde{\mathcal{L}}^\pi f\left(s_k, X_{s_k}^{\Delta}, \alpha_{s_k}\right)\right] \\
+ & \sum_k \int_{s_k}^{s_{k+1} \wedge t}\left[\tilde{\mathcal{L}}^\pi f\left(s_k, X_{s_k}^{\Delta}, \alpha_{s_k}\right)-\tilde{\mathcal{L}}^\pi f\left(s, X_s^{\Delta}, \alpha_{s-}\right)\right] d s.
\end{align*}
The first and third terms are local freezing errors. Under the Lipschitz assumptions on the coefficients and the policy, together with the standard increment estimate
$$
\mathbb{E}\left[\left|X_s^{\Delta}-X_{s_k}^{\Delta}\right|\right] \leq C\left|s-s_k\right|^{1 / 2}, \quad s \in\left[s_k, s_{k+1}\right],
$$
their expected contribution over one interval is of order $O(\Delta^{3 / 2}) .$
The possible change of regime over one interval contributes only $O(\Delta^2)$ in expectation, since the probability of a regime transition during an interval of length $\Delta$ is $O(\Delta)$. Consequently, after summing over $O(\Delta^{-1})$ intervals, the accumulated freezing error is $O(\Delta^{1 / 2})$.

For the middle term, define
$$
Z_k^{\Delta}:=\mathcal{L}^{a_{s_k}} f\left(s_k, X_{s_k}^{\Delta}, \alpha_{s_k}\right)-\tilde{\mathcal{L}}^\pi f\left(s_k, X_{s_k}^{\Delta}, \alpha_{s_k}\right) .
$$
The conditional averaging identity above yields
$\mathbb{E}\left[Z_k^{\Delta} \mid \mathcal{G}_{s_k}\right]=0 .$
Hence the accumulated term
$\sum_k(s_{k+1}-s_k) Z_k^{\Delta}$
is a discrete martingale sum. Its second moment is of order
$\sum_k(s_{k+1}-s_k)^2=O(\Delta),$
and therefore
$\sum_k(s_{k+1}-s_k) Z_k^{\Delta}\rightarrow 0$
in $L^2$, and hence in probability, as $\Delta \rightarrow 0$. Combining the three estimates gives
$\sup _{0 \leq t \leq T}|R_t^{\Delta}| \rightarrow 0$
in probability.
On the other hand, applying Itô's formula with regime-switching to the finite-grid controlled process gives
\begin{equation*}
M_t^{\Delta}:= f(t, X_t^{\Delta}, \alpha_t)-f(0, X_0^{\Delta}, \alpha_0) -\int_0^t \mathcal{L}^{a_{\delta_{\Delta}(s)}} f(s, X_s^{\Delta}, \alpha_{s-}) d s,
\end{equation*}
where $\delta_{\Delta}(s)=s_k$ for $s \in(s_k, s_{k+1}]$. The process $M^{\Delta}$ is a martingale. Since
$$
\int_0^t \mathcal{L}^{a_{\delta_{\Delta}(s)}} f\left(s, X_s^{\Delta}, \alpha_{s-}\right) d s=\int_0^t \tilde{\mathcal{L}}^\pi f\left(s, X_s^{\Delta}, \alpha_{s-}\right) d s+R_t^{\Delta},
$$
the vanishing of $R^{\Delta}$ implies that every weak limit $(X^*, \alpha)$ of $(X^{\Delta}, \alpha)$ satisfies
\begin{equation*}
M_t^*:= f(t, X_t^*, \alpha_t)-f(0, X_0^*, \alpha_0)-\int_0^t \tilde{\mathcal{L}}^\pi f(s, X_s^*, \alpha_{s-}) d s
\end{equation*}
as a martingale. Therefore, $(X^*, \alpha)$ solves the martingale problem associated with the relaxed generator $\tilde{\mathcal{L}}^\pi$.

Under the Lipschitz and growth conditions established in Lemma \ref{lem:growth-condition}, the relaxed SDE admits a unique strong solution and hence a unique law. Accordingly, the associated martingale problem is well posed; see, e.g., \citet{Karatzas1991}, Chapter 5. Since the family $\{(X^{\Delta}, \alpha)\}_{\Delta>0}$ is tight and every weakly convergent subsequence has a limit solving the same martingale problem, all subsequential limits have the same law as the relaxed exploratory state process $(\widetilde{X}^\pi, \alpha)$. It follows that the entire sequence $(X^{\Delta}, \alpha)$ converges weakly to $(\widetilde{X}^\pi, \alpha)$ as $\Delta$ tends to zero.
Consequently, $X^{\Delta}$ converges weakly to $\widetilde{X}^\pi$ as $\Delta \to 0$.
Thus, in the high-frequency limit, the finite-grid sampled-and-held randomization converges to the relaxed regime-switching diffusion whose local characteristics are obtained by averaging the controlled generator under the policy $\pi$.
\end{proof}

\section{Policy Evaluation under regime-switching}
\label{sec:PE}
This section develops the policy-evaluation framework for stochastic policies under Markov regime-switching. We first derive the coupled system of policy-evaluation equations associated with the regime-dependent value functions. We then establish an equivalent martingale characterization that explicitly incorporates both the Brownian and regime-switching martingale components. Based on this representation, we examine the mean-square temporal-difference error and identify the additional jump contribution arising from the predictable quadratic variation. Finally, we formulate the martingale loss and martingale-orthogonality conditions, which provide the theoretical basis for the critic construction used in the subsequent actor-critic analysis.

\subsection{The martingale characterization of policy evaluation}

We first characterize the relaxed state-regime process through its martingale problem. This process-level characterization provides the probabilistic foundation for the subsequent policy-evaluation results and clarifies when alternative representations generate the same relaxed state-regime law.

\begin{proposition}
    Fix an admissible Markov policy $\pi$ and $(t,x,i)\in[0,T)\times\mathbb R^d\times\mathcal{M}$. Under $\mathbb{P}_{t,x,i}^{\pi}$, let $\mathbb G^{\widetilde{X}^\pi,\alpha}$ denote the augmented natural filtration of the relaxed state-regime process. Then, for every $f\in\mathscr D$,
\begin{equation*}
    M_s^{f,\pi} := f(s,\widetilde{X}_s^\pi,\alpha_s) -f(t,x,i) -\int_t^s \widetilde{\mathcal L}^{\pi} f(u,\widetilde{X}_u^\pi,\alpha_{u-})\,du, \qquad s\in[t,T],
\end{equation*}
is a $\mathbb G^{\widetilde{X}^\pi,\alpha}$-local martingale. Hence, $(\widetilde{X}^\pi,\alpha)$ solves the martingale problem associated with $(\widetilde{\mathcal L}^{\pi},\mathscr D)$.

Conversely, let $(Y,\eta)$ be a process starting from $(Y_t,\eta_t)=(x,i)$, and let $\mathbb G^{Y,\eta}$ be its augmented natural filtration. Suppose that
\begin{equation} \label{eq:Y_decomp}
    Y_s = x + \int_t^s \widetilde{b}(u,Y_u,\eta_{u-},\pi)du + M_s^Y,
\end{equation}
where $M^Y$ is a continuous  $\mathbb G^{Y,\eta}$-local martingale satisfies $d\langle M^Y\rangle_s = \tilde{c}(s,Y_s,\eta_{s-},\pi)ds$, and, for each $k\neq j$, the transition-counting process $N_{kj}^{\eta}(s)$ has predictable compensator $\int_t^s \mathbf{1}_{\{\eta_{u-}=k\}} q_{kj}du$.
Then $(Y,\eta)$ solves the same martingale problem as $(\widetilde{X}^\pi,\alpha)$. If the martingale problem associated with $(\widetilde{\mathcal L}^{\pi},\mathscr D)$ is well posed, then $(Y, \eta) \stackrel{d}{=}(\widetilde{X}^\pi, \alpha)$ on $C([t,T];\mathbb{R}^d) \times D([t,T];\mathcal{M})$.
\end{proposition}

\begin{proof}
The two assertions follow from the same characteristic calculation. Let $(Z, \zeta)$ denote either $(\widetilde{X}^\pi, \alpha)$ or $(Y, \eta)$, and write its continuous semimartingale part as
$$
d Z_s=\tilde{b}(s, Z_s, \zeta_{s-}, \pi) d s+d M_s^Z, \quad d\langle M^Z\rangle_s=\tilde{c}(s, Z_s, \zeta_{s-}, \pi) d s .
$$
For $k \neq j$, let
$\widetilde{N}_{k j}^\zeta(s)=N_{k j}^\zeta(s)-\int_t^s \mathbf{1}_{\left\{\zeta_{u-}=k\right\}} q_{k j} d u$
be the corresponding compensated transition-counting local martingale. Applying the generalized Itô formula to $f\left(s, Z_s, \zeta_s\right), f \in \mathscr{D}$, yields
\begin{align*}
f\left(s, Z_s, \zeta_s\right)-f(t, x, i)= & \int_t^s \tilde{\mathcal{L}}^\pi f\left(u, Z_u, \zeta_{u-}\right) d u +\int_t^s\left(\frac{\partial f}{\partial x}\left(u, Z_u, \zeta_{u-}\right)\right)^{\top} d M_u^Z \\
& +\sum_{k=1}^L \sum_{j \neq k} \int_t^s \mathbf{1}_{\left\{\zeta_{u-}=k\right\}}\left[f\left(u, Z_u, j\right)-f\left(u, Z_u, k\right)\right] d \widetilde{N}_{k j}^\zeta(u) .
\end{align*}
Hence,
\begin{align*}
M_s^{f, Z, \zeta}= \int_t^s\left(\frac{\partial f}{\partial x}(u, Z_u, \zeta_{u-})\right)^{\top} d M_u^Z +\sum_{k=1}^L \sum_{j \neq k} \int_t^s \mathbf{1}_{\{\zeta_{u-}=k\}}[f(u, Z_u, j)-f(u, Z_u, k)] d \widetilde{N}_{k j}^\zeta(u) .
\end{align*}
is a local martingale. Therefore both $( \widetilde{X}^\pi, \alpha)$ and $(Y, \eta)$ solve the martingale problem associated with $(\tilde{\mathcal{L}}^\pi, \mathscr{D})$. Well-posedness then yields equality in law.
\end{proof}

This transfer principle also yields an immediate equivalence at the level of expected rewards. Suppose that an action process $a=\{a_s\}_{s \in[t, T]}$ associated with $(Y, \eta)$ satisfies
$$
R_s^a:=r(s, Y_s, \eta_s, a_s)+\gamma p(s, Y_s, \eta_s, a_s, \pi(\cdot \mid s, Y_s, \eta_s)),
$$
and
$$
\mathbb{E}[R_s^a \mid \mathcal{F}_s^{Y, \eta}]=\mathcal{R}(s, Y_s, \eta_s, \pi)
$$
for $d s \otimes d \mathbb{P}$-almost every $(s, \omega)$. Then, by the tower property and the equality in law,
\begin{equation*}
\mathbb{E}\left[\int_t^T e^{-\beta(s-t)} R_s^a d s+e^{-\beta(T-t)} h(Y_T)\right]
=  \mathbb{E}\left[\int_t^T e^{-\beta(s-t)} \mathcal{R}(s, \widetilde{X}_s^\pi, \alpha_s, \pi) d s+e^{-\beta(T-t)} h(\widetilde{X}_T^\pi)\right] .
\end{equation*}
Thus, representations with the same state-regime law and the same conditional averaged reward yield the same expected discounted return. This exact transfer principle is distinct from the finite-grid sampled-and-held formulation, whose connection with the relaxed process is provided by the high-frequency weak-convergence result established previously.

We now turn from the state process to the policy value.
\begin{theorem}
Fix an admissible Markov policy $\pi$.
\begin{itemize}
    \item[(i)] Suppose that the policy-evaluation equation
\begin{equation}\label{eq:policy-evaluation}
\tilde{\mathcal{L}}^\pi v(t, x, i)+\mathcal{R}(t, x, i, \pi)-\beta v(t, x, i)=0, \quad(t, x, i) \in[0, T) \times \mathbb{R}^d \times \mathcal{M},
\end{equation}
with terminal condition $v(T, x, i)=h(x)$,
admits a unique viscosity solution $v \in C([0, T] \times \mathbb{R}^d \times \mathcal{M})$ satisfying a polynomial growth condition. Then
$v(t, x, i)=J(t, x, i ; \pi)$.
\item[(ii)] Let $v:[0, T] \times \mathbb{R}^d \times \mathcal{M} \rightarrow \mathbb{R}$ be continuous with polynomial growth, satisfy
$v(T, x, i)=h(x)$,
and assume that, for every initial condition $(t, x, i)$,
$$
\mathbb{E}_{t, x, i}^\pi\Big[\sup _{s \in[t, T]}|v(s, \widetilde{X}_s^\pi, \alpha_s)|+\int_t^T|\mathcal{R}(s, \widetilde{X}_s^\pi, \alpha_s, \pi)| d s\Big]<\infty.
$$
Define
$$
\mathcal{M}_s^{t, \pi, v}:=e^{-\beta(s-t)} v(s, \widetilde{X}_s^\pi, \alpha_s)+\int_t^s e^{-\beta(u-t)} \mathcal{R}(u, \widetilde{X}_u^\pi, \alpha_u, \pi) d u, \quad s \in[t, T].
$$
Then
$v(t, x, i)=J(t, x, i ; \pi)$
for every $(t, x, i)$ if and only if $\mathcal{M}^{t, \pi, v}$ is a martingale with respect to the augmented natural filtration of $(\widetilde{X}^\pi, \alpha)$ under $\mathbb{P}_{t, x, i}^\pi$.
\item[(iii)] If, in addition, $J(\cdot, \cdot, i ; \pi) \in C\left([0, T] \times \mathbb{R}^d\right) \cap C^{1,2}\left([0, T) \times \mathbb{R}^d\right), i \in \mathcal{M},$
then the process
$$
\mathcal{M}_s^{t, \pi}:=e^{-\beta(s-t)} J(s, \widetilde{X}_s^\pi, \alpha_s ; \pi)+\int_t^s e^{-\beta(u-t)} \mathcal{R}(u, \widetilde{X}_u^\pi, \alpha_u, \pi) d u
$$
admits
\begin{align*}
d \mathcal{M}_s^{t, \pi}= & e^{-\beta(s-t)} J_x(s, \widetilde{X}_s^\pi, \alpha_{s-} ; \pi)^{\top} \widetilde{\sigma}(s, \widetilde{X}_s^\pi, \alpha_{s-}, \pi) d \widetilde{W}_s \\
& +e^{-\beta(s-t)} \sum_{k=1}^L \sum_{j \neq k} \mathbf{1}_{\left\{\alpha_{s-}=k\right\}}\left[J(s, \widetilde{X}_s^\pi, j ; \pi)-J(s, \widetilde{X}_s^\pi, k ; \pi)\right] d \widetilde{N}_{k j}(s).
\end{align*}
Consequently, its predictable quadratic variation satisfies
\begin{align*}
d\left\langle\mathcal{M}^{t, \pi}\right\rangle_s=e^{-2 \beta(s-t)} & \bigg[J_x(s, \widetilde{X}_s^\pi, \alpha_{s-}; \pi)^{\top} \tilde{c}(s, \widetilde{X}_s^\pi, \alpha_{s-}, \pi) J_x(s, \widetilde{X}_s^\pi, \alpha_{s-} ; \pi) \\
& +\sum_{k=1}^L \mathbf{1}_{\{\alpha_{s-}=k\}} \sum_{j \neq k} q_{k j}(J(s, \widetilde{X}_s^\pi, j ; \pi)-J(s, \widetilde{X}_s^\pi, k ; \pi))^2\bigg] d s .
\end{align*}
\end{itemize}

\end{theorem}

\begin{proof}
We first prove part (i). By the DPP, for any stopping time $\tau \in [t, T]$,
\begin{equation}\label{eq:DDP}
J(t, x, i ; \pi)=\mathbb{E}_{t,x,i}^\pi\bigg[\int_t^\tau e^{-\beta(s-t)} \mathcal{R}(s, \widetilde{X}_s^\pi, \alpha_s,\pi) d s +e^{-\beta(\tau-t)} J(\tau, \widetilde{X}_\tau^\pi, \alpha_\tau ; \pi) \bigg].
\end{equation}
We begin by establishing that $J(t, x, i ; \pi)$ is a viscosity subsolution. Fix $i\in\mathcal M$, and let $\phi(\cdot, \cdot, i) \in C([0, T] \times \mathbb{R}^d) \cap C^{1,2}([0, T) \times \mathbb{R}^d)$ be a test function such that $J(\cdot, \cdot, i;\pi) - \phi(\cdot, \cdot, i)$ attains a local maximum at $(t_0, x_0)$ with $J(t_0, x_0, i ; \pi) = \phi(t_0, x_0, i)$. Hence,
$J(s, x, i ; \pi) \leq \phi(s, x, i)$
for $(s, x)$ sufficiently close to $(t_0, x_0)$.
Let $\tau_{\varepsilon} = (t_0+\varepsilon) \wedge \inf\{s>t_0 : \widetilde{X}_s^\pi \notin B_\delta(x_0)\}$ where $B_\delta(x_0)$ is a small neighborhood of $x_0$. Applying the DPP over $[t_0, \tau_{\varepsilon}]$, we obtain
\begin{align*}
0&=\mathbb{E}_{t_0, x_0, i}^\pi \bigg[\int_{t_0}^{\tau_{\varepsilon}} e^{-\beta(s-t_0)} \mathcal{R}(s, \widetilde{X}_s^\pi, \alpha_s, \pi) d s +e^{-\beta(\tau_{\varepsilon}-t_0)} J(\tau_{\varepsilon}, \widetilde{X}_{\tau_{\varepsilon}}^\pi, \alpha_{\tau_{\varepsilon}} ; \pi)-\phi(t_0, x_0, i)\bigg]\notag\\
&\leq \mathbb{E}_{t_0, x_0, i}^\pi\bigg[ \int_{t_0}^{\tau_{\varepsilon}} e^{-\beta(s-t_0)} \mathcal{R}(s, \widetilde{X}_s^\pi, \alpha_s, \pi) d s +e^{-\beta(\tau_{\varepsilon}-t_0)} \phi(\tau_{\varepsilon}, \widetilde{X}_{\tau_{\varepsilon}}^\pi, i) \mathbf{1}_{\{\alpha_{\tau_{\varepsilon}}=i\}} \\
&\quad +e^{-\beta(\tau_{\varepsilon}-t_0)} \sum_{j \neq i} J(\tau_{\varepsilon}, \widetilde{X}_{\tau_{\varepsilon}}^\pi, j ; \pi) \mathbf{1}_{\{\alpha_{\tau_{\varepsilon}}=j\}} -\phi(t_0, x_0, i)\bigg] .
\end{align*}
Combining the generalized It\^o expansion of $e^{-\beta(s-t_0)} \phi(s, \widetilde{X}_s^\pi, \alpha_s)$ over $[t_0, \tau_{\varepsilon}]$ with the transition intensities
$\mathbb{P}(\alpha_{s+d s}=j \mid \alpha_s=i)=q_{i j} d s+o(d s)$, for $j \neq i$,
yields
\begin{align*}
0 \leq &\mathbb{E}_{t_0, x_0, i}^\pi \int_{t_0}^{\tau_{\varepsilon}} e^{-\beta(s-t_0)} \bigg[ \frac{\partial \phi}{\partial s}(s, \widetilde{X}_s^\pi, i)+\tilde{b}(s, \widetilde{X}_s^\pi, i, \pi)^{\top} \frac{\partial \phi}{\partial x} (s, \widetilde{X}_s^\pi, i) +\frac{1}{2} \operatorname{tr}\Big[\tilde{c}(s, \widetilde{X}_s^\pi, i, \pi) \frac{\partial^2 \phi}{\partial x^2} (s, \widetilde{X}_s^\pi, i)\Big] \notag\\
&\quad +\sum_{j \neq i} q_{i j}\Big(J(s, \widetilde{X}_s^\pi, j ; \pi)-\phi(s, \widetilde{X}_s^\pi, i)\Big) + \mathcal{R}(s, \widetilde{X}_s^\pi, i, \pi) - \beta \phi(s, \widetilde{X}_s^\pi, i) \bigg] ds+o(\varepsilon).
\end{align*}
Dividing by $\varepsilon$ and letting $\varepsilon \downarrow 0$, we obtain
\begin{align*}
0 \leq & \frac{\partial \phi}{\partial t}(t_0, x_0, i)+\tilde{b}(t_0, x_0, i, \pi)^{\top} \frac{\partial \phi}{\partial x} (t_0, x_0, i) +\frac{1}{2} \operatorname{tr}\Big[\tilde{c}(t_0, x_0, i, \pi) \frac{\partial^2 \phi}{\partial x^2} (t_0, x_0, i)\Big] \notag\\
&\quad +\sum_{j \neq i} q_{i j}[J(t_0, x_0, j ; \pi)-\phi(t_0, x_0, i)] + \mathcal{R}(t_0, x_0, i, \pi) - \beta \phi(t_0, x_0, i) .
\end{align*}
Therefore, in the $i$-th equation, the regime-switching contribution at the contact point is
$$
\sum_{j \neq i} q_{i j}[J(t_0, x_0, j ; \pi)-\phi(t_0, x_0, i)]=\sum_{j \neq i} q_{i j}[J(t_0, x_0, j ; \pi)-J(t_0, x_0, i ; \pi)],
$$
and hence
$$
\widetilde{\mathcal{L}}^\pi \phi(t_0, x_0, i) + \mathcal{R}(t_0, x_0, i,\pi) - \beta \phi(t_0, x_0, i) \geq 0,
$$
where, in the $i$-th viscosity equation, the switching part is understood as
$
\sum_{j \neq i} q_{i j}[J(t_0, x_0, j ; \pi)-\phi(t_0, x_0, i)]$. This proves $J(\cdot, \cdot, i ; \pi)$ is the viscosity sub-solution.
A symmetric argument with reversed inequalities gives the corresponding super-solution inequality. Hence $J(t, x, i ; \pi)$ is a viscosity solution of
\begin{equation}\label{eq:PE-equ-of-J}
\tilde{\mathcal{L}}^\pi J(t, x, i ; \pi)+\mathcal{R}(t, x, i, \pi)-\beta J(t, x, i ; \pi)=0,
\end{equation}
with terminal condition $J(T, x, i ; \pi)=h(x)$.
Finally, the standing moment assumptions imply the polynomial growth estimate $|J(t, x, i ; \pi)| \leq C(1+|x|^\mu)$. Since the viscosity solution is assumed to be unique within this growth class, we conclude that $v(t, x, i)=J(t, x, i ; \pi)$.

For part (ii), suppose first that
$v(t, x, i)=J(t, x, i ; \pi)$.
By the Markov property and the definition of the value function,
\begin{equation*}
e^{-\beta(s-t)} J(s, \widetilde{X}_s^\pi, \alpha_s ; \pi)=\mathbb{E}_{t, x, i}^\pi \Big[\int_s^T e^{-\beta(u-t)} \mathcal{R}(u, \widetilde{X}_u^\pi, \alpha_u, \pi) d u +e^{-\beta(T-t)} h(\widetilde{X}_T^\pi) \mid \mathcal{G}_s\Big] .
\end{equation*}
Adding the accumulated reward $\int_t^s e^{-\beta(u-t)} \mathcal{R}(u, \widetilde{X}_u^\pi, \alpha_u, \pi) d u$ over $[t, s]$ to both sides gives
\begin{equation*}
\mathcal{M}_s^{t, \pi, J}=\mathbb{E}_{t, x, i}^\pi\bigg[\int_t^T e^{-\beta(u-t)} \mathcal{R}(u, \widetilde{X}_u^\pi, \alpha_u, \pi) d u +e^{-\beta(T-t)} h(\widetilde{X}_T^\pi) \mid \mathcal{G}_s\bigg].
\end{equation*}
Thus, $\mathcal{M}^{t, \pi, J}$ is a $\mathbb{G}$-martingale.
Conversely, suppose that $\mathcal{M}^{t, \pi, v}$ is a $\mathbb{G}$-martingale. Since its expectation is constant,
$v(t, x, i)=\mathbb{E}_{t, x, i}^\pi[\mathcal{M}_T^{t, \pi, v}]$.

Using $v(T, x, i)=h(x)$, we obtain
\begin{equation*}
v(t, x, i)=\mathbb{E}_{t, x, i}^\pi \bigg[\int_t e^{-\beta(s-t)} \mathcal{R}(s, \widetilde{X}_s^\pi, \alpha_s, \pi) d s  +e^{-\beta(T-t)}h(\widetilde{X}_T^\pi)\bigg]=J(t, x, i ; \pi).
\end{equation*}
This proves the martingale characterization.

For part (iii), apply the generalized Itô formula to $J(s, \widetilde{X}_s^\pi, \alpha_s ; \pi)$ yields
\begin{align*}
d J(s, \widetilde{X}_s^\pi, \alpha_s ; \pi)= & \tilde{\mathcal{L}}^\pi J(s, \widetilde{X}_s^\pi, \alpha_{s-} ; \pi) d s +J_x(s, \widetilde{X}_s^\pi, \alpha_{s-} ; \pi)^{\top} \widetilde{\sigma}(s, \widetilde{X}_s^\pi, \alpha_{s-}, \pi) d \widetilde{W}_s \\
& +\sum_{k=1}^L \sum_{j \neq k} \mathbf{1}_{\{\alpha_{s-}=k\}}[J(s, \widetilde{X}_s^\pi, j ; \pi)-J(s, \widetilde{X}_s^\pi, k ; \pi)] d \widetilde{N}_{k j}(s).
\end{align*}
Combining with \eqref{eq:PE-equ-of-J}
and $\alpha_s=\alpha_{s-}, d s \otimes d \mathbb{P}$-almost everywhere, the finite-variation terms cancel after applying the product rule to $e^{-\beta(s-t)} J(s, \widetilde{X}_s^\pi, \alpha_s ; \pi)$. Hence,
\begin{align*}
d \mathcal{M}_s^{t, \pi}= & e^{-\beta(s-t)} J_x(s, \widetilde{X}_s^\pi, \alpha_{s-} ; \pi)^{\top} \widetilde{\sigma}(s, \widetilde{X}_s^\pi, \alpha_{s-}, \pi) d \widetilde{W}_s \\
& +e^{-\beta(s-t)} \sum_{k=1}^L \sum_{j \neq k} \mathbf{1}_{\{\alpha_{s-}=k\}}[J(s, \widetilde{X}_s^\pi, j ; \pi)-J(s, \widetilde{X}_s^\pi, k ; \pi)] d \widetilde{N}_{k j}(s).
\end{align*}
The first term is a continuous martingale, whereas the second is a purely discontinuous martingale. Since $\widetilde{\sigma}(t, x, i, \pi) \widetilde{\sigma}(t, x, i, \pi)^{\top}=\tilde{c}(t, x, i, \pi)$,
the continuous contribution satisfies
$$
d\left\langle\mathcal{M}^{t, \pi, c}\right\rangle_s=e^{-2 \beta(s-t)} J_x(s, \widetilde{X}_s^\pi, \alpha_{s-} ; \pi)^{\top} \tilde{c}(s, \widetilde{X}_s^\pi, \alpha_{s-}, \pi) J_x(s, \widetilde{X}_s^\pi, \alpha_{s-} ; \pi) d s .
$$
For the jump component,
$d\langle\widetilde{N}_{k j}\rangle_s=\mathbf{1}_{\left\{\alpha_{s-}=k\right\}} q_{k j} d s$,
and distinct transition martingales are orthogonal. Therefore,
\begin{equation*}
d\left\langle\mathcal{M}^{t, \pi, d}\right\rangle_s=e^{-2 \beta(s-t)} \sum_{k=1}^L \mathbf{1}_{\{\alpha_{s-}=k\}} \sum_{j \neq k} q_{k j}[J(s, \widetilde{X}_s^\pi, j ; \pi)-J(s, \widetilde{X}_s^\pi, k ; \pi)]^2 d s .
\end{equation*}
The orthogonality of the continuous and purely discontinuous martingale components implies
$\langle\mathcal{M}^{t, \pi, c}, \mathcal{M}^{t, \pi, d}\rangle_s=0$,
and hence yields
\begin{align*}
d\left\langle\mathcal{M}^{t, \pi}\right\rangle_s=e^{-2 \beta(s-t)} & \big[J_x(s, \widetilde{X}_s^\pi, \alpha_{s-} ; \pi)^{\top} \tilde{c}(s, \widetilde{X}_s^\pi, \alpha_{s-}, \pi) J_x(s, \widetilde{X}_s^\pi, \alpha_{s-} ; \pi) \\
& +\sum_{k=1}^L \mathbf{1}_{\{\alpha_{s-}=k\}} \sum_{j \neq k} q_{k j}[J(s, \widetilde{X}_s^\pi, j ; \pi)-J(s, \widetilde{X}_s^\pi, k ; \pi)]^2\big] d s .
\end{align*}
Under the standing integrability assumptions, the two stochastic integrals are square integrable, and hence $\mathcal{M}^{t, \pi}$ is a square-integrable martingale. This establishes the PDE characterization, the martingale characterization, and the predictable quadratic-variation decomposition.
\end{proof}
The regime-switching contribution to the predictable quadratic variation is given by
$$
\sum_{k=1}^L \mathbf{1}_{\{\alpha_{s-}=k\}} \sum_{j \neq k} q_{k j}[J(s, \widetilde{X}_s^\pi, j ; \pi)-J(s, \widetilde{X}_s^\pi, k ; \pi)]^2.
$$
It disappears in the pure-diffusion case and is therefore the additional regime-switching contribution that must be retained in quadratic-variation-based temporal-difference criteria. This term provides the starting point for the MSTDE analysis below.

\subsection{Martingale loss}

For a fixed policy $\pi$, let $\mathbb{G}=\{\mathcal{G}_t\}_{t \in[0, T]}$ be the augmented natural filtration generated by $(\widetilde X^\pi,\alpha)$, and define the total discounted reward by
\begin{equation}\label{eq:definition-of-G-pi}
G^\pi:=\int_0^T e^{-\beta s} \mathcal{R}(s, \widetilde{X}_s^\pi, \alpha_s,\pi) d s+e^{-\beta T} h(\widetilde{X}_T^\pi).
\end{equation}
It follows from the Assumption \ref{assumption} and the Lemma \ref{lem:growth-condition} that $G^\pi \in L^2(\Omega, \mathcal{G}_T, \mathbb{P})$. For any $t \in[0, T]$, the Markov property of $(\widetilde{X}^\pi, \alpha)$ and the definition of the value function give
\begin{equation}\label{eq:value-conditional-expectation}
e^{-\beta t} J(t, \widetilde{X}_t^\pi, \alpha_t ; \pi) =\mathbb{E}\left[\int_t^T e^{-\beta s} \mathcal{R}(s, \widetilde{X}_s^\pi, \alpha_s,\pi) d s+e^{-\beta T} h(\widetilde{X}_T^\pi) \mid \mathcal{G}_t\right].
\end{equation}
Since the accumulated reward up to time $t$,
$\int_0^t e^{-\beta s} \mathcal{R}(s, \widetilde{X}_s^\pi, \alpha_s, \pi) d s$,
is $\mathcal{G}_t$-measurable, \eqref{eq:value-conditional-expectation} can be rewritten as
\begin{equation*}
e^{-\beta t} J(t, \widetilde{X}_t^\pi, \alpha_t ; \pi)+\int_0^t e^{-\beta s} \mathcal{R}(s, \widetilde{X}_s^\pi, \alpha_s, \pi) d s =\mathbb{E}[G^\pi \mid \mathcal{G}_t] .
\end{equation*}
Therefore, defining
\begin{equation}\label{eq:definition-of-M_t-pi}
M_t^\pi := e^{-\beta t} J(t, \widetilde{X}_t^\pi, \alpha_t ; \pi)+\int_0^t e^{-\beta s} \mathcal{R}(s, \widetilde{X}_s^\pi, \alpha_s,\pi) d s,
\end{equation}
we have
\begin{equation}\label{eq:martingale_M}
M_t^\pi=\mathbb{E}[G^\pi \mid \mathcal{G}_t].
\end{equation}
Because $G^\pi \in L^2$, the conditional-expectation process $\{M_t^\pi\}_{t \in[0, T]}$ is a square-integrable $\mathbb{G}$ - martingale. Moreover, at the terminal time,
$M_T^\pi=G^\pi$.

The preceding representation is obtained directly from conditional expectations and therefore does not require classical differentiability of $J(t, x, i ; \pi)$ with respect to the time-state variables. To approximate the value function, let $\Theta \subset \mathbb{R}^{d_\theta}$
be the critic parameter set, and consider a family of measurable functions
$\{J^\theta(t, x, i):[0, T] \times \mathbb{R}^d \times \mathcal{M} \rightarrow \mathbb{R}: \theta \in \Theta\}.$
For each $\theta \in \Theta$, assume that
$$
J^\theta(T, x, i)=h(x), \quad \mathbb{E} \int_0^T e^{-2 \beta t}|J^\theta(t, \widetilde{X}_t^\pi, \alpha_t)|^2 d t<\infty.
$$
Thus, $J^\theta(t, x, i)$ represents an admissible critic approximation of $J(t, x, i ; \pi)$. When gradient-based critic updates are considered later, we additionally require $J^\theta(t, x, i)$ to be differentiable with respect to $\theta$, together with the integrability condition
$\mathbb{E} \int_0^T e^{-2 \beta t}\|\nabla_\theta J^\theta(t, \widetilde{X}_t^\pi, \alpha_{t-})\|^2 d t<\infty$.
For each $\theta \in \Theta$, define
\begin{equation} \label{eq:M_theta_def}
    M_t^\theta := e^{-\beta t} J^\theta(t,\widetilde{X}_t^\pi,\alpha_t) + \int_0^t e^{-\beta s} \mathcal{R}(s,\widetilde{X}_s^\pi,\alpha_s,\pi)\,ds, \quad t\in[0, T].
\end{equation}
The terminal condition of the critic immediately gives
$$
M_T^\theta=e^{-\beta T} h(\widetilde{X}_T^\pi)+\int_0^T e^{-\beta s} \mathcal{R}(s, \widetilde{X}_s^\pi, \alpha_s,\pi) d s=G^\pi.
$$
Hence every admissible critic $J^\theta(t, x, i)$ generates a process $M^\theta$ with the same terminal value as the true value process:
$$
M_T^\theta=G^\pi, \quad \theta \in \Theta .
$$
The distinction is that, for a general approximation $J^\theta(t, x, i)$, the process $M^\theta$ need not be a martingale. When
$J^\theta(t, x, i)=J(t, x, i ; \pi)$,
the process $M^\theta$ coincides with $M^\pi$, which satisfies
$M_t^\pi=\mathbb{E}[G^\pi \mid \mathcal{G}_t]$,
and is therefore the square-integrable $\mathbb{G}$-martingale. This observation provides the basis for constructing martingale-based critic objectives.

Motivated by the representation \eqref{eq:martingale_M}, a natural idea is to approximate this conditional expectation by the parameterized process $M_t^\theta$. We therefore define the martingale loss by the integrated mean-square error
\begin{equation}\label{eq:ML(theta)-definition}
M L(\theta) :=\frac{1}{2} \mathbb{E} \int_0^T|G^\pi-M_t^\theta|^2 d t =\frac{1}{2} \mathbb{E} \int_0^T|M_T^\theta-M_t^\theta|^2 d t,
\end{equation}
where the second equality follows from $M_T^\theta=G^\pi$.
We now express the loss in terms of the critic and the observed rewards. Substituting the definitions of $G^\pi$ in \eqref{eq:definition-of-G-pi} and $M^\theta$ in \eqref{eq:M_theta_def}, we obtain
\begin{align}\label{eq:td_error}
G^\pi-M_t^\theta & =\int_t^T e^{-\beta s} \mathcal{R}(s, \widetilde{X}_s^\pi, \alpha_s,\pi) d s+e^{-\beta T} h(\widetilde{X}_T^\pi)-e^{-\beta t} J^\theta(t, \widetilde{X}_t^\pi, \alpha_t) \notag\\
& =e^{-\beta t}\bigg[\int_t^T e^{-\beta(s-t)} \mathcal{R}(s, \widetilde{X}_s^\pi, \alpha_s,\pi) d s+e^{-\beta(T-t)} h(\widetilde{X}_T^\pi)-J^\theta(t, \widetilde{X}_t^\pi, \alpha_t)\bigg].
\end{align}
Substituting \eqref{eq:td_error} into \eqref{eq:ML(theta)-definition} yields
\begin{equation}\label{eq:ML(theta)-weighted-least-squares criterion}
M L(\theta)=\frac{1}{2} \mathbb{E} \int_0^T e^{-2 \beta t} \bigg| J^\theta(t, \widetilde{X}_t^\pi, \alpha_t) -\int_t^T e^{-\beta(s-t)} \mathcal{R}(s, \widetilde{X}_s^\pi, \alpha_s,\pi) d s -e^{-\beta(T-t)} h(\widetilde{X}_T^\pi)\bigg|^2 d t.
\end{equation}
Hence, the martingale loss has a direct least-squares interpretation. At each time $t$, the critic
$J^\theta(t, \widetilde{X}_t^\pi, \alpha_t)$
is compared with the realized discounted return from $t$ to $T$,
$$
\int_t^T e^{-\beta(s-t)} \mathcal{R}(s, \widetilde{X}_s^\pi, \alpha_s, \pi) d s+e^{-\beta(T-t)} h(\widetilde{X}_T^\pi),
$$
and the squared error is weighted by $e^{-2 \beta t} $.
Therefore, minimizing $M L(\theta)$ amounts to fitting the critic to the future discounted return along the relaxed state-regime trajectories. The next result makes this interpretation precise by decomposing $M L(\theta)$ into a term independent of $\theta$ and a weighted value-function approximation error. Consequently, minimizing the martingale loss is equivalent to minimizing the corresponding weighted approximation error.

\begin{theorem}[Martingale-loss decomposition]
Let the preceding assumptions hold. Then, for any $\theta \in \Theta$,
\begin{equation} \label{eq:ML_decomp}
    ML(\theta) = C^\pi + \frac{1}{2} \mathbb{E} \int_0^T e^{-2\beta t} \big| J(t,\widetilde{X}_t^\pi,\alpha_t;\pi) - J^\theta(t,\widetilde{X}_t^\pi,\alpha_t) \big|^2dt,
\end{equation}
where
$$
C^\pi := \frac{1}{2} \mathbb{E} \int_0^T |G^\pi-M_t^\pi|^2dt=
\frac{1}{2} \int_0^T \mathbb{E}[\operatorname{Var}\left(G^\pi \mid \mathcal{G}_t\right)] d t
$$
is independent of $\theta$. Consequently,
\begin{equation*} %\label{eq:ML_argmin}
    \underset{\theta\in\Theta}{\arg\min}\,ML(\theta) = \underset{\theta\in\Theta}{\arg\min}\, \mathbb{E} \int_0^T e^{-2\beta t} \big| J(t,\widetilde{X}_t^\pi,\alpha_t;\pi) - J^\theta(t,\widetilde{X}_t^\pi,\alpha_t) \big|^2dt.
\end{equation*}
\end{theorem}

\begin{proof}
Noting $G^\pi-M_t^\theta=(G^\pi-M_t^\pi)+(M_t^\pi-M_t^\theta)$,
we obtain
\begin{equation*}
\mathbb{E}|G^\pi-M_t^\theta|^2= \mathbb{E}|G^\pi-M_t^\pi|^2  +\mathbb{E}|M_t^\pi-M_t^\theta|^2+2 \mathbb{E}[(G^\pi-M_t^\pi)(M_t^\pi-M_t^\theta)]  .
\end{equation*}
It follows from \eqref{eq:martingale_M} that
$\mathbb{E}[G^\pi-M_t^\pi \mid \mathcal{G}_t]=0 $. Since $M_t^\pi-M_t^\theta$ is $\mathcal{G}_t$-measurable, we have
\begin{equation*}
\mathbb{E}[(G^\pi-M_t^\pi)(M_t^\pi-M_t^\theta)] =\mathbb{E}[(M_t^\pi-M_t^\theta) \mathbb{E}[G^\pi-M_t^\pi \mid \mathcal{G}_t]]=0 .
\end{equation*}
Hence,
$$
\mathbb{E}|G^\pi-M_t^\theta|^2=\mathbb{E}|G^\pi-M_t^\pi|^2+\mathbb{E}|M_t^\pi-M_t^\theta|^2 .
$$
Integrating over $t \in[0, T]$ and multiplying by $1 / 2$ yields
\begin{equation}\label{eq:ML_theta_decomposition}
M L(\theta)= \frac{1}{2} \mathbb{E} \int_0^T\left|G^\pi-M_t^\pi\right|^2 d t  +\frac{1}{2} \mathbb{E} \int_0^T\left|M_t^\pi-M_t^\theta\right|^2 d t.
\end{equation}
For the first term of \eqref{eq:ML_theta_decomposition}, it follows from \eqref{eq:martingale_M} that
$\mathbb{E}[|G^\pi-M_t^\pi|^2 \mid \mathcal{G}_t]=\operatorname{Var}\left(G^\pi \mid \mathcal{G}_t\right),$
and therefore
$$
\frac{1}{2} \mathbb{E} \int_0^T\left|G^\pi-M_t^\pi\right|^2 d t=\frac{1}{2} \int_0^T \mathbb{E}\left[\operatorname{Var}\left(G^\pi \mid \mathcal{G}_t\right)\right] d t=C^\pi .
$$
This term depends only on the policy $\pi$ and is therefore independent of $\theta$.

As for the second term in \eqref{eq:ML_theta_decomposition}, applying \eqref{eq:definition-of-M_t-pi} and \eqref{eq:M_theta_def} gives
\begin{equation}\label{eq:M_t-pi-M_t-theta}
M_t^\pi-M_t^\theta=e^{-\beta t}\left[J(t, \widetilde{X}_t^\pi, \alpha_t ; \pi)-J^\theta(t, \widetilde{X}_t^\pi, \alpha_t)\right] .
\end{equation}
Therefore,
$$
|M_t^\pi-M_t^\theta|^2=e^{-2 \beta t}|J(t, \widetilde{X}_t^\pi, \alpha_t ; \pi)-J^\theta(t, \widetilde{X}_t^\pi, \alpha_t)|^2 .
$$
Substituting this identity into \eqref{eq:ML_theta_decomposition} gives
$$
M L(\theta)=C^\pi+\frac{1}{2} \mathbb{E} \int_0^T e^{-2 \beta t}|J(t, \widetilde{X}_t^\pi, \alpha_t ; \pi)-J^\theta(t, \widetilde{X}_t^\pi, \alpha_t)|^2 d t .
$$
Finally, $C^\pi$ is independent of $\theta$, and the multiplicative factor $1 / 2$ does not affect the minimizer. Hence,
$$
\underset{\theta \in \Theta}{\arg \min } M L(\theta)=\underset{\theta \in \Theta}{\arg \min }\, \mathbb{E} \int_0^T e^{-2 \beta t}|J(t, \widetilde{X}_t^\pi, \alpha_t ; \pi)-J^\theta(t, \widetilde{X}_t^\pi, \alpha_t)|^2 d t .
$$
\end{proof}

The preceding theorem shows that, for a fixed admissible policy $\pi$, the martingale loss $M L(\theta)$ differs from the discounted mean-square critic error $\mathbb{E}\int_0^T e^{-2\beta t} | J(t,\widetilde{X}_t^\pi,\alpha_t;\pi) - J^\theta(t,\widetilde{X}_t^\pi,\alpha_t) |^2dt$
only by the constant $C^\pi$, which is independent of $\theta$. Thus, regime-switching enters the approximation problem through both the regime-dependent value function and the distribution of $(t, \widetilde{X}_t^\pi, \alpha_t)$. The second term is therefore naturally interpreted as an approximation error under the time-state-regime occupation measure generated by the policy \(\pi\). By contrast, the constant $C^\pi$, depends on the conditional uncertainty of future returns and is not determined by the occupation measure alone.

When $\beta>0$, the factor $e^{-2 \beta t}$ places different weights on approximation errors at different times. Therefore, if the critic family cannot represent the value function, minimizing this weighted error need not yield the same minimizers as minimizing the unweighted error. To remove this deterministic time weight, define
\begin{equation*}
    ML_w(\theta) := \frac{1}{2} \mathbb{E}\int_0^T e^{2\beta t} |M_T^\theta-M_t^\theta|^2 dt,
\end{equation*}
which imply that $0 \leq M L_w(\theta) \leq e^{2 \beta T} M L(\theta)<\infty$.

Applying the same orthogonal decomposition as above gives
\begin{align*}
M L_w(\theta) & =\frac{1}{2} \mathbb{E} \int_0^T e^{2 \beta t}\left|G^\pi-M_t^\pi\right|^2 d t+\frac{1}{2} \mathbb{E} \int_0^T e^{2 \beta t}\left|M_t^\pi-M_t^\theta\right|^2 d t \\
& =C_w^\pi+\frac{1}{2} \mathbb{E} \int_0^T\left|J(t, \widetilde{X}_t^\pi, \alpha_t ; \pi)-J^\theta(t, \widetilde{X}_t^\pi, \alpha_t)\right|^2 d t,
\end{align*}
where $C_w^\pi := \frac{1}{2} \mathbb{E} \int_0^T e^{2 \beta t}|G^\pi-M_t^\pi|^2 d t$
is independent of $\theta$ and satisfies $0 \leq C_w^\pi \leq e^{2 \beta T} C^\pi<\infty .$
Consequently,
\begin{equation*}
    \arg\min_{\theta\in\Theta}ML_w(\theta) = \arg\min_{\theta\in\Theta} \mathbb{E}\int_0^T \left| J(t,\widetilde{X}_t^\pi,\alpha_t;\pi) - J^\theta(t,\widetilde{X}_t^\pi,\alpha_t) \right|^2dt.
\end{equation*}
In particular, when $\beta=0$, the two losses coincide: $M L_w(\theta)=M L(\theta)$.
Suppose now that the critic class is realizable along the policy-induced trajectories, in the sense that there exists $\theta^\pi \in \Theta$ such that
$$
J^{\theta^\pi}(t, \widetilde{X}_t^\pi, \alpha_t)=J(t, \widetilde{X}_t^\pi, \alpha_t ; \pi), \quad d t \otimes d \mathbb{P} \text {-a.e.}
$$
Then both the weighted and unweighted approximation errors can attain zero. Since $0<e^{-2 \beta T} \leq e^{-2 \beta t} \leq 1, t \in[0, T],$
the weighted and unweighted errors have the same zero set. Therefore,
\begin{align*}
\underset{\theta \in \Theta}{\arg \min } M L(\theta) & =\underset{\theta \in \Theta}{\arg \min } M L_w(\theta) \\
& =\left\{\theta \in \Theta: J^\theta(t, \widetilde{X}_t^\pi, \alpha_t)=J(t, \widetilde{X}_t^\pi, \alpha_t ; \pi), \quad d t \otimes d \mathbb{P} \text {-a.e. }\right\} .
\end{align*}
Thus, under realizability, both losses identify the same critic along the time-state-regime trajectories visited under $\pi$. This identification is naturally an almost-everywhere statement with respect to the corresponding occupation measure.

For sampled-trajectory implementation, define the realized discounted return from time $t$ by
\begin{equation*}
    \widehat G_t^\pi := \int_t^T e^{-\beta(s-t)} \mathcal R(s,\widetilde{X}_s^\pi,\alpha_s,\pi)\,ds + e^{-\beta(T-t)} h(\widetilde{X}_T^\pi).
\end{equation*}
By the definition of the value function and the Markov property, we have
$$
\mathbb{E}[\widehat{G}_t^\pi \mid \mathcal{G}_t]=J(t, \widetilde{X}_t^\pi, \alpha_t ; \pi).
$$
Moreover, combining this with \eqref{eq:td_error} yields
\begin{equation}\label{eq:G_pi-M_t-theta}
G^\pi-M_t^\theta=e^{-\beta t}[\widehat{G}_t^\pi-J^\theta(t, \widetilde{X}_t^\pi, \alpha_t)].
\end{equation}
Therefore, under the terminal condition $J^\theta(T, x, i)=h(x)$, so that $M_T^\theta=G^\pi$, the two loss functions can be written as
\begin{equation*}
M L(\theta) =\frac{1}{2} \mathbb{E} \int_0^T e^{-2 \beta t}\left|\widehat{G}_t^\pi-J^\theta(t, \widetilde{X}_t^\pi, \alpha_t)\right|^2 d t, \quad
M L_w(\theta) =\frac{1}{2} \mathbb{E} \int_0^T\left|\widehat{G}_t^\pi-J^\theta(t, \widetilde{X}_t^\pi, \alpha_t)\right|^2 d t.
\end{equation*}
Hence, both losses can be evaluated using the same sampled future return $\widehat{G}_t^\pi$, and differ only by the deterministic weight $e^{-2 \beta t}$.
If the terminal condition is not imposed exactly, the fixed-return identity \eqref{eq:G_pi-M_t-theta}
still holds. However, $M_T^\theta \neq G^\pi$
in general. More precisely,
\begin{equation*}
M_T^\theta-M_t^\theta= e^{-\beta t}\left[\widehat{G}_t^\pi-J^\theta(t, \widetilde{X}_t^\pi, \alpha_t)\right] +e^{-\beta T}\left[J^\theta(T, \widetilde{X}_T^\pi, \alpha_T)-h(\widetilde{X}_T^\pi)\right] .
\end{equation*}
Therefore, the regression losses with the fixed target $\widehat{G}_t^\pi$ coincides with the terminal-increment formulation precisely when the terminal discrepancy vanishes. Under the imposed terminal condition, the two formulations are therefore identical.

\subsection{The quadratic-variation limit of MSTDE}

We next study the mean-square TD error under a fixed admissible policy $\pi$. The key point is to separate the finite-variation component associated with the policy-evaluation residual from the martingale components generated by the diffusion and regime-switching. To this end, we first derive the semimartingale representation of $M^\theta$.

For each $\theta \in \Theta$ and $i \in \mathcal{M}$, assume that
$J^\theta(\cdot, \cdot, i) \in C^{1,2}([0, T) \times \mathbb{R}^d) \cap C([0, T] \times \mathbb{R}^d)$. Throughout this section, the policy $\pi$ is fixed. Hence, the dynamics of $(\widetilde{X}^\pi, \alpha)$ do not depend on $\theta$. Let $\tilde{r}^\pi$ denote the policy-averaged running reward appearing in the definition of $M^\theta$.

For $i \in \mathcal{M}$, define the policy-evaluation residual associated with the critic $J^\theta(t, x, i)$ by
\begin{align} \label{eq:F_theta_def}
    F^\theta(t,x,i) := & J_t^\theta(t, x, i)+\tilde{b}(t, x, i, \pi)^{\top} J_x^\theta(t, x, i) +\frac{1}{2} \operatorname{Tr}\left[\tilde{c}(t, x, i, \pi) J_{x x}^\theta(t, x, i)\right] \nonumber\\
    & +\sum_{j=1}^L q_{i j} J^\theta(t, x, j)+\tilde{r}^\pi(t, x, i)-\beta J^\theta(t, x, i) .
\end{align}
Applying the generalized Itô's formula to $M_t^\theta$, gives
\begin{align}\label{eq:dM_theta_Ito}
    dM_t^\theta =\;& e^{-\beta t} F^\theta(t,\widetilde{X}_t^\pi,\alpha_{t-})dt + e^{-\beta t} J_x^\theta(t,\widetilde{X}_t^\pi,\alpha_{t-})^\top \widetilde{\sigma} (t,\widetilde{X}_t^\pi,\alpha_{t-},\pi) \,d\widetilde W_t \nonumber \\
    &\quad+ e^{-\beta t} \sum_{i=1}^{L} \sum_{j\neq i} \mathbf{1}_{\{\alpha_{t-}=i\}} \Big[ J^\theta(t,\widetilde{X}_t^\pi,j) - J^\theta(t,\widetilde{X}_t^\pi,i) \Big] d\widetilde{N}_{ij}(t).
\end{align}
Thus $M^\theta$ can be decomposed as
$$
M_t^\theta = M_0^\theta + \mathcal{D}_t^\theta + \mathcal{N}_t^\theta,
$$
where the finite-variation part is
$$
\mathcal{D}_t^\theta = \int_0^t e^{-\beta s} F^\theta(s, \widetilde{X}_s^\pi, \alpha_{s-}) d s,
$$
and $\mathcal{N}^\theta$ collects the Brownian and compensated regime-switching martingale terms:
\begin{align}\label{eq:mathcal{N}-theta}
    \mathcal{N}^\theta =\;& e^{-\beta t} J_x^\theta(t,\widetilde{X}_t^\pi,\alpha_{t-})^\top \widetilde{\sigma} (t,\widetilde{X}_t^\pi,\alpha_{t-},\pi) \,d\widetilde W_t \nonumber \\
    &\quad+ e^{-\beta t} \sum_{i=1}^{L} \sum_{j\neq i} \mathbf{1}_{\{\alpha_{t-}=i\}} \Big[ J^\theta(t,\widetilde{X}_t^\pi,j) - J^\theta(t,\widetilde{X}_t^\pi,i) \Big] d\widetilde{N}_{ij}(t).
\end{align}
To describe its predictable quadratic variation, define
\begin{equation*} \label{eq:Gamma_theta_def}
    \Gamma^\theta(t,x,i) :=  J_x^\theta(t,x,i)^\top \tilde{c}(t,x,i,\pi) J^\theta_x(t,x,i) + \sum_{j=1}^L q_{ij} \left[ J^\theta(t,x,j)-J^\theta(t,x,i) \right]^2.
\end{equation*}
The first term represents the contribution of the Brownian component. Recall \eqref{eq:the-relaxed-drift-covariance}, we have
\begin{equation*} %\label{eq:diffusion_term_integral}
     J_x^\theta(t,x,i)^\top \tilde{c}(t,x,i,\pi) J_x^\theta(t,x,i) = \int_{\mathcal{A}} \left| \sigma(t,x,i,a)^\top J_x^\theta(t,x,i) \right|^2 \pi(da\mid t,x,i).
\end{equation*}
The Brownian martingale is continuous, while the compensated regime-switching martingale is purely discontinuous, so their predictable covariation is zero. Moreover, different transition-counting processes have no simultaneous jumps. Since
$d \widetilde{N}_{i j}(t)=d N_{i j}(t)-\mathbf{1}_{\{\alpha_{t-}=i\}} q_{i j} d t, i \neq j$,
the predictable quadratic variation of the martingale component $\mathcal{N}^\theta$ is therefore
$$
d\left\langle\mathcal{N}^\theta\right\rangle_t=e^{-2 \beta t} \Gamma^\theta(t, \widetilde{X}_t^\pi, \alpha_{t-}) d t .
$$
Equivalently,
$$
\left\langle\mathcal{N}^\theta\right\rangle_T=\int_0^T e^{-2 \beta t} \Gamma^\theta(t, \widetilde{X}_t^\pi, \alpha_{t-}) d t .
$$
For later use, define
\begin{equation*}
  D(\theta) := \mathbb{E} \int_0^T e^{-2\beta t} |F^\theta(t,\widetilde{X}_t^\pi,\alpha_{t-})|^2dt, \quad \mathcal V_\beta(\theta) := \mathbb{E} \int_0^T e^{-2\beta t} \Gamma^\theta(t,\widetilde{X}_t^\pi,\alpha_{t-})dt.
\end{equation*}
Whenever
$D(\theta)<\infty$, $\mathcal{V}_\beta(\theta)<\infty$, the drift term is square-integrable and $\mathcal{N}^\theta$ is a square-integrable martingale. In particular,
$$
\mathbb{E}|\mathcal{N}_T^\theta|^2=\mathbb{E}\langle\mathcal{N}^\theta\rangle_T=\mathcal{V}_\beta(\theta) .
$$
Consequently, the integrated decomposition can be written on $[0, T]$ as
\begin{equation}\label{eq:the-semimartingale-decomposition}
M_t^\theta=M_0^\theta+\int_0^t e^{-\beta s} F^\theta(s, \widetilde{X}_s^\pi, \alpha_{s-}) d s+\mathcal{N}_t^\theta, \quad 0 \leq t \leq T.
\end{equation}

Now, let $\mathcal{P} = \{0=t_0<t_1<\cdots<t_N=T\}$ be a deterministic partition of $[0, T]$, with mesh size $|\mathcal{P}| := \max_{0\leq k\leq N-1} (t_{k+1}-t_k)$.
For each $k=0, \ldots, N-1$, define the exact discounted TD increment by
\begin{align}
    \delta_k^\theta := &M_{t_{k+1}}^\theta-M_{t_k}^\theta \nonumber \\
    = &e^{-\beta t_{k+1}} J^\theta( t_{k+1}, \widetilde{X}_{t_{k+1}}^\pi, \alpha_{t_{k+1}} ) - e^{-\beta t_k} J^\theta( t_k, \widetilde{X}_{t_k}^\pi, \alpha_{t_k} )+ \int_{t_k}^{t_{k+1}} e^{-\beta s} \tilde r^\pi(s,\widetilde{X}_s^\pi,\alpha_s)\,ds. \label{eq:delta_k_theta_def}
\end{align}
The corresponding mean-square TD error is
$$
\operatorname{MSTDE}_{\mathcal{P}}(\theta) := \mathbb{E} \sum_{k=0}^{N-1} |\delta_k^\theta|^2.
$$
The following theorem characterizes the limit of $\operatorname{MSTDE}_{\mathcal{P}}(\theta)$ as $|\mathcal{P}| \rightarrow 0$ and provides an explicit error estimate in terms of the mesh size. %In particular, it shows that the Bellman-residual contribution vanishes in the limit, while the diffusion and regime-switching contributions remain through the expected quadratic variation of $\mathcal{N}^\theta$, as specified by $\Gamma^\theta$.

\begin{theorem}[MSTDE limit and mesh-error bound]
Let the preceding assumptions hold. Then, for any $\theta \in \Theta$ and any deterministic partition $\mathcal{P}$ of $[0, T]$, we have
\begin{equation} \label{eq:MSTDE_bound}
    \left| \operatorname{MSTDE}_{\mathcal{P}}(\theta) - \mathcal V_\beta(\theta) \right| \leq |\mathcal{P}|D(\theta) + 2 \sqrt{ |\mathcal{P}|D(\theta)\mathcal V_\beta(\theta) }.
\end{equation}
Consequently, for every fixed $\theta$,
\begin{equation} \label{eq:MSTDE_limit}
    \lim _{|\mathcal{P}| \rightarrow 0}\operatorname{MSTDE}_{\mathcal{P}}(\theta) = \mathcal V_\beta(\theta) = \mathbb{E} \langle\mathcal{N}^\theta\rangle_T .
\end{equation}
\end{theorem}

\begin{proof}
For any given $\theta \in \Theta$, the semimartingale representation of $M^\theta$ gives
$$
\delta_k^\theta=\mathcal{D}_{t_{k+1}}^\theta-\mathcal{D}_{t_k}^\theta+\mathcal{N}_{t_{k+1}}^\theta-\mathcal{N}_{t_k}^\theta, \quad k=0, \ldots, N-1 .
$$
By the definition of $\mathcal{D}^\theta$ and the Cauchy-Schwarz inequality, we obtain
\begin{align}\label{eq:D_tk}
\mathbb{E} \sum_{k=0}^{N-1}|\mathcal{D}_{t_{k+1}}^\theta-\mathcal{D}_{t_k}^\theta|^2 & =\mathbb{E} \sum_{k=0}^{N-1}\left|\int_{t_k}^{t_{k+1}} e^{-\beta s} F^\theta(s, \widetilde{X}_s^\pi, \alpha_{s-}) d s\right|^2 \notag\\
& \leq \sum_{k=0}^{N-1}\left(t_{k+1}-t_k\right) \mathbb{E} \int_{t_k}^{t_{k+1}} e^{-2 \beta s}|F^\theta(s, \widetilde{X}_s^\pi, \alpha_{s-})|^2 d s \notag\\
& \leq|\mathcal{P}| D(\theta).
\end{align}
On the other hand, since $\mathcal{N}^\theta$ is a square-integrable martingale starting from zero, the martingale isometry yields
\begin{equation*}
\mathbb{E} \sum_{k=0}^{N-1}|\mathcal{N}_{t_{k+1}}^\theta-\mathcal{N}_{t_k}^\theta|^2 =\mathbb{E} \sum_{k=0}^{N-1}(\langle\mathcal{N}^\theta\rangle_{t_{k+1}}-\langle\mathcal{N}^\theta\rangle_{t_k}) =\mathbb{E}\left\langle\mathcal{N}^\theta\right\rangle_T=\mathcal{V}_\beta(\theta).
\end{equation*}
Applying the Cauchy-Schwarz inequality once more, we have
\begin{align*}
 \Big|\mathbb{E} \sum_{k=0}^{N-1}(\mathcal{D}_{t_{k+1}}^\theta-\mathcal{D}_{t_k}^\theta)(\mathcal{N}_{t_{k+1}}^\theta-\mathcal{N}_{t_k}^\theta)\Big|  &\leq\Big(\mathbb{E} \sum_{k=0}^{N-1}|\mathcal{D}_{t_{k+1}}^\theta-\mathcal{D}_{t_k}^\theta|^2\Big)^{1/2}\Big(\mathbb{E} \sum_{k=0}^{N-1}|\mathcal{N}_{t_{k+1}}^\theta-\mathcal{N}_{t_k}^\theta|^2\Big)^{1 / 2} \\
& \leq \sqrt{|\mathcal{P}| D(\theta) \mathcal{V}_\beta(\theta)} .
\end{align*}
Thus, expanding the squared TD increments and using the preceding estimates, we obtain
\begin{align*}
 \left|\operatorname{MSTDE}_{\mathcal{P}}(\theta)-\mathcal{V}_\beta(\theta)\right|  &\leq \mathbb{E} \sum_{k=0}^{N-1}|\mathcal{D}_{t_{k+1}}^\theta-\mathcal{D}_{t_k}^\theta|^2 +2\Big|\mathbb{E} \sum_{k=0}^{N-1}(\mathcal{D}_{t_{k+1}}^\theta-\mathcal{D}_{t_k}^\theta)(\mathcal{N}_{t_{k+1}}^\theta-\mathcal{N}_{t_k}^\theta)\Big| \\
&\leq|\mathcal{P}| D(\theta)+2 \sqrt{|\mathcal{P}| D(\theta) \mathcal{V}_\beta(\theta)}
\end{align*}
Since $D(\theta)$ and $\mathcal{V}_\beta(\theta)$ are finite, letting $|\mathcal{P}| \rightarrow 0$ the desired convergence.

\end{proof}

The theorem shows that, as the observation mesh is refined, the MSTDE converges to the expected quadratic variation of the martingale component $\mathcal{N}^\theta$. The Bellman residual affects the finite-mesh error through $D(\theta)$, but it does not appear as a separate leading-order term in the limiting objective. Hence, in general,
$\mathcal{V}_\beta(\theta) \neq D(\theta),$
and it also need not coincide with the unweighted squared Bellman residual
$$
\mathbb{E} \int_0^T\left|F^\theta(t, \widetilde{X}_t^\pi, \alpha_{t-})\right|^2 d t .
$$
This distinction is particularly clear when
$\mathcal{V}_\beta(\theta)=0 .$
In this case, the error bound reduces to
$$
0 \leq \operatorname{MSTDE}_{\mathcal{P}}(\theta) \leq|\mathcal{P}| D(\theta),
$$
and therefore
$\operatorname{MSTDE}_{\mathcal{P}}(\theta) \rightarrow 0$
as $|\mathcal{P}| \rightarrow 0$, even if
$D(\theta)>0 .$
Thus, vanishing high-frequency MSTDE does not imply that the Bellman residual itself vanishes.

The convergence can also be made uniform over a class of critic parameters. Let $\Theta_0 \subseteq \Theta$ be nonempty and suppose that
$\sup _{\theta \in \Theta_0} D(\theta)<\infty, \sup _{\theta \in \Theta_0} \mathcal{V}_\beta(\theta)<\infty .$
Then
\begin{equation*}
\sup_{\theta \in \Theta_0}\left|\operatorname{MSTDE}_{\mathcal{P}}(\theta)-\mathcal{V}_\beta(\theta)\right|  \leq|\mathcal{P}| \sup _{\theta \in \Theta_0} D(\theta)+2 \sqrt{|\mathcal{P}|(\sup_{\theta \in \Theta_0} D(\theta))(\sup_{\theta \in \Theta_0} \mathcal{V}_\beta(\theta))} \rightarrow 0 .
\end{equation*}
Therefore, the MSTDE converges uniformly to $\mathcal{V}_\beta$ on $\Theta_0$. This statement concerns convergence of the objective functions themselves; additional conditions are needed to infer convergence of their minimizers or of a stochastic-approximation algorithm.

To understand the drift and martingale parts over a single observation interval, for a fixed $\theta$, assume that the second moments of the drift density and the martingale-variation density are uniformly bounded in time, namely,
\begin{equation*}
 \underset{0 \leq t<T}{\operatorname{ess} \sup } \,\mathbb{E}\left[e^{-2 \beta t}|F^\theta(t, \widetilde{X}_t^\pi, \alpha_{t-})|^2\right]<\infty,  \quad \underset{0 \leq t<T}{\operatorname{ess} \sup } \,\mathbb{E}\left[e^{-2 \beta t} \Gamma^\theta(t, \widetilde{X}_t^\pi, \alpha_{t-})\right]<\infty .
\end{equation*}
Applying \eqref{eq:D_tk} on the interval $[t_k, t_{k+1}]$ gives
$$
\mathbb{E}|\mathcal{D}_{t_{k+1}}^\theta-\mathcal{D}_{t_k}^\theta|^2 =O(h_k^2),
$$
whereas
$$
\mathbb{E}|\mathcal{N}_{t_{k+1}}^\theta-\mathcal{N}_{t_k}^\theta|^2  =\mathbb{E} \int_{t_k}^{t_{k+1}} e^{-2 \beta s} \Gamma^\theta(s, \widetilde{X}_s^\pi, \alpha_{s-}) d s=O(h_k).
$$
The interaction between the two components satisfies
$$
|\mathbb{E}[(\mathcal{D}_{t_{k+1}}^\theta-\mathcal{D}_{t_k}^\theta)(\mathcal{N}_{t_{k+1}}^\theta-\mathcal{N}_{t_k}^\theta)]|  =O(h_k^{3 / 2}) .
$$
Therefore, the one-step mean-square TD error satisfies
$$
\mathbb{E}\left|\delta_k^\theta\right|^2=\mathbb{E} \int_{t_k}^{t_{k+1}} e^{-2 \beta s} \Gamma^\theta(s, \widetilde{X}_s^\pi, \alpha_{s-}) d s+O(h_k^{3 / 2})+O(h_k^2).
$$
The orders above show the basic difference between the two components:
$$
\text { drift contribution }=O\left(h_k^2\right), \quad \text { martingale contribution }=O\left(h_k\right) .
$$
Hence the martingale variation is typically the dominant term in the one-step mean-square TD error.

More precisely, suppose that the average martingale-variation density over the intervals under consideration is bounded away from zero, that is,
$$
\frac{1}{h_k} \mathbb{E} \int_{t_k}^{t_{k+1}} e^{-2 \beta s} \Gamma^\theta(s, \widetilde{X}_s^\pi, \alpha_{s-}) d s\geq c_\theta>0 .
$$
Then
$$
\frac{\mathbb{E}|\mathcal{D}_{t_{k+1}}^\theta-\mathcal{D}_{t_k}^\theta|^2}{\mathbb{E}|\mathcal{N}_{t_{k+1}}^\theta-\mathcal{N}_{t_k}^\theta|^2}=O(h_k) \rightarrow 0.
$$
Thus, whenever the martingale variation is locally nondegenerate, it determines the leading-order behavior of the one-step mean-square TD error. Consequently, refining the observation mesh does not generally transform the unnormalized MSTDE into a squared Bellman-residual criterion. Instead, the high-frequency limit is governed by the quadratic variation of the martingale component.

To characterize the regime-switching contribution to $\mathcal{V}_\beta(\theta)$, for any Borel set $A \subseteq[0, T] \times \mathbb{R}^d$, define the regime-specific occupation measure by
$$
\mu^\pi(A, i) := \mathbb{E} \int_0^T \mathbf{1}_{\{(t, \widetilde{X}_t^\pi) \in A, \alpha_t=i\}} d t, \quad i\in\mathcal{M},
$$
and let $\bar{\mu}^\pi(A) := \sum_{i=1}^L \mu^\pi(A, i)$
be the total occupation measure. Since $\alpha_t=\alpha_{t-}, d t \otimes d \mathbb{P}$-almost every $(t, \omega)$, the jump contribution to $\mathcal{V}_\beta(\theta)$ can be written as
$$
\mathcal{V}_{\beta, Q}(\theta) = \sum_{i=1}^L \sum_{j \neq i} q_{i j} \int_{[0, T] \times \mathbb{R}^d} e^{-2 \beta t}\left[J^\theta(t, x, j)-J^\theta(t, x, i)\right]^2 \mu^\pi(d t, d x, i).
$$
Thus, $\mathcal{V}_{\beta, Q}(\theta)$ measures the regime-to-regime discrepancy of the critic along the state-regime pairs visited by the process, weighted by the transition intensities $q_{i j}$.

Since $\mu^\pi(\cdot, i) \ll \bar{\mu}^\pi$, we write
$$
\rho^\pi(t, x, i) := \frac{d \mu^\pi(\cdot, i)}{d \bar{\mu}^\pi}(t, x), \quad i \in \mathcal{M}.
$$
Then $0 \leq \rho^\pi(t, x, i) \leq 1$ and $\sum_{i=1}^L \rho^\pi(t, x, i)=1$, for $\bar{\mu}^\pi$-a.e.

The following proposition characterizes when the jump contribution is zero.
\begin{proposition}[Zero jump energy]\label{pro:zero-jump-energy}
For any fixed $\theta \in \Theta$, $\mathcal{V}_{\beta, Q}(\theta)=0$ holds if and only if, for every pair $(i, j)$ such that $q_{i j}>0$,
$$
J^\theta(t, x, j)=J^\theta(t, x, i), \quad \mu^\pi(\cdot, i) \text {-a.e. }
$$
Suppose further that $Q$ irreducible and that every regime is visited with positive occupation density in the sense that
$$
\rho^\pi(t, x, i)>0, \quad \bar{\mu}^\pi \text {-a.e. }, \quad i \in \mathcal{M} .
$$
Then $\mathcal{V}_{\beta, Q}(\theta)=0$ implies
$$
J^\theta(t, x, 1)=\cdots=J^\theta(t, x, L), \quad \bar{\mu}^\pi \text {-a.e. }
$$
\end{proposition}

\begin{proof}
From the representation of $\mathcal{V}_{\beta, Q}(\theta)$, every term is nonnegative. Hence,
$\mathcal{V}_{\beta, Q}(\theta)=0$
holds if and only if, for every pair $(i, j)$ with $q_{i j}>0$,
$$
\int_{[0, T] \times \mathbb{R}^d} e^{-2 \beta t}\left[J^\theta(t, x, j)-J^\theta(t, x, i)\right]^2 \mu^\pi(d t, d x, i)=0.
$$
By $e^{-2 \beta t}>0,$
the preceding integral is zero if and only if
$J^\theta(t, x, j)=J^\theta(t, x, i)$
for $\mu^\pi(\cdot, i)$-almost every $(t, x)$. This proves the first assertion.

Now assume that
$$
\rho^\pi(t, x, i)>0, \quad \bar{\mu}^\pi \text {-a.e. }
$$
for every $i \in \mathcal{M}$. Since
$d \mu^\pi(\cdot, i)=\rho^\pi(t, x, i) d \bar{\mu}^\pi(t, x)$,
the measures $\mu^\pi(\cdot, i)$ and $\bar{\mu}^\pi$ are equivalent. Therefore, for every transition $i \rightarrow j$ with $q_{i j}>0$,
$$
J^\theta(t, x, j)=J^\theta(t, x, i), \quad \bar{\mu}^\pi \text {-a.e. }
$$
Because $Q$ is irreducible, any two regimes can be connected by a finite sequence
$i=i_0 \rightarrow i_1 \rightarrow \cdots \rightarrow i_m=j$
such that
$q_{i_{\ell} i_{\ell+1}}>0, \ell=0, \ldots, m-1 .$
Applying the preceding equality along this path gives
$$
J^\theta(t, x, i_0)=J^\theta(t, x, i_1)=\cdots=J^\theta(t, x, i_m), \quad \bar{\mu}^\pi \text {-a.e. }
$$
Hence,
$$
J^\theta(t, x, 1)=\cdots=J^\theta(t, x, L), \quad \bar{\mu}^\pi \text {-a.e. }
$$
which proves the second assertion.
\end{proof}

For a further interpretation of the jump contribution, let $\nu=\left(\nu_1, \ldots, \nu_L\right)^{\top}$ the Markov chain, satisfying $\nu^{\top} Q=0$. For fixed $(t, x)$, define the regime-switching quadratic energy
$$
\mathcal{E}_\nu\left(J^\theta ; t, x\right) :=-\sum_{i=1}^L \nu_i J^\theta(t, x, i) \sum_{j=1}^L q_{i j} J^\theta(t, x, j) .
$$
Noting that the stationarity relation $\sum_{i=1}^L \nu_i q_{i j}=0$, we expand the squared regime differences as
\begin{align*}
& \frac{1}{2} \sum_{i, j=1}^L \nu_i q_{i j}[J^\theta(t, x, j)-J^\theta(t, x, i)]^2 \\
& =\frac{1}{2} \sum_{j=1}^L J^\theta(t, x, j)^2 \sum_{i=1}^L \nu_i q_{i j}+\frac{1}{2} \sum_{i=1}^L \nu_i J^\theta(t, x, i)^2 \sum_{j=1}^L q_{i j}-\sum_{i, j=1}^L \nu_i q_{i j} J^\theta(t, x, i) J^\theta(t, x, j) .
\end{align*}
The first two terms vanish, and therefore
$$
\mathcal{E}_\nu(J^\theta ; t, x)=\frac{1}{2} \sum_{i, j=1}^L \nu_i q_{i j}[J^\theta(t, x, j)-J^\theta(t, x, i)]^2 .
$$
Thus, $\mathcal{E}_\nu$ measures the squared variation of the critic across regime transitions, weighted by the stationary regime probabilities and transition intensities. This identity follows directly from stationarity and does not require the Markov chain to be reversible.

The jump contribution appearing in the MSTDE decomposition has a similar form, but it is weighted by the actual regime occupation of the controlled process. Writing
$$
\mu^\pi(d t, d x, i)=\rho^\pi(t, x, i) \bar{\mu}^\pi(d t, d x),
$$
the jump term is
\begin{equation}\label{eq:the-jump-term}
\mathcal{V}_{\beta, Q}(\theta)=\int_{[0, T] \times \mathbb{R}^d} e^{-2 \beta t} \sum_{i=1}^L \rho^\pi(t, x, i) \sum_{j \neq i} q_{i j}\left[J^\theta(t, x, j)-J^\theta(t, x, i)\right]^2 \bar{\mu}^\pi(d t, d x).
\end{equation}
If the regime weights coincide with the stationary distribution,
$$
\rho^\pi(t, x, i)=\nu_i, \quad \bar{\mu}^\pi \text {-a.e. },
$$
then \eqref{eq:the-jump-term} gives
$$
\mathcal{V}_{\beta, Q}(\theta)=2 \int_{[0, T] \times \mathbb{R}^d} e^{-2 \beta t} \mathcal{E}_\nu\left(J^\theta ; t, x\right) \bar{\mu}^\pi(d t, d x).
$$
Hence, in this special case, the jump contribution to MSTDE is exactly the occupation-weighted integral of the stationary regime-switching energy.

It is important, however, to distinguish the stationary distribution of the Markov chain from the conditional regime weights $\rho^\pi(t, x, i)$. Even if $\alpha_t$ has stationary marginal distribution $\nu$, one does not generally have
$$
\rho^\pi(t, x, i)=\nu_i,
$$
because the state process may depend on the current regime. Thus, the MSTDE jump contribution is generally governed by the state-regime occupation law rather than by the stationary distribution alone.

The expression for $\mathcal{V}_{\beta, Q}(\theta)$ also provides a direct interpretation of the jump term. It penalizes discrepancies
$J^\theta(t, x, j)-J^\theta(t, x, i)$
across regime transitions that occur with positive occupation weight and positive transition intensity. In particular, the jump contribution becomes small when the critic varies little across frequently visited regime transitions.

For a fixed critic and a fixed occupation law, the dependence on the off-diagonal transition rates $q_{i j}$ is linear:
$$
\mathcal{V}_{\beta, Q}(\theta) \propto \sum_{i=1}^L \sum_{j \neq i} q_{i j}\left[J^\theta(t, x, j)-J^\theta(t, x, i)\right]^2 .
$$
This linearity explains why increasing a transition intensity strengthens the corresponding jump-energy contribution when all other quantities are kept fixed. In the full regime-switching model, however, changing $Q$ also changes the state-regime occupation law and, in general, the policy value function itself. Therefore, the jump-energy representation should be interpreted as describing the quadratic-variation contribution generated by regime transitions, rather than as a general monotonicity statement for policy-evaluation error.

Most importantly, the jump term is a measure of regime-to-regime variation of the critic, not a direct measure of approximation error relative to the true value function. Consequently, minimizing MSTDE need not force the critic toward the true policy value, even when the true value function belongs to the critic class and the terminal condition is satisfied exactly. The following example illustrates this distinction explicitly.

The following example illustrates that minimizing MSTDE need not recover the true policy value, even when the critic class contains the exact value function and satisfies the terminal condition.
\begin{example}\label{exm:counterexample}
Consider a model on $[0, T]$, in which the continuous state is constant remains constant, while the regime process is a two-state continuous-time Markov chain with generator
$Q=\left(\begin{array}{cc}
-q & q \\
q & -q
\end{array}\right),  q>0 .$
Let the running rewards be regime dependent,
$$
\tilde{r}^\pi(t, x, 1)=1, \quad \tilde{r}^\pi(t, x, 2)=-1,
$$
and let the terminal reward be $h(x)=0$.
For $\beta \geq 0$, define
$$
\chi(t) := \frac{1-e^{-(\beta+2 q)(T-t)}}{\beta+2 q}.
$$
Then
$$
\chi^{\prime}(t)-(\beta+2 q) \chi(t)+1=0, \quad \chi(T)=0 .
$$
Hence the coupled policy-evaluation equations admit the solution
$$
J(t, x, 1 ; \pi)=\chi(t), \quad J(t, x, 2 ; \pi)=-\chi(t) .
$$
Thus, although the continuous state is constant, the value function differs across the two regimes.
Now consider the one-parameter critic family
$$
J^\theta(t, x, i)=\theta J(t, x, i ; \pi), \quad \theta \in \mathbb{R}.
$$
Since $J(T, x, i ; \pi)=0$, every member of this family satisfies the terminal condition
$J^\theta(T, x, i)=0$,
and the true value function is obtained at
$\theta=1$.
By linearity of the policy-evaluation operator and the fact that $J$ solves the exact PE equations,
$$
F^\theta(t, x, i)=(1-\theta) \tilde{r}^\pi(t, x, i), \quad i=1,2 .
$$
Therefore, $F^\theta=0$ if and only if $\theta=1$.
Thus, the Bellman or PE residual correctly identifies the true value function.

The behavior of MSTDE is different. Since the continuous state is constant, there is no Brownian contribution to the quadratic variation. The two regime values under $J^\theta$ differ by
$$
J^\theta(t, x, 1)-J^\theta(t, x, 2)=2 \theta \chi(t) .
$$
Hence, in either regime, the jump quadratic-variation density is
$$
\Gamma^\theta(t, x, i)=q(2 \theta \chi(t))^2=4 q \theta^2 \chi^2(t), \quad i=1,2 .
$$
It follows that the limiting MSTDE objective is
$$
\mathcal{V}_\beta(\theta)=4 q \theta^2 \int_0^T e^{-2 \beta t} \chi(t)^2 d t.
$$
Since $\chi(t)>0$ for $0 \leq t<T$, the integral is strictly positive. Therefore,
$$
\underset{\theta \in \mathbb{R}}{\arg \min } \mathcal{V}_\beta(\theta)=\{0\} .
$$
This differs from the true value parameter
$\theta=1$.
Thus, although the critic class contains the exact value function, the limiting MSTDE criterion is minimized by the zero critic rather than by the true value function.

The same phenomenon persists for sufficiently fine deterministic time partitions. Let $\mathcal{P}$ be such a partition. At the true value parameter $\theta=1$, $F^1=0 .$
Hence the drift-residual contribution in the preceding mesh-error estimate vanishes, and
$$
\operatorname{MSTDE}_{\mathcal{P}}(1)=\mathcal{V}_\beta(1)=4 q \int_0^T e^{-2 \beta t} \chi(t)^2 d t>0.
$$
At $\theta=0$, the jump quadratic-variation term disappears: $\mathcal{V}_\beta(0)=0$.
The corresponding residual contribution is
$$
D(0)=\int_0^T e^{-2 \beta t} d t.
$$
The mesh-error estimate therefore gives
$$
0 \leq \operatorname{MSTDE}_{\mathcal{P}}(0) \leq|\mathcal{P}| \int_0^T e^{-2 \beta t} d t \leq T|\mathcal{P}| .
$$
Consequently,
$$
\operatorname{MSTDE}_{\mathcal{P}}(0) \rightarrow 0 \quad \text { as }|\mathcal{P}| \rightarrow 0 .
$$
In contrast,
$$
\operatorname{MSTDE}_{\mathcal{P}}(1)=4 q \int_0^T e^{-2 \beta t} \chi^2(t) d t
$$
remains strictly positive. Therefore, for all sufficiently fine partitions,
$$
\operatorname{MSTDE}_{\mathcal{P}}(0)<\operatorname{MSTDE}_{\mathcal{P}}(1) .
$$
\end{example}

\begin{remark}
  The Proposition \ref{pro:zero-jump-energy} and Example \ref{exm:counterexample} highlight an important distinction between the martingale-loss and MSTDE criteria under Markov regime-switching. While the martingale loss measures the weighted approximation error of the value function, the limiting MSTDE criterion measures the predictable quadratic variation of the martingale component, including the additional contribution from regime transitions. Consequently, minimizing MSTDE may favor critic approximations with smaller differences across regimes, even when such approximations do not satisfy the policy-evaluation equation. The counterexample shows that this inconsistency can occur even when the critic class contains the true value function. Therefore, the vanishing of the jump contribution should not be interpreted as evidence of accurate policy evaluation, and minimizing MSTDE alone does not generally guarantee recovery of the true regime-dependent value function.
\end{remark}
\subsection{Martingale orthogonality and continuous-time TD}

The preceding martingale characterization provides a natural basis for continuous-time TD learning. Recall that the exact discounted TD increment on a partition $\mathcal{P}$ has been defined by
$$
\delta_k^\theta=M_{t_{k+1}}^\theta-M_{t_k}^\theta .
$$
To establish the corresponding continuous-time moment conditions, we introduce the undiscounted TD differential associated with $M^\theta$. By the definition of $M^\theta$ and the product rule, we have
\begin{equation*}
d M_t^\theta=e^{-\beta t} \left\{d J^\theta(t, \widetilde{X}_t^\pi, \alpha_t)+\left[\tilde{r}^\pi(t, \widetilde{X}_t^\pi, \alpha_{t-})-\beta J^\theta(t, \widetilde{X}_t^\pi, \alpha_{t-})\right] d t\right\}.
\end{equation*}
This motivates the following definition of the continuous-time TD increment:
\begin{align} \label{eq:partial-delta}
d\delta_t^\theta :=& e^{\beta t}\,dM_t^\theta \notag\\
=& dJ^\theta(t,\widetilde X_t^\pi,\alpha_t) + \Big[ \tilde r^\pi(t,\widetilde X_t^\pi,\alpha_{t-}) - \beta J^\theta(t,\widetilde X_t^\pi,\alpha_{t-}) \Big]\,dt.
\end{align}
Moreover, the semimartingale decomposition established in \eqref{eq:the-semimartingale-decomposition} gives
$$
d M_t^\theta=e^{-\beta t} F^\theta(t, \widetilde{X}_t^\pi, \alpha_{t-}) d t+d \mathcal{N}_t^\theta .
$$
Multiplying both sides by $e^{\beta t}$ and using the definition of $d \delta_t^\theta$, we obtain
\begin{equation} \label{eq:partial-delta-rewritten}
d\delta_t^\theta = F^\theta(t,\widetilde X_t^\pi,\alpha_{t-})\,dt + e^{\beta t}\,d\mathcal N_t^\theta,
\end{equation}
where $F^\theta(t,x,i)$ and $\mathcal{N}^\theta$ are defined in \eqref{eq:F_theta_def}, and \eqref{eq:mathcal{N}-theta}, respectively.
This representation separates the continuous-time TD increment \eqref{eq:partial-delta} into the policy-evaluation residual and a martingale term. In particular, when the critic coincides with the true value function, $F^\theta=0$, and the TD increment has zero expected integral against any suitable predictable test process. This provides the basis for the martingale orthogonality conditions used in continuous-time TD learning.

Let $\xi$ be a $\mathbb G$-predictable test process for which the following integrals are well defined. Multiplying \eqref{eq:partial-delta-rewritten} by $\xi_t$ and integrating over $[0,T]$, we obtain
$$
\int_0^T \xi_t d \delta_t^\theta=\int_0^T \xi_t F^\theta(t, \widetilde{X}_t^\pi, \alpha_{t-}) d t+\int_0^T e^{\beta t} \xi_t d \mathcal{N}_t^\theta .
$$
The stochastic integral $\int_0^T e^{\beta t} \xi_t d \mathcal{N}_t^\theta$
has zero expectation, because $e^{\beta t} \xi_t$ is predictable and satisfies the required integrability condition. Therefore,
$$
\mathbb{E}\int_0^T\xi_t\,d\delta_t^\theta = \mathbb{E}\int_0^T \xi_tF^\theta(t,\widetilde X_t^\pi,\alpha_{t-})\,dt.
$$
This identity is the basic regime-switching martingale orthogonality relation. It shows that {\it the martingale orthogonality condition (or the TD moment condition)}
\begin{equation}\label{eq:the-martingale-moment condition}
\mathbb{E}\int_0^T\xi_t\,d\delta_t^\theta=0
\end{equation}
is equivalent to requiring the Bellman residual to be orthogonal to the chosen test process:
\begin{equation} \label{eq:the-TD-moment-condition}
\mathbb{E} \int_0^T \xi_t F^\theta(t, \widetilde{X}_t^\pi, \alpha_{t-}) d t=0.
\end{equation}
When $J^\theta(t, x, i)=J(t, x, i;\pi)$, the Bellman residual vanishes, i.e., $F^\theta(t, x, i)=0,$
and hence \eqref{eq:the-TD-moment-condition} holds for every admissible predictable test process.
Conversely, if
$\mathbb{E} \int_0^T \xi_t F^\theta(t, \widetilde{X}_t^\pi, \alpha_{t-}) d t=0$
for every bounded predictable test process $\xi$, then
$$
F^\theta(t, \widetilde{X}_t^\pi, \alpha_{t-})=0, \quad d t \otimes d \mathbb{P} \text {-a.e. }
$$
Together with the terminal condition
$J^\theta(T, x, i)=h(x)$,
this recovers the policy-evaluation equation along the state-regime process.

The preceding semimartingale decomposition establishes a direct connection between continuous-time TD learning and the policy-evaluation equation. By exploiting the semimartingale decomposition of the TD increment, the Bellman residual can be expressed through martingale orthogonality conditions that can be evaluated from observed trajectories without explicitly computing the model generator.
In practice, TD algorithms use only a finite number of test processes, so the resulting moment conditions do not necessarily guarantee that the learned critic equals the true value function. Such a guarantee requires the critic class to contain the true value function and the selected moment conditions to identify it uniquely.
This unified framework provides the basis for continuous-time TD, TD with eligibility traces, LSTD, and gradient-TD methods, which differ mainly in their choice of test processes and numerical procedures for solving the resulting moment equations.

\begin{itemize}
    \item \textbf{Regime-Switching CTD(0).} To construct a parameter-based TD estimating equation, we choose a test process that reflects how the critic changes with its parameters. A natural choice is the gradient of the critic with respect to $\theta$, whose components represent the directions of variation within the critic family. We therefore set
    $\xi_t^\theta=\nabla_\theta J^\theta(t, \widetilde{X}_t^\pi, \alpha_{t-})$.
    Substituting this test process into \eqref{eq:the-martingale-moment condition} gives
    \begin{equation}\label{eq:CTD0}
    \mathbb{E}\int_0^T \nabla_\theta J^\theta(t,\widetilde X_t^\pi,\alpha_{t-}) \,d\delta_t^\theta =0.
    \end{equation}
    or, equivalently,
    $$
    \mathbb{E}\int_0^T \nabla_\theta J^\theta(t,\widetilde X_t^\pi,\alpha_{t-}) F^\theta(t,\widetilde X_t^\pi,\alpha_{t-})\,dt = 0.
    $$
    Thus, CTD(0) requires the PE residual $F^\theta$ to be orthogonal, on average, to the directions generated by the critic derivatives. If some parameter $\theta^{*}$ represents the true value function,i.e.,
    $J^{\theta^{*}}(t, x, i)=J(t, x, i ; \pi)$,
    then $F^{\theta^{*}}(t, x, i)=0$,
    and therefore $\theta^{*}$ satisfies the CTD(0) moment equation. The moment equation may nevertheless admit more than one solution when the critic is represented by a general function approximator.
    While CTD(0) uses the current critic gradient as the test process, historical gradient information can also be incorporated through an eligibility trace. Specifically, define the continuous-time {\it eligibility trace} by  % 资格迹
    $$
    z_t^\theta = \int_0^t e^{-\kappa(t-s)} \nabla_\theta J^\theta(s,\widetilde X_s^\pi,\alpha_{s-})\,ds, \qquad \kappa\ge0.
    $$
    Equivalently, $dz_t^\theta$ satisfies
    $$dz_t^\theta = \Big[ \nabla_\theta J^\theta(t,\widetilde X_t^\pi,\alpha_{t-}) - \kappa z_t^\theta \Big]\,dt, \qquad z_0^\theta=0.
    $$
    Since $z^\theta$ is continuous and adapted, it is predictable and can therefore be used as a test process. The corresponding moment condition becomes
    $$  \mathbb{E}\int_0^Tz_t^\theta\,d\delta_t^\theta=0.
    $$
    This gives an exponential-trace extension of CTD(0). Over a discrete interval of length $h_k$, the trace decays by the factor $e^{-\kappa h_k}$.
    Hence CTD(0) and the eligibility-trace formulation are based on the same TD increment $d \delta_t^\theta$; their difference lies only in the choice of test process.
    \item \textbf{Linear Approximation and Regime-Switching CLSTD.} A particularly simple moment equation is obtained when the critic is affine in its parameters. Let
    $$J^\theta(t,x,i) = J^0(t,x,i)+\psi(t,x,i)^\top\theta,$$
    where $J^0$ is a fixed function and
    $J^0(T,x,i)=h(x), \psi(T,x,i)=0$.
    This construction guarantees $J^\theta(T, x, i)=h(x)$ for every $\theta$.
    Let $\xi_t \in \mathbb{R}^{d_\theta}$ be a predictable test process that does not depend on $\theta$. Substituting the affine critic into the TD increment gives
    \begin{equation*}
        d\delta_t^\theta =d\delta_t^0 + \Big[ d(\psi(t,\widetilde X_t^\pi,\alpha_t)^\top) - \beta\psi(t,\widetilde X_t^\pi,\alpha_{t-})^\top\,dt \Big]\theta.
    \end{equation*}
    where
    $$
    d \delta_t^0=d J^0(t, \widetilde{X}_t^\pi, \alpha_t)+\left[\mathcal{R}(t, \widetilde{X}_t^\pi, \alpha_{t-}, \pi)-\beta J^0(t, \widetilde{X}_t^\pi, \alpha_{t-})\right] d t .
    $$
    The moment equation \eqref{eq:the-martingale-moment condition}
    therefore becomes
    \begin{equation*}
        \Bigg[ \mathbb{E}\int_0^T \xi_t\Big\{ \beta\psi(t,\widetilde X_t^\pi,\alpha_{t-})^\top\,dt - d\bigl(\psi(t,\widetilde X_t^\pi,\alpha_t)^\top\bigr) \Big\} \Bigg]\theta = \mathbb{E}\int_0^T\xi_t\,d\delta_t^0.
    \end{equation*}
    Thus the critic parameters are determined by a linear system
    $$A \theta=b,$$
    with
    $$
    A=\mathbb{E} \int_0^T \xi_t\left\{\beta \psi(t, \widetilde{X}_t^\pi, \alpha_{t-})^{\top} d t-d(\psi(t, \widetilde{X}_t^\pi, \alpha_t)^{\top})\right\},
    \quad
    b=\mathbb{E} \int_0^T \xi_t d \delta_t^0.
    $$
    If $A$ is nonsingular, the moment equation \eqref{eq:the-martingale-moment condition} has the unique solution
    $\theta=A^{-1} b$.
    Replacing the expectations by empirical averages over observed trajectories gives a continuous-time least-squares TD implementation. The term $d \delta_t^0$ contains the contribution of the fixed offset $J^0$, which is required when the terminal reward is nonzero. Importantly, both $A$ and $b$ can be constructed from observed critic and feature increments, so the environmental drift, diffusion coefficients, and switching generator do not need to be evaluated explicitly.

    \item \textbf{General Critic and GTD Formulation.} For a general nonlinear critic, define the vector of martingale moments by
    $$
    m(\theta;\pi) = \mathbb{E}\int_0^T\xi_t^\theta\,d\delta_t^\theta.
   $$
   Equivalently,
   $$
   m(\theta;\pi)= \mathbb{E}\int_0^T \xi_t^\theta F^\theta(t,\widetilde X_t^\pi,\alpha_{t-})\,dt.
   $$
   The basic PE condition is therefore
   $m(\theta ; \pi)=0$.
   Instead of solving this vector equation directly, one may minimize its weighted squared norm:
   $$
   L_{\mathrm{GTD}}(\theta;\pi) = \frac{1}{2}m(\theta;\pi)^\top Wm(\theta;\pi), \qquad W\succ0,
   $$
   where $W$ is fixed.
   Its gradient is
   $\nabla_\theta L_{\mathrm{GTD}}(\theta ; \pi)=\mathcal{J}_m(\theta ; \pi)^{\top} W m(\theta ; \pi)$,
   where $\mathcal{J}_m(\theta ; \pi)=\frac{\partial m(\theta ; \pi)}{\partial \theta^{\top}}$
   is the Jacobian of the moment vector. Both the critic and the test process may depend on $\theta$. Accordingly, for each parameter component $\theta_j$,
   \begin{equation*}
        \frac{\partial m(\theta;\pi)}{\partial\theta_j} =\mathbb{E}\int_0^T \frac{\partial\xi_t^\theta}{\partial\theta_j} \,d\delta_t^\theta+ \mathbb{E}\int_0^T\xi_t^\theta \Bigg[ d\left( \frac{\partial J^\theta}{\partial\theta_j} (t,\widetilde X_t^\pi,\alpha_t) \right) - \beta \frac{\partial J^\theta}{\partial\theta_j} (t,\widetilde X_t^\pi,\alpha_{t-})\,dt \Bigg].
    \end{equation*}
    The first term accounts for the dependence of the test process on $\theta$, whereas the second term accounts for the dependence of the TD increment on the critic parameters.
    For the affine critic
    $J^\theta(t, x, i)=J^0(t, x, i)+\psi(t, x, i)^{\top} \theta$
    with a parameter-independent test process $\xi$, the moment vector has the affine form $m(\theta ; \pi)=b-A \theta$,
    up to the sign convention used in defining $A$. Hence,
    $L_{\mathrm{GTD}}(\theta ; \pi)=\frac{1}{2}(b-A \theta)^{\top} W(b-A \theta)$
    is a quadratic objective in $\theta$.
    For a nonlinear critic or a parameter-dependent test process, $m(\theta ; \pi)$ is generally nonlinear, and the resulting GTD objective need not be quadratic.
    The relation among these constructions can therefore be summarized as
    $\mathbb{E} \int_0^T \xi_t^\theta d \delta_t^\theta=0$
    as the common martingale moment condition,
    $\xi_t^\theta=\nabla_\theta J^\theta(t, \widetilde{X}_t^\pi, \alpha_{t-})$
    for CTD(0),
    $\xi_t^\theta=z_t^\theta$
    for the eligibility-trace version,
    $J^\theta(t, x, i)=J^0(t, x, i)+\psi(t, x, i)^{\top} \theta$
    for the linear CLSTD formulation, and
    $L_{\mathrm{GTD}}=\frac{1}{2} m^{\top} W m$
    for the GTD-type formulation.
\end{itemize}

\section{Policy Gradient under regime-switching}
\label{sec:PG}

Let $\Phi \subseteq \mathbb{R}^{d_\phi}$ be open, and consider a family of admissible policy densities $\{\pi^\phi: \phi \in \Phi\}$ with respect to a fixed reference measure $\lambda$. Suppose that the densities are positive and continuously differentiable in $\phi$, with a common support independent of $\phi$. Throughout the following analysis, $b, \sigma, r, h$, and $Q$ are independent of $\phi$ when the action is fixed. We further assume that $J(t,x,i;\pi^\phi)$ is sufficiently smooth in $(t,x,\phi)$ to justify the following differentiation under the policy integral and the policy-evaluation equation.
Motivated by the occupation-measure approach to controlled Markov processes (see, \cite{bhatt1996occupation,kurtz1998existence}), we impose the following condition to connect the relaxed policy formulation with observed trajectories.
\begin{assumption}\label{assum:occupation-matching}
Fix a parameterized policy $\pi^\phi$ and an initial condition $(t, x, i)$. Let $(X, \alpha, a)$ be an admissible controlled realization with predictable actions, and let $(\widetilde{X}^\pi, \alpha)$ be the corresponding relaxed state-regime process under $\pi^\phi$. We assume that, for every bounded Borel measurable function $f$ : $[t, T] \times \mathbb{R}^d \times \mathcal{M} \times \mathcal{A} \rightarrow \mathbb{R}$,
\begin{equation*}
\mathbb{E}_{t, x, i} \int_t^T f(s, X_s, \alpha_{s-}, a_s) d s =\mathbb{E}_{t, x, i}^{\pi^\phi} \int_t^T \int_{\mathcal{A}} f(s, \widetilde{X}_s^{\pi^\phi}, \alpha_{s-}, a) \pi^\phi(d a \mid s, \widetilde{X}_s^{\pi^\phi}, \alpha_{s-}) d s.
\end{equation*}
\end{assumption}
Here, $\mathbb{E}_{t, x, i}^{\pi^\phi}$ denotes expectation for the relaxed state-regime process starting from $(x, i)$ at time $t$.
This condition requires the controlled realization and its relaxed counterpart to have the same expected joint occupation measure on the time-state-regime-action space. It does not require pathwise equality of the two processes. The identity also extends to unbounded measurable test functions whenever the corresponding absolute integrability conditions hold.

\begin{remark}
The Assumption \ref{assum:occupation-matching} connects the relaxed policy formulation with observable controlled trajectories, thereby justifying the trajectory-based policy-gradient representation. A sufficient condition is that, for almost every $s \in[t, T]$,
\begin{equation*}
\mathcal{L}(a_s \mid X_s, \alpha_{s-}) =\pi^\phi(\cdot \mid s, X_s, \alpha_{s-}),\quad
\mathcal{L}(X_s, \alpha_{s-})  =\mathcal{L}(\widetilde{X}_s^{\pi^\phi}, \alpha_{s-}),
\end{equation*}
where the joint state-regime laws are taken under their respective formulations. The desired identity then follows from conditional expectation and Fubini's theorem, under suitable integrability conditions. However, sampling actions from $\pi^\phi$ at discrete decision times does not guarantee Assumption \ref{assum:occupation-matching} between these times. A separate consistency argument is therefore needed for finite-grid implementations as the mesh size tends to zero.
\end{remark}

Differentiating with respect to $\phi$ yields
\begin{equation}\label{eq:the-normalization-identity}
\int_{\mathcal{A}} \nabla_\phi \log \pi^\phi(a \mid t, x, i) \pi^\phi(a \mid t, x, i) \lambda(d a) =\int_{\mathcal{A}} \nabla_\phi \pi^\phi(a \mid t, x, i) \lambda(d a) =0 .
\end{equation}
%Recall that the policy averaged extended generator is given by$$(\tilde{\mathcal{L}}^{\pi^\phi} f)(t, x, i)=\int_{\mathcal{A}}\left(\mathcal{L}^a f\right)(t, x, i) \pi^\phi(a \mid t, x, i) \lambda(d a).$$
For the parameterized policy $\pi^\phi$, the policy-evaluation equation takes the form
\begin{equation}\label{eq:policy-evaluation-equ-partial-phi}
(\widetilde{\mathcal{L}}^{\pi^\phi} J)(t, x, i ; \pi^\phi)+\mathcal{R}(t, x, i, \pi^\phi) -\beta J(t, x, i ; \pi^\phi)=0,
\end{equation}
with terminal condition $J(T, x, i ; \pi^\phi)=h(x)$.
Differentiating \eqref{eq:policy-evaluation-equ-partial-phi} with respect to $\phi$
%and using$\nabla_\phi \pi^\phi(a \mid t, x, i)=\pi^\phi(a \mid t, x, i) \nabla_\phi \log \pi^\phi(a \mid t, x, i)$
, we obtain
\begin{equation}\label{eq:the-auxiliary-PE-equation}
\begin{aligned}
0= & (\widetilde{\mathcal{L}}^{\pi^\phi} \nabla_\phi J)(t, x, i ; \pi^\phi)-\beta \nabla_\phi J(t, x, i ; \pi^\phi) +\int_{\mathcal{A}} \nabla_\phi \log \pi^\phi(a \mid t, x, i)\Big[(\mathcal{L}^a J)(t, x, i ; \pi^\phi) \\
&\quad+r(t, x, i, a)+\gamma p(t, x, i, a, \pi^\phi(\cdot \mid t, x, i))\Big] \pi^\phi(a \mid t, x, i) \lambda(d a) \\
&\quad +\gamma \int_{\mathcal{A}} \nabla_\phi p(t, x, i, a, \pi^\phi(\cdot \mid t, x, i)) \pi^\phi(a \mid t, x, i) \lambda(d a).
\end{aligned}
\end{equation}
Here, $\nabla_\phi p$ denotes the derivative of the regularizer with respect to $\phi$ through its dependence on $\pi^\phi$, while $(t, x, i, a)$ are held fixed.
%The score identity \eqref{eq:the-normalization-identity} implies that any term independent of the action $a$ has zero score-weighted policy average. In particular,$$\int_{\mathcal{A}} \nabla_\phi \log \pi^\phi(a \mid t, x, i)\left[-\beta J(t, x, i ; \pi^\phi)\right] \pi^\phi(a \mid t, x, i) \lambda(d a)=0 .$$Hence, the term $-\beta J(t, x, i ; \pi^\phi)$ can be included in the score-weighted expression in \eqref{eq:the-auxiliary-PE-equation} without changing its value.
Define the auxiliary reward by
\begin{align*}
\check{r}(t, x, i, a ; \phi)&:=\nabla_\phi \log \pi^\phi(a \mid t, x, i)\Big[(\mathcal{L}^a J)(t, x, i ; \pi^\phi)+r(t, x, i, a)+\gamma p(t, x, i, a, \pi^\phi(\cdot \mid t, x, i)) \\
&\quad-\beta J(t, x, i ; \pi^\phi)\Big]
+\gamma \nabla_\phi p(t, x, i, a, \pi^\phi(\cdot \mid t, x, i)) .
\end{align*}
Equation \eqref{eq:the-auxiliary-PE-equation} can then be written as %the auxiliary policy-evaluation equation
\begin{equation}\label{eq:nabla_phi-J}
\left\{\begin{array}{l}
(\widetilde{\mathcal{L}}^{\pi^\phi} \nabla_\phi J)(t, x, i ; \pi^\phi)+\int_{\mathcal{A}} \check{r}(t, x, i, a ; \phi) \pi^\phi(a \mid t, x, i) \lambda(d a)-\beta \nabla_\phi J(t, x, i ; \pi^\phi)=0, \\
\nabla_\phi J(T, x, i ; \pi^\phi)=0 .
\end{array}\right.
\end{equation}
%The terminal condition $\nabla_\phi J(T, x, i ; \pi^\phi)=0$ follows from the independence of $h(x)$ with respect to $\phi$.
To identify the contribution of regime-switching to \eqref{eq:nabla_phi-J}, observe that the terms $\frac{\partial J}{\partial t}(t, x, i ; \pi^\phi)$, $\sum_{j =1}^L q_{i j}J(t, x, j ; \pi^\phi)$,
and $-\beta J(t, x, i ; \pi^\phi)$ are independent of the current action $a$ for fixed $(t, x, i)$. Hence, by the score identity
\eqref{eq:the-normalization-identity},
\begin{equation*}
\int_{\mathcal{A}} \nabla_\phi \log \pi^\phi(a \mid t, x, i)\Big[J_t(t, x, i ; \pi^\phi)+\sum_{j=1}^L q_{i j} J(t, x, j ; \pi^\phi)-\beta J(t, x, i ; \pi^\phi)\Big] \pi^\phi(a \mid t, x, i) \lambda(d a)=0 .
\end{equation*}
Consequently, these action-independent terms vanish from the score-weighted policy average in \eqref{eq:nabla_phi-J}. However, regime-switching remains explicitly accounted for in
$\tilde{\mathcal{L}}^{\pi^\phi} \nabla_\phi J$ through the term $\sum_{j =1}^L q_{i j}J(t, x, j ; \pi^\phi)$.
More generally, for any scalar function $c(t, x, i)$ independent of $a$, define
$$
\check{r}_c(t, x, i, a ; \phi):=\check{r}(t, x, i, a ; \phi)+ c(t, x, i)\nabla_\phi \log \pi^\phi(a \mid t, x, i).
$$
By \eqref{eq:the-normalization-identity},
\begin{equation*}
 \int_{\mathcal{A}} \check{r}_c(t, x, i, a ; \phi) \pi^\phi(a \mid t, x, i) \lambda(d a)=\int_{\mathcal{A}} \check{r}(t, x, i, a ; \phi) \pi^\phi(a \mid t, x, i) \lambda(d a) .
\end{equation*}
Hence, action-independent terms can be added to or removed from the score-weighted factor without changing \eqref{eq:nabla_phi-J}. This equivalence holds after integration with respect to $\pi^\phi(a \mid t, x, i)$, although the corresponding expressions for $\check{r}_c(t, x, i, a ; \phi)$ and $\check{r}(t, x, i, a ; \phi)$ may differ pointwise.

For the subsequent trajectory-based policy-gradient representation, we retain $\check{r}(t, x, i, a ; \phi)$ in its original form involving $\mathcal{L}^a J$. Under the corresponding controlled realization, Itô's formula expresses $(\mathcal{L}^{a_t} J) d t$ in terms of the observed increment $d J(t, X_t, \alpha_t ; \pi^\phi)$ and a martingale increment. Under the required predictability and integrability conditions, the latter has zero expectation when integrated against the policy score. This allows the policy gradient to be represented using observed state-regime trajectories, rewards, and policy scores, without explicitly evaluating $b, \sigma$, or $Q$. In particular, regime transitions are captured by the observed jumps in $d J(t, X_t, \alpha_t ; \pi^\phi)$, rather than being omitted from the policy-gradient representation.

We next derive a probabilistic representation of the policy gradient based on the preceding gradient equation. To justify this representation, we impose the following conditions. Specifically, for every initial condition $(t, x, i)$, we assume that the parameter gradient $\nabla_\phi J(t, x, i ; \pi^\phi)$ and the auxiliary reward $\check{r}(t, x, i, a ; \phi)$ satisfy
\begin{equation*}
\mathbb{E}_{t, x, i}^{\pi^\phi} \bigg[\sup_{t \leq s \leq T}\|\nabla_\phi J(s, \widetilde{X}_s^{\pi^\phi}, \alpha_s ; \pi^\phi)\|+\int_t^T \int_{\mathcal{A}}\|\check{r}(s, \widetilde{X}_s^{\pi^\phi}, \alpha_s, a ; \phi)\| \pi^\phi(a \mid s, \widetilde{X}_s^{\pi^\phi}, \alpha_s) \lambda(d a) d s\bigg]<\infty.
\end{equation*}
%It follows from \eqref{eq:nabla_phi-J} that\begin{align*}(\tilde{\mathcal{L}}^\phi \nabla_\phi J)(s, x, i ; \pi^\phi)-\beta \nabla_\phi J(s, x, i ; \pi^\phi)  =-\int_{\mathcal{A}} \check{r}(s, x, i, a ; \phi) \pi^\phi(a \mid s, x, i) \lambda(d a) .\end{align*}
Let $(\tau_n)_{n \geq 1}$ be a localizing sequence for the local martingale terms, with $\tau_n \uparrow T$ almost surely. Applying the generalized Itô's formula componentwise to $e^{-\beta(s-t)} \nabla_\phi J(s, \widetilde{X}_s^{\pi^\phi}, \alpha_s ; \pi^\phi)$ on $[t, u \wedge \tau_n]$, where $t<u<T$, and using \eqref{eq:nabla_phi-J} yields
\begin{align*}
\nabla_\phi J(t, x, i ; \pi^\phi)=\mathbb{E}_{t, x, i}^{\pi^\phi} & \bigg[e^{-\beta(u \wedge \tau_n-t)} \nabla_\phi J(u \wedge \tau_n, \widetilde{X}_{u \wedge \tau_n}^{\pi^\phi}, \alpha_{u \wedge \tau_n} ; \pi^\phi) \\
& +\int_t^{u \wedge \tau_n} e^{-\beta(s-t)} \int_{\mathcal{A}} \check{r}(s, \widetilde{X}_s^{\pi^\phi}, \alpha_s, a ; \phi) \pi^\phi(a \mid s, \widetilde{X}_s^{\pi^\phi}, \alpha_s) \lambda(d a) d s\bigg] .
\end{align*}
where the stopped local martingale has zero expectation. The use of $\alpha_s$ in the ordinary time integral is equivalent to using $\alpha_{s-}$, since the two coincide for $d s \otimes d \mathbb{P}$-almost every $(s, \omega)$.

We next remove the localization. Since $\beta \geq 0$, the stopped terminal term is bounded by
\begin{equation*}
\left\|e^{-\beta(u \wedge \tau_n-t)} \nabla_\phi J(u \wedge \tau_n, \widetilde{X}_{u \wedge \tau_n}^{\pi^\phi}, \alpha_{u \wedge \tau_n} ; \pi^\phi)\right\| \leq \sup_{t \leq s \leq T}\left\|\nabla_\phi J(s, \widetilde{X}_s^{\pi^\phi}, \alpha_s ; \pi^\phi)\right\| .
\end{equation*}
Moreover, using the triangle inequality and the fact that $\alpha_s=\alpha_{s-}$ for $d s \otimes d \mathbb{P}$-almost every $(s, \omega)$, we have
\begin{align*}
&\left\|\int_t^{u \wedge \tau_n} e^{-\beta(s-t)} \int_{\mathcal{A}} \check{r}(s, \widetilde{X}_s^{\pi^\phi}, \alpha_{s-}, a ; \phi) \pi^\phi(a \mid s, \widetilde{X}_s^{\pi^\phi}, \alpha_{s-}) \lambda(d a) d s\right\| \\
&\quad\leq \int_t^T \int_{\mathcal{A}}\left\|\check{r}(s, \widetilde{X}_s^{\pi^\phi}, \alpha_s, a ; \phi)\right\| \pi^\phi(a \mid s, \widetilde{X}_s^{\pi^\phi}, \alpha_s) \lambda(d a) d s.
\end{align*}
Both upper bounds are integrable by assumption. Therefore, dominated convergence allows $n \rightarrow \infty$, giving
\begin{equation*}
\nabla_\phi J(t, x, i ; \pi^\phi)=\mathbb{E}_{t, x, i}^{\pi^\phi} \int_t^T e^{-\beta(s-t)} \int_{\mathcal{A}} \check{r}(s, \widetilde{X}_s^{\pi^\phi}, \alpha_s, a ; \phi) \pi^\phi(a \mid s, \widetilde{X}_s^{\pi^\phi}, \alpha_s) \lambda(d a) d s.
\end{equation*}
Finally, since the terminal reward is independent of $\phi$,
$$
\nabla_\phi J(T, x, i ; \pi^\phi)=0, \quad(x, i) \in \mathbb{R}^d \times \mathcal{M} .
$$
By the continuity of $\nabla_\phi J(t, x, i ; \pi^\phi)$ up to the terminal time, and the path properties of the relaxed state-regime process,
$$
\nabla_\phi J(u, \widetilde{X}_u^{\pi^\phi}, \alpha_u ; \pi^\phi) \rightarrow 0 \quad \text { a.s. as } u \uparrow T .
$$
The preceding integrability condition again permits dominated convergence. Letting $u \uparrow T$, we obtain
$$
\nabla_\phi J(t, x, i ; \pi^\phi)=\mathbb{E}_{t, x, i}^{\pi^\phi} \int_t^T e^{-\beta(s-t)} \int_{\mathcal{A}} \check{r}(s, \widetilde{X}_s^{\pi^\phi}, \alpha_s, a ; \phi) \pi^\phi(a \mid s, \widetilde{X}_s^{\pi^\phi}, \alpha_s) \lambda(d a) d s.
$$
Thus, the policy gradient can be representation as the discounted policy evaluation of the auxiliary reward $\check{r}(t, x, i, a ; \phi)$ under the current relaxed policy $\pi^\phi$.

To obtain a trajectory-based representation, consider an admissible controlled realization $(X, \alpha, a)$ starting from $(x, i)$ at time $t$, with predictable actions $a_s$ and state-regime characteristics corresponding to the controlled generator $\mathcal{L}^{a_s}$. Suppose that the randomized realization preserves the environmental martingales and switching compensators and satisfies the Assumption \ref{assum:occupation-matching} for the auxiliary-reward integrands. We also assume that the score-weighted stochastic integrals are integrable and have zero expectation.

Under these conditions, Itô's formula gives
$$
d J(s, X_s, \alpha_s ; \pi^\phi)=(\mathcal{L}^{a_s} J)(s, X_s, \alpha_{s-} ; \pi^\phi) d s+d M_s^J,
$$
where $M^J$ denotes the local martingale component of the value process.

The martingale term has zero expectation when multiplied by the predictable policy score and integrated over time. Consequently, the action-dependent generator term in the auxiliary reward can be replaced, in expectation, by the observed value increment. Combining this identity with the Assumption \ref{assum:occupation-matching} gives the following trajectory-based policy-gradient representation.

\begin{theorem}[Trajectory policy-gradient representation under regime-switching]
Let the preceding regularity, integrability, and occupation-matching assumptions hold. Then, for $(t, x, i) \in$ $[0, T) \times \mathbb{R}^d \times \mathcal{M}$, the policy gradient admits the trajectory representation
\begin{align}\label{eq:PG}
\nabla_\phi J(t, x, i ; \pi^\phi)=&\mathbb{E}_{t, x, i}^\phi\bigg[\int_t^T e^{-\beta(s-t)}\Big\{  \nabla_\phi \log \pi^\phi(a_s \mid s, X_s, \alpha_{s-}) \times\big[d J(s, X_s, \alpha_s ; \pi^\phi)\notag\\
&  +(r(s, X_s, \alpha_{s-}, a_s)+\gamma p(s, X_s, \alpha_{s-}, a_s, \pi^\phi(\cdot \mid s, X_s, \alpha_{s-}))\notag\\
&\qquad\qquad-\beta J(s, X_s, \alpha_{s-} ; \pi^\phi)) d s\big] \notag\\
& +\gamma \nabla_\phi p(s, X_s, \alpha_{s-}, a_s, \pi^\phi(\cdot \mid s, X_s, \alpha_{s-})) \Big\}d s\bigg].
\end{align}
Here, $\mathbb{E}_{t, x, i}^\phi$ denotes expectation under the randomized controlled realization $(X, \alpha, a)$.
\end{theorem}
\begin{proof}
By the probabilistic  policy-gradient representation in \eqref{eq:nabla_phi-J} and the Assumption \ref{assum:occupation-matching},
\begin{align*}
\nabla_\phi J(t, x, i ; \pi^\phi)=\mathbb{E}_{t, x, i}^\phi \bigg[\int_t^T e^{-\beta(s-t)} & \Big\{\nabla_\phi \log \pi^\phi(a_s \mid s, X_s, \alpha_{s-})\times[(\mathcal{L}^{a_s} J)(s, X_s, \alpha_{s-} ; \pi^\phi) \\
& \quad +r(s, X_s, \alpha_{s-}, a_s)+\gamma p(s, X_s, \alpha_{s-}, a_s, \pi^\phi(\cdot \mid s, X_s, \alpha_{s-})) \\
&\qquad\qquad-\beta J(s, X_s, \alpha_{s-} ; \pi^\phi)] \\
& \quad+\gamma \nabla_\phi p(s, X_s, \alpha_{s-}, a_s, \pi^\phi(\cdot \mid s, X_s, \alpha_{s-}))\Big\} d s \bigg].
\end{align*}
From the generalized Itô formula of the value process,
$$
(\mathcal{L}^{a_s} J)(s, X_s, \alpha_{s-} ; \pi^\phi) d s=d J(s, X_s, \alpha_s ; \pi^\phi)-d M_s^J,
$$
where $M^J$ contains both the Brownian and compensated regime-switching martingale components. Consequently,
\begin{align*}
&\nabla_\phi \log \pi^\phi(a_s \mid s, X_s, \alpha_{s-})(\mathcal{L}^{a_s} J)(s, X_s, \alpha_{s-} ; \pi^\phi) d s \\
&\quad=\nabla_\phi \log \pi^\phi(a_s \mid s, X_s, \alpha_{s-}) d J(s, X_s, \alpha_s ; \pi^\phi) -\nabla_\phi \log \pi^\phi(a_s \mid s, X_s, \alpha_{s-}) d M_s^J .
\end{align*}
Since $\nabla_\phi \log \pi^\phi(a_s \mid s, X_s, \alpha_{s-})$ satisfies the corresponding square-integrability condition, we have
$$
\mathbb{E}_{t, x, i}^\phi \int_t^T e^{-\beta(s-t)} \nabla_\phi \log \pi^\phi(a_s \mid s, X_s, \alpha_{s-}) d M_s^J=0.
$$
Hence,
\begin{align*}
&\mathbb{E}_{t, x, i}^\phi \int_t^T e^{-\beta(s-t)} \nabla_\phi \log \pi^\phi(a_s \mid s, X_s, \alpha_{s-})(\mathcal{L}^{a_s} J)(s, X_s, \alpha_{s-} ; \pi^\phi) d s \\
&\quad=\mathbb{E}_{t, x, i}^\phi \int_t^T e^{-\beta(s-t)} \nabla_\phi \log \pi^\phi(a_s \mid s, X_s, \alpha_{s-}) d J(s, X_s, \alpha_s ; \pi^\phi).
\end{align*}
Combining this identity with the Assumption \ref{assum:occupation-matching} yields the desired trajectory policy-gradient representation \eqref{eq:PG}.
\end{proof}
The first term in the policy-gradient representation \eqref{eq:PG} is a stochastic integral with respect to the value process $J(s, X_s, \alpha_s ; \pi^\phi)$, with predictable integrand
$e^{-\beta(s-t)} \nabla_\phi \log \pi^\phi(a_s \mid s, X_s, \alpha_{s-})$,
whereas the remaining terms are ordinary time integrals (all integrals of the form $\int_0^T(\cdots) d t$ ). \textbf{Thus, once the current policy value has been obtained from policy evaluation, the representation can be evaluated using observed value increments, rewards, and quantities determined by the policy and the regularizer; it does not require explicit evaluation of the model coefficients $b, \sigma$, or the regime generator $Q$.}

The absence of an explicit regime-switching term in \eqref{eq:PG} does not mean that regime-switching has been ignored. For each fixed $(s, x, i)$, the term
$\sum_{j=1}^L q_{i j} J(s, x, j ; \pi^\phi)$
is independent of the action $a$.
Hence, by the score identity \eqref{eq:the-normalization-identity}, its contribution to the score-weighted policy average vanishes:
\begin{equation}\label{eq:score-weighted-policy-average}
\begin{aligned}
\int_{\mathcal{A}} & \nabla_\phi \log \pi^\phi(a \mid s, x, i)\times \sum_{j=1}^L q_{i j} J(s, x, j ; \pi^\phi) \pi^\phi(a \mid s, x, i) \lambda(d a) \\
& =\sum_{j=1}^L q_{i j} J(s, x, j ; \pi^\phi) \int_{\mathcal{A}} \nabla_\phi \log \pi^\phi(a \mid s, x, i) \pi^\phi(a \mid s, x, i) \lambda(d a)=0 .
\end{aligned}
\end{equation}
Meanwhile, the observed increase $d J(s, X_s, \alpha_s ; \pi^\phi)$ already captures changes in the value process caused by regime transitions. Therefore, the absence of an explicit regime-switching term in the trajectory representation follows from the score identity \eqref{eq:the-normalization-identity} and the Itô decomposition, rather than from neglecting the regime-switching dynamics. The action-independent terms $J_t(t, x, i ; \pi^\phi)$ and $-\beta J(t, x, i ; \pi^\phi)$ can be removed similarly from the score-weighted policy average \eqref{eq:score-weighted-policy-average}. We retain $-\beta J(t, x, i ; \pi^\phi)$ in \eqref{eq:PG} to preserve the continuous-time TD increment \eqref{eq:partial-delta} used in the subsequent actor-critic formulation.

Finally, this representation \eqref{eq:PG} applies to a continuous-time randomized realization satisfying the Assumption \ref{assum:occupation-matching}. In a finite-grid implementation, actions are sampled according to
$$
a_k \sim \pi^\phi(\cdot \mid t_k, X_{t_k}, \alpha_{t_k}), \quad a_s=a_k, \quad s \in[t_k, t_{k+1}),
$$
Because the state and regime processes continue to evolve between decision times while the action remains fixed, this implementation does not automatically satisfy Assumption \ref{assum:occupation-matching}. Its consistency with the continuous-time representation therefore requires convergence of the relevant occupation integrals and discrete approximations of the stochastic and ordinary time integrals as the mesh size tends to zero.

\section{Actor-Critic Algorithms}\label{sec:Actor-Critic-Algorithms}

We now construct the actor-critic algorithms from the policy-evaluation and policy-gradient results established in the preceding sections. For a given parameterized policy $\pi^\phi$, the critic approximates the corresponding policy value, while the actor uses the resulting value approximation to construct a policy-gradient estimate. Throughout this section, the actor parameter $\phi$ and the critic parameter $\theta$ are treated as separate variables.

\subsection{Offline regime-switching Actor-Critic}

\textit{Continuous-Time Representation.} We first establish the continuous-time offline actor-critic algorithms. Define {\it the regularized instantaneous reward} by
\begin{equation}\label{eq:the-regularized-instantaneous-reward}
\ell^\phi(t,x,i,a) = r(t,x,i,a) + \gamma p( t,x,i,a, \pi^\phi(\cdot\mid t,x,i) ).
\end{equation}
Let $J^\phi(t, x, i):=J(t, x, i ; \pi^\phi)$
denote the value function of the current policy $\pi^\phi$.
We assume that the critic class is {\it realizable}: for each fixed $\phi$, there exists $\theta^*\in\Theta$ such that
$$
J^{\theta^*}(t,x,i)=J^\phi(t, x,i), \quad (t, x, i) \in[0, T] \times \mathbb{R}^d \times \mathcal{M}.
$$
As in the preceding policy-gradient theorem, we distinguish the relaxed state-regime process $(\widetilde{X}^{\pi^\phi}, \alpha)$,  from the sampled-action controlled realization $(X, \alpha, a)$. The latter satisfies the Assumption \ref{assum:occupation-matching} established above, which provides the connection between policy-averaged expectations and expectations over sampled trajectories.

\textbf{Critic update.} We derive a trajectory-based critic estimating equation by connecting the policy-evaluation equation to the TD increment observed along sampled-action trajectories. The main idea is to establish a zero-moment condition for the exact critic, even though the sampled-action TD increment is not necessarily a martingale.

Recall that the critic class is realizable, so that $J^{\theta^*}(t, x, i)=J^\phi(t, x, i)$, the policy-evaluation equation \eqref{eq:policy-evaluation-equ-partial-phi} implies
\begin{equation}\label{eq:PE-for-J-theta(phi)}
0=\int_{\mathcal{A}}  \left[(\mathcal{L}^a J^{\theta^*})(t, x, i)+\ell^\phi(t, x, i, a)-\beta J^{\theta^*}(t, x, i)\right] \pi^\phi(a \mid t, x, i) \lambda(d a).
\end{equation}
As established in Section \ref{sec:PE},  this equation, implies that
\begin{equation*}
d J^\phi(t, \widetilde{X}_t^{\pi^\phi}, \alpha_t) +\left[\mathcal{R}(t, \widetilde{X}_t^{\pi^\phi}, \alpha_{t-}, \pi^\phi)-\beta J^\phi(t, \widetilde{X}_t^{\pi^\phi}, \alpha_{t-})\right] d t
\end{equation*}
is a martingale increment under the relaxed state-regime process $(\widetilde{X}_t^{\pi^\phi}, \alpha)$.
However, the corresponding property does not necessarily hold along sampled-action trajectories, as shown by the following result.

\begin{lemma}
Let $(X, \alpha, a)$ be the sampled-action controlled realization under the current policy $\pi^\phi$. Define the continuous-time TD increment along this trajectory by
\begin{equation}\label{eq:the-TD-increment-along-sampled-trajectory}
d\delta_t^{\theta,\phi} := dJ^\theta(t,X_t,\alpha_t) + \left[ \ell^\phi(t,X_t,\alpha_{t-},a_t) - \beta J^\theta(t,X_t,\alpha_{t-}) \right]dt.
\end{equation}
Then $d \delta_t^{\theta, \phi}$ admits the decomposition
\begin{equation}\label{eq:continuous-time-semimartingale-increment}
d \delta_t^{\theta, \phi}=\left[(\mathcal{L}^{a_t} J^\theta)(t, X_t, \alpha_{t-})+\ell^\phi(t, X_t, \alpha_{t-}, a_t)-\beta J^\theta(t, X_t, \alpha_{t-})\right] d t+d \mathcal{N}_t^{\theta, \phi},
\end{equation}
where $\mathcal{N}^{\theta, \phi}$ is the local martingale term arising from the Brownian motion and compensated regime transitions.
\end{lemma}
The result follows directly by applying the generalized Itô formula to $J^\theta(t, X_t, \alpha_t)$ and substituting the resulting decomposition into the definition of $d \delta_t^{\theta, \phi}$.
\begin{remark}
An important distinction from the relaxed formulation is that $d \delta_t^{\theta^*, \phi}$ need not be a martingale increment along the sampled-action trajectory, even when the critic is exact. Indeed, \eqref{eq:PE-for-J-theta(phi)} requires the PE residual to vanish only after averaging over $\pi^\phi$, whereas the drift in \eqref{eq:continuous-time-semimartingale-increment} depends on the sampled action $a_t$. In particular, the identity
$$
\int_{\mathcal{A}}\left[(\mathcal{L}^a J^{\theta^*})(t, x, i)+\ell^\phi(t, x, i, a)-\beta J^{\theta^*}(t, x, i)\right] \pi^\phi(a \mid t, x, i) \lambda(d a)=0,
$$
and does not imply that
\begin{equation*}
(\mathcal{L}^a J^{\theta^*})(t, x, i)+\ell^\phi(t, x, i, a)-\beta J^{\theta^*}(t, x, i)=0,
\end{equation*}
for every $a \in \mathcal{A}$. Thus, the finite-variation term in \eqref{eq:continuous-time-semimartingale-increment} may remain nonzero along an individual sampled-action trajectory. \textbf{The critic moment condition must therefore be established through policy averaging, rather than by directly treating the sampled-action TD increment as a martingale.} This distinction motivates the use of policy averaging to establish the critic moment condition.
\end{remark}
\begin{proposition}\label{pro:Trajectory-based-critic moment-condition}
Suppose that Assumption \ref{assum:occupation-matching} holds and that the corresponding predictability and integrability conditions are satisfied. If $J^{\theta^*}(\cdot,\cdot, \cdot)=J^\phi(\cdot,\cdot, \cdot)$, then the sampled-action TD increment satisfies the zero-moment condition when tested against the critic parameter gradient.
\end{proposition}
\begin{proof}
We first establish the zero-moment condition for the finite-variation component in \eqref{eq:continuous-time-semimartingale-increment}. Since $\nabla_\theta J^{\theta^*}(t, x, i)$ is independent of the current action \(a\), multiplying \eqref{eq:PE-for-J-theta(phi)} by this gradient gives
$$
0=\int_{\mathcal{A}} \nabla_\theta J^{\theta^*}(t, x, i)\left[(\mathcal{L}^a J^{\theta^*})(t, x, i)+\ell^\phi(t, x, i, a)-\beta J^{\theta^*}(t, x, i)\right] \pi^\phi(a \mid t, x, i) \lambda(d a) .
$$
The Assumption \ref{assum:occupation-matching} allows us to express the expected time integral of this policy-averaged identity in terms of the sampled-action trajectory. Hence,
\begin{equation}\label{eq:the-finite-variation}
0=\mathbb{E}_{0, x, i}^\phi \int_0^T \nabla_\theta J^{\theta^*}(t, X_t, \alpha_{t-})\left[(\mathcal{L}^{a_t} J^{\theta^*})(t, X_t, \alpha_{t-})+\ell^\phi(t, X_t, \alpha_{t-}, a_t)-\beta J^{\theta^*}(t, X_t, \alpha_{t-})\right] d t.
\end{equation}
Under the assumed filtration and integrability conditions, $\nabla_\theta J^{\theta^*}(t, X_t, \alpha_{t-})$ is predictable, and therefore
\begin{equation}\label{eq:the-martingale-part}
\mathbb{E}_{0, x, i}^\phi \int_0^T \nabla_\theta J^{\theta^*}(t, X_t, \alpha_{t-}) d \mathcal{N}_t^{\theta, \phi}=0.
\end{equation}
Combining \eqref{eq:continuous-time-semimartingale-increment}-\eqref{eq:the-martingale-part}, we obtain the trajectory-based critic moment condition
\begin{equation}\label{eq:trajectory-based-critic moment-condition}
\mathbb{E}_{0, x, i}^\phi\left[\int_0^T \nabla_\theta J^{\theta^*}(t, X_t, \alpha_{t-}) d \delta_t^{\theta^*, \phi}\right]=0.
\end{equation}
This establishes the desired trajectory-based critic moment condition. In particular, although $d \delta_t^{\theta^*, \phi}$ is generally not a martingale increment, its integral against the critic parameter gradient has zero expectation. This follows from Assumption \ref{assum:occupation-matching} for the finite-variation component and the zero-expectation property of the stochastic integral.
\end{proof}
Proposition \ref{pro:Trajectory-based-critic moment-condition} establishes the connection between the theoretical PE equation and the moment condition evaluated along sampled-action trajectories. Since the exact parameter $\theta^*$ is unknown, we extend \eqref{eq:trajectory-based-critic moment-condition} to a general critic parameter $\theta$ and consider the estimating equation
\begin{equation}\label{eq:estimating-moment-equation}
\mathbb{E}_{0, x, i}^\phi\left[\int_0^T \nabla_\theta J^\theta(t, X_t, \alpha_{t-}) d \delta_t^{\theta, \phi}\right]=0.
\end{equation}
When the critic class is realizable, the exact parameter $\theta^*$ satisfies \eqref{eq:estimating-moment-equation}. Nevertheless, the moment equation may admit multiple solutions unless the critic class and the chosen test process satisfy suitable identification conditions. If the critic class is misspecified, the equation instead characterizes a moment-matching approximation, provided that a solution exists.

Accordingly, \eqref{eq:estimating-moment-equation} motivates the following trajectory-based critic estimating direction:
$$
\Delta \theta=\int_0^T \nabla_\theta J^\theta\left(t, X_t, \alpha_{t-}\right) d \delta_t^{\theta, \phi}.
$$
Its expectation coincides with the left-hand side of \eqref{eq:estimating-moment-equation}
The critic can therefore be updated through stochastic approximation to seek a root of \eqref{eq:estimating-moment-equation}.

The resulting estimating direction depends only on observed value increments, regularized rewards, and the gradients of the critic parameter, without requiring explicit evaluation of the environmental drift, diffusion coefficients, or regime-switching generator. In particular, the effects of regime transitions are incorporated through the observed jumps in $J^\theta(t, X_t, \alpha_t)$. In general, $\Delta \theta$ is a stochastic estimating direction rather than the gradient of a specified scalar loss function.

\textbf{Actor update.} While the critic seeks a solution to the policy-evaluation moment equation, the actor aims to estimate the gradient of the policy objective. We first establish a trajectory-based representation of this gradient using the exact critic and then examine the effect of critic approximation. For a parameterized critic $J^\theta(t, x, i)$, define the continuous-time actor increment by
$$
d\mathcal G_t^{\theta,\phi} := e^{-\beta t} \left[ \nabla_\phi \log\pi^\phi(a_t\mid t,X_t,\alpha_{t-}) \,d\delta_t^{\theta,\phi} + \gamma \nabla_\phi p( t,X_t,\alpha_{t-},a_t, \pi^\phi(\cdot\mid t,X_t,\alpha_{t-}))dt \right].
$$
The following proposition establishes the connection between this increment and the policy gradient.
\begin{proposition}\label{pro:Trajectory-based-policy-gradient-representation}
Suppose that Assumption \ref{assum:occupation-matching} and the conditions of the preceding policy-gradient theorem hold. Assume further that the sampled-action realization preserves the required martingale properties and that the policy score is predictable and satisfies the corresponding stochastic-integral integrability conditions. If the critic is exact, namely,
$$
J^{\theta^*}(t, x, i)=J^\phi(t, x, i)=J(t, x, i ; \pi^\phi) .
$$
then
\begin{equation*}
\mathbb{E}_{0, x, i}^\phi\left[\int_0^T d \mathcal{G}_t^{\theta^*, \phi}\right]=\nabla_\phi J\left(0, x, i ; \pi^\phi\right)
\end{equation*}

\end{proposition}

\begin{proof}
By the generalized Itô formula, the corresponding TD increment satisfies
$$
d \delta_t^{\theta^*, \phi}=\bigg[(\mathcal{L}^{a_t} J^\phi)(t, X_t, \alpha_{t-})+\ell^\phi(t, X_t, \alpha_{t-}, a_t)
-\beta J^\phi(t, X_t, \alpha_{t-})\bigg] d t+d \mathcal{N}_t^{J^\phi}.
$$
where \(\mathcal N^{\theta^*,\phi}\) contains the Brownian and compensated regime-switching martingale components. Since the policy score is predictable and satisfies the required integrability conditions, the stochastic integral has zero expectation:
$$
\mathbb{E}_{0, x, i}^\phi\left[\int_0^T e^{-\beta t} \nabla_\phi \log \pi^\phi\left(a_t \mid t, X_t, \alpha_{t-}\right) d \mathcal{N}_t^{\theta^*}\right]=0.
$$
Consequently, substituting the above decomposition into the definition of $d \mathcal{G}_t^{\theta^*, \phi}$ and taking expectations eliminates the martingale term. The remaining score-weighted drift, together with the regularizer derivative, is precisely the auxiliary reward $\check{r}$ defined in the preceding policy-gradient theorem. Hence,
$$
\mathbb{E}_{0, x, i}^\phi\left[\int_0^T d \mathcal{G}_t^{\theta^*, \phi}\right]=\mathbb{E}_{0, x, i}^\phi \int_0^T e^{-\beta t} \check{r}(t, X_t, \alpha_{t-}, a_t ; \phi) d t .
$$
By Assumption \ref{assum:occupation-matching}, the expected time integral can be expressed in terms of the relaxed state-regime process:
\begin{align*}
& \mathbb{E}_{0, x, i}^\phi \int_0^T e^{-\beta t} \check{r}(t, X_t, \alpha_{t-}, a_t ; \phi) d t \\
& =\mathbb{E}_{0, x, i}^{\pi^\phi} \int_0^T e^{-\beta t} \int_{\mathcal{A}} \check{r}(t, \widetilde{X}_t^{\pi^\phi}, \alpha_{t-}, a ; \phi) \pi^\phi(a \mid t, \widetilde{X}_t^{\pi^\phi}, \alpha_{t-}) \lambda(d a) d t \\
& =\nabla_\phi J(0, x, i ; \pi^\phi),
\end{align*}
where the last equality follows from the preceding policy-gradient theorem.
\end{proof}
The proposition shows that, when the critic is exact, the policy gradient can be estimated directly from observed state-regime trajectories, regularized rewards, and policy scores. In particular, the regime-switching generator does not need to be evaluated explicitly: its contribution is already incorporated through the observed regime transitions in $d J^\phi(t, X_t, \alpha_t)$.

Motivated by this representation, we define the trajectory-based actor direction for a general critic parameter $\theta$ by
$$
\Delta \phi=\int_0^T d \mathcal{G}_t^{\theta, \phi}.
$$
Under the reward-maximization convention, this direction can be used to construct a stochastic policy-gradient ascent update. The preceding proposition guarantees that
$$
\mathbb{E}_{0, x, i}^\phi[\Delta \phi]=\nabla_\phi J(0, x, i ; \pi^\phi), \quad \theta=\theta^* .
$$
For an approximate critic, however, the expected actor direction may differ from the true policy gradient. The following result characterizes this discrepancy.

\begin{corollary}[Effect of critic approximation on the actor direction]
Under the conditions of the preceding proposition, let $J^\theta$ be a sufficiently regular critic approximation. Then the bias of the trajectory-based actor direction is given by
\begin{align*}
 \mathbb{E}_{0, x, i}^\phi[\Delta \phi]-\nabla_\phi J(0, x, i ; \pi^\phi)
& =\mathbb{E}_{0, x, i}^\phi \int_0^T e^{-\beta t} \nabla_\phi \log \pi^\phi(a_t \mid t, X_t, \alpha_{t-}) \\
& \quad \times\left[d(J^\theta-J^\phi)(t, X_t, \alpha_t)-\beta(J^\theta-J^\phi)(t, X_t, \alpha_{t-}) d t\right].
\end{align*}
\end{corollary}
\begin{proof}
Subtracting the actor increment associated with the exact critic $J^\phi$ from that associated with $J^\theta$, the regularizer derivative terms cancel. The difference between the two TD increments is
\begin{equation*}
d \delta_t^{\theta, \phi}-d \delta_t^{\theta^*, \phi}= d(J^\theta-J^\phi)(t, X_t, \alpha_t) -\beta(J^\theta-J^\phi)(t, X_t, \alpha_{t-}) d t .
\end{equation*}
Integrating the resulting difference between the actor increments, taking expectations, and applying the preceding proposition yields the stated identity.
\end{proof}
This corollary shows that the bias of the actor direction is determined by the contribution of the critic approximation error to the score-weighted TD moment. In particular, the bias vanishes when this contribution has zero expectation, which holds whenever $J^\theta(\cdot, \cdot, \cdot) = J^\phi(\cdot, \cdot, \cdot)$. Thus, an approximate critic may introduce bias into the actor update, although an exact critic is sufficient rather than necessary for obtaining an unbiased policy-gradient direction.

The distinction between the critic and actor can now be summarized through their respective estimating equations. At the exact critic, the policy-evaluation condition gives
$$
\mathbb{E}_{0, x, i}^\phi\left[\int_0^T \nabla_\theta J^{\theta^*}\left(t, X_t, \alpha_{t-}\right) d \delta_t^{\theta^*, \phi}\right]=0,
$$
whereas the actor satisfies
$$
\mathbb{E}_{0, x, i}^\phi\left[\int_0^T d \mathcal{G}_t^{\theta^*, \phi}\right]=\nabla_\phi J(0, x, i ; \pi^\phi) .
$$
Hence, the critic seeks a solution to the policy-evaluation moment equation, while the actor estimates the gradient of the policy objective. Unlike the critic direction, the expected actor direction is generally nonzero and vanishes at an interior stationary policy $\phi^*$ satisfying
$$
\nabla_\phi J(0, x, i ; \pi^{\phi^*})=0 .
$$
Finally, these results characterize continuous-time actor-critic directions under the stated sampled-action realization and occupation-matching conditions. Their finite-grid implementation requires a separate consistency analysis. In particular, sampling an action at each decision time and holding it fixed over the subsequent interval does not automatically preserve the continuous-time PE and PG identities. The corresponding discrete actor-critic updates are therefore connected to the above representations through the occupation and time-discretization consistency results established in the subsequent analysis.

\textit{Finite-Grid Approximation.} To implement the preceding continuous-time actor-critic framework, we approximate the state-regime-action trajectories on a finite time grid and construct discrete counterparts of the critic and actor estimating directions. The main objective is to establish the consistency of these discrete directions with the continuous-time PE moment condition and policy-gradient representation as the mesh size tends to zero. Consider a time partition
$$
\mathcal{P}=\{0=t_0<t_1<\cdots<t_N=T\}, \qquad h_k=t_{k+1}-t_k,
$$
with mesh size $|\mathcal{P}|:=\max _{0 \leq k<N} h_k$.
At each decision time $t_k$, the state-regime pair $(X_{t_k},\alpha_{t_k})$ is observed before an action is sampled. Let $\mathcal{H}_k$ denote the information available immediately before this sampling. The action is generated according to
$$
a_k\mid\mathcal{H}_k \sim \pi^\phi(\cdot\mid t_k,X_{t_k},\alpha_{t_k}),
$$
and is held fixed over the interval $(t_k, t_{k+1}]$. Accordingly, the sample-and-hold action process is defined by
\begin{equation}\label{eq:the-action-process}
a_t=a_k, \quad t \in(t_k, t_{k+1}].
\end{equation}
We assume that the sampled actions are included in the filtration at their respective decision times and that the action randomization preserves the Brownian martingale property and the prescribed compensators of the regime transitions. Under this construction, the action process in \eqref{eq:the-action-process} is predictable in the enlarged filtration.

The sample-and-hold scheme differs from continuous-time policy randomization because actions are updated only at the decision times. In particular, a state change or regime transition within $(t_k,t_{k+1}]$ does not trigger a new action. Consequently, the finite-grid trajectory need not satisfy the continuous-time occupation-matching condition exactly. This distinction will be taken into account in the consistency analysis below.

Throughout this subsection, $\mathbb{E}_{0, x, i}^{\phi, \mathcal{P}}$ denotes expectation under the finite-grid sampled-action dynamics starting from $(x, i)$. The policy parameters are held fixed when analyzing the consistency of the finite-grid approximation.

We first construct the discrete TD increment. Let $\widehat{\ell}_k$ be an observed estimate of the regularized instantaneous reward satisfying
$$
\mathbb{E}^{\phi, \mathcal{P}} \left[ \widehat{\ell}_k \mid \mathcal{H}_k,a_k \right] = r(t_k,X_{t_k},\alpha_{t_k},a_k) + \gamma p( t_k,X_{t_k},\alpha_{t_k},a_k, \pi^\phi(\cdot\mid t_k,X_{t_k},\alpha_{t_k}) ).
$$
Thus, $\hat{\ell}_k$ is a conditionally unbiased estimate of $\ell^\phi(t_k, X_{t_k}, \alpha_{t_k}, a_k)$.

Integrating the continuous-time sampled-action TD increment \eqref{eq:the-TD-increment-along-sampled-trajectory} over $[t_k, t_{k+1}]$ yields
\begin{align}\label{eq: first-orde-approximation-for-k-interval}
\delta_{\left[t_k, t_{k+1}\right]}^{\theta, \phi} :=& J^\theta(t_{k+1},X_{t_{k+1}},\alpha_{t_{k+1}}) - J^\theta(t_k,X_{t_k},\alpha_{t_k})\notag\\
&\quad+\int_{t_k}^{t_{k+1}}[\ell^\phi(t, X_t, \alpha_{t-}, a_t)-\beta J^\theta(t, X_t, \alpha_{t-})] d t.
\end{align}

Since the value-function increment in \eqref{eq: first-orde-approximation-for-k-interval} can be evaluated directly from consecutive observations, we retain this increment and approximate only the ordinary time integral using its left-endpoint value. This yields
\begin{equation}\label{eq:the-discrete-TD-increment}
\delta_{k}^{\theta, \phi} := J^\theta(t_{k+1},X_{t_{k+1}},\alpha_{t_{k+1}}) - J^\theta(t_k,X_{t_k},\alpha_{t_k})+h_k\left[\hat{\ell}_k-\beta J^\theta(t_k, X_{t_k}, \alpha_{t_k})\right] .
\end{equation}
The approximation in \eqref{eq:the-discrete-TD-increment} is applied only to the ordinary time integral; the observed value-function increment is retained without further approximation. In particular, the effects of regime transitions are incorporated through the change from $\alpha_{t_k}$ to $\alpha_{t_{k+1}}$. Hence, no additional switching-generator correction is required in the discrete TD increment.

We next construct the corresponding discrete actor increment. Approximating the continuous-time score-weighted TD integral by its left-endpoint value and using \eqref{eq:the-discrete-TD-increment}, define
\begin{equation}\label{eq:finite-grid-policy-gradient-increment}
\mathcal{G}_k^{\theta,\phi} := e^{-\beta t_k} \Bigg[\nabla_\phi \log \pi^\phi (a_k\mid t_k,X_{t_k},\alpha_{t_k}) \,\delta_k^{\theta,\phi} + \gamma h_k \nabla_\phi p\left( t_k,X_{t_k},\alpha_{t_k},a_k, \pi^\phi(\cdot\mid t_k,X_{t_k},\alpha_{t_k}) \right) \Bigg].
\end{equation}
The finite-grid critic and actor directions are given by
\begin{equation}\label{eq:finite-grid-approximations-of-PE-and-pG-increment}
\Delta\theta_{\mathcal{P}} = \sum_{k=0}^{N-1} \nabla_\theta J^\theta(t_k,X_{t_k},\alpha_{t_k}) \,\delta_k^{\theta,\phi}, \quad \Delta\phi_{\mathcal{P}} = \sum_{k=0}^{N-1} \mathcal{G}_k^{\theta,\phi}.
\end{equation}
These directions provide finite-grid approximations of the continuous-time quantities
\begin{equation*}
\int_0^T \nabla_\theta J^\theta(t, X_t, \alpha_{t-}) d \delta_t^{\theta, \phi} \quad
\text{ and } \quad
\int_0^T d \mathcal{G}_t^{\theta, \phi},
\end{equation*}
respectively. Both are constructed using the same discrete TD increment $\delta_k^{\theta, \phi}$ increment \eqref{eq:the-discrete-TD-increment}, but they serve different purposes: $\Delta\theta_{\mathcal{P}}$ is used to solve the PE moment equation, whereas $\Delta\phi_{\mathcal{P}}$ is used to estimate the policy gradient.

For $M$ independently simulated trajectories, the corresponding empirical estimating directions are obtained by averaging \eqref{eq:finite-grid-approximations-of-PE-and-pG-increment} across trajectories:
\begin{equation}\label{eq:empirical-estimate-directions}
\widehat{\Delta \theta}_{M, \mathcal{P}}=\frac{1}{M} \sum_{n=1}^M \Delta \theta_{\mathcal{P}}^{(n)}, \quad \widehat{\Delta \phi}_{M, \mathcal{P}}=\frac{1}{M} \sum_{n=1}^M \Delta \phi_{\mathcal{P}}^{(n)} .
\end{equation}
Thus, time discretization and Monte Carlo sampling introduce two distinct sources of approximation error. The following analysis focuses first on the time-discretization error in the expected directions.

We now examine how \eqref{eq:the-discrete-TD-increment} approximates the continuous-time TD decomposition. The key step is to relate the conditional expectation of the discrete TD increment to the PE residual evaluated at the beginning of each time interval.

For notational convenience, define the action-dependent PE residual by
\begin{equation}\label{eq:the-action-dependent-PE-residual}
H^{\theta, \phi}(t, x, i, a):= (\mathcal{L}^a J^\theta)(t, x, i) +\ell^\phi(t, x, i, a)-\beta J^\theta(t, x, i) .
\end{equation}
This quantity is the drift coefficient in the sampled-action TD semimartingale decomposition established in \eqref{eq:continuous-time-semimartingale-increment}. In particular, for the exact critic $J^{\theta^*}=J^\phi$, the PE equation \eqref{eq:PE-for-J-theta(phi)} implies
\begin{equation}\label{eq:averaging-PE-equation}
\int_{\mathcal{A}} H^{\theta^*, \phi}(t, x, i, a) \pi^\phi(d a \mid t, x, i)=0 .
\end{equation}
Notice that \eqref{eq:averaging-PE-equation} holds after averaging over the current policy and does not require $H^{\theta^*,\phi}(t,x,i,a)$ to vanish for each individual action. This distinction remains essential in the finite-grid analysis.

On each interval $(t_k,t_{k+1}]$, the state process evolves according to the controlled dynamics
\begin{equation}\label{eq:sample-and-hold-dynamic}
dX_t=b(t,X_t,\alpha_{t-},a_k),dt+\sigma(t,X_t,\alpha_{t-},a_k)dW_t,
\end{equation}
where the sampled action $a_k$ remains fixed while the state and regime continue to evolve.
Applying the generalized Itô formula to $J^\theta(t,X_t,\alpha_t)$ over $[t_k,t_{k+1}]$ and and taking conditional expectations given $(\mathcal{H}_k,a_k)$, we obtain
\begin{align*}
&\mathbb{E}^{\phi, \mathcal{P}}\left[J^\theta(t_{k+1}, X_{t_{k+1}}, \alpha_{t_{k+1}})-J^\theta(t_k, X_{t_k}, \alpha_{t_k}) \mid \mathcal{H}_k, a_k\right]\\
&\quad=\mathbb{E}^{\phi, \mathcal{P}}\left[\int_{t_k}^{t_{k+1}}(\mathcal{L}^{a_k} J^\theta)(t, X_t, \alpha_{t-}) d t \mid \mathcal{H}_k, a_k\right] .
\end{align*}
Here the conditional expectation of the Brownian and compensated regime-switching martingale increments is zero. Combining this identity with the conditional unbiasedness of $\widehat{\ell}_k$, we obtain the local consistency relation
\begin{equation}\label{eq:the-local-consistency-relation}
\mathbb{E}^{\phi, \mathcal{P}}\left[\delta_k^{\theta, \phi} \mid \mathcal{H}_k, a_k\right]=h_k H^{\theta, \phi}\left(t_k, X_{t_k}, \alpha_{t_k}, a_k\right)+\varepsilon_{k, \mathcal{P}}^{\theta, \phi},
\end{equation}
where $\varepsilon_{k,\mathcal{P}}^{\theta,\phi}$ denotes the local discretization remainder.
For the exact sample-and-hold dynamics defined in \eqref{eq:sample-and-hold-dynamic}, this remainder can be expressed as
\begin{equation*}
\varepsilon_{k, \mathcal{P}}^{\theta, \phi}=\mathbb{E}^{\phi, \mathcal{P}}\left[\int_{t_k}^{t_{k+1}} \left\{(\mathcal{L}^{a_k} J^\theta)(t, X_t, \alpha_{t-})-(\mathcal{L}^{a_k} J^\theta)(t_k, X_{t_k}, \alpha_{t_k})\right\} d t \mid \mathcal{H}_k, a_k\right]
\end{equation*}
If the state-regime trajectory is generated by a numerical approximation of the controlled dynamics, the remainder in \eqref{eq:the-local-consistency-relation} also includes the corresponding conditional weak approximation error.

Equation \eqref{eq:the-local-consistency-relation} provides the key connection between discrete TD learning and continuous-time PE. It shows that the conditional expectation of the discrete TD increment is approximated by the local PE residual multiplied by the time step, with the remaining discrepancy represented by $\varepsilon_{k,\mathcal{P}}^{\theta,\phi}$.
In particular, this connection is established without discarding the regime-switching contribution. The operator $\mathcal{L}^{a_k}$ in \eqref{eq:averaging-PE-equation} contains the switching term
$\sum_{j=1}^Lq_{ij}J^\theta(t,x,j)$,
whose effect is captured in the conditional expectation of the observed value-function increment. The generator is therefore used in the consistency analysis, although it does not need to be explicitly evaluated in the TD algorithm.

The local relation \eqref{eq:the-local-consistency-relation} allows us to derive the limiting properties of the discrete critic and actor directions. %For the critic, the relevant condition is the vanishing policy-averaged PE residual in \eqref{eq:averaging-PE-equation}. For the actor, the score-weighted residual must converge to the auxiliary reward appearing in the preceding policy-gradient theorem. These two arguments are formalized in the following proposition.
\begin{proposition}[Consistency of finite-grid Actor-Critic directions]
Fix an admissible policy $\pi^\phi$ and suppose that the critic class is realizable, with $J^{\theta^*}=J^\phi$. Let the finite-grid trajectories be generated by \eqref{eq:the-action-process}, and suppose that the following conditions hold as $|\mathcal{P}|\to0$.
\begin{enumerate}
    \item The sampled-action dynamics and reward estimates satisfy the local consistency relation \eqref{eq:the-local-consistency-relation}, and the underlying filtration preserves the required environmental martingales and switching compensators.

    \item The accumulated discretization remainders are negligible when weighted by the critic parameter gradient and policy score. More precisely,
    \begin{equation}\label{eq:corvingence-of-critic parameter-gradient-and-policy}
         \sum_{k=0}^{N-1} \mathbb{E}^{\phi, \mathcal{P}}\left[\left(\|\nabla_\theta J^{\theta^*}(t_k, X_{t_k}, \alpha_{t_k})\|+e^{-\beta t_k}\|\nabla_\phi \log \pi^\phi(a_k \mid t_k, X_{t_k}, \alpha_{t_k})\|\right)|\varepsilon_{k, \mathcal{P}}^{\theta^*, \phi}|\right] \rightarrow 0 .
    \end{equation}

    \item The sample-and-hold trajectories satisfy the occupation consistency required for the auxiliary reward $\check{r}$. In particular,
    \begin{align}\label{eq:convergence-of auxiliary}
        & \mathbb{E}_{0, x, i}^{\phi, \mathcal{P}} \sum_{k=0}^{N-1} h_k e^{-\beta t_k} \int_{\mathcal{A}} \check{r}\left(t_k, X_{t_k}, \alpha_{t_k}, a ; \phi\right) \pi^\phi\left(d a \mid t_k, X_{t_k}, \alpha_{t_k}\right) \notag\\
        & \quad \rightarrow \mathbb{E}_{0, x, i}^{\pi^\phi} \int_0^T e^{-\beta t} \int_{\mathcal{A}} \check{r}(t, \widetilde{X}_t^{\pi^\phi}, \alpha_{t-}, a ; \phi) \pi^\phi(d a \mid t, \widetilde{X}_t^{\pi^\phi}, \alpha_{t-}) d t.
    \end{align}
    This convergence is understood under the required regularity, moment bounds, and uniform integrability conditions for the corresponding occupation integrals.
    \end{enumerate}
     Then the finite-grid critic and actor directions defined in \eqref{eq:finite-grid-approximations-of-PE-and-pG-increment} satisfy
\begin{equation}\label{eq:convergence-of-finite-grid-critic-actor-directions}
\lim _{|\mathcal{P}| \rightarrow 0} \mathbb{E}_{0, x, i}^{\phi, \mathcal{P}}\left[\Delta \theta_{\mathcal{P}}\right]=0, \quad
\lim _{|\mathcal{P}| \rightarrow 0} \mathbb{E}_{0, x, i}^{\phi, \mathcal{P}}\left[\Delta \phi_{\mathcal{P}}\right] =\nabla_\phi J\left(0, x, i ; \pi^\phi\right) .
\end{equation}
\end{proposition}
\begin{proof}
 We establish the two limits separately.

 \textbf{Step 1. Consistency of the critic direction.} By the definition of $\Delta\theta_{\mathcal{P}}$ in \eqref{eq:finite-grid-approximations-of-PE-and-pG-increment}, the tower property of conditional expectation, and the local consistency relation \eqref{eq:the-local-consistency-relation}, we have
\begin{align*}
\mathbb{E}_{0, x, i}^{\phi, \mathcal{P}}\left[\Delta \theta_{\mathcal{P}}\right] & =\sum_{k=0}^{N-1} \mathbb{E}_{0, x, i}^{\phi, \mathcal{P}}\left[\nabla_\theta J^{\theta^*}\left(t_k, X_{t_k}, \alpha_{t_k}\right) \delta_k^{\theta^*, \phi}\right] \\
& =\sum_{k=0}^{N-1} \mathbb{E}_{0, x, i}^{\phi, \mathcal{P}}\left[\nabla_\theta J^{\theta^*}(t_k, X_{t_k}, \alpha_{t_k})\left\{h_k H^{\theta^*, \phi}(t_k, X_{t_k}, \alpha_{t_k}, a_k)+\varepsilon_{k, \mathcal{P}}^{\theta^*, \phi}\right\}\right] .
\end{align*}
Since the critic parameter gradient is measurable with respect to $\mathcal{H}_k$ and does not depend on the newly sampled action $a_k$, we can average the leading term over the conditional distribution of $a_k$. Using \eqref{eq:averaging-PE-equation}, this gives
\begin{align*}
& \mathbb{E}_{0, x, i}^{\phi, \mathcal{P}}\left[\nabla_\theta J^{\theta^*}(t_k, X_{t_k}, \alpha_{t_k}) H^{\theta^*, \phi}(t_k, X_{t_k}, \alpha_{t_k}, a_k)\right] \\
& =\mathbb{E}_{0, x, i}^{\phi, \mathcal{P}}\left[\nabla_\theta J^{\theta^*}(t_k, X_{t_k}, \alpha_{t_k}) \int_{\mathcal{A}} H^{\theta^*, \phi}(t_k, X_{t_k}, \alpha_{t_k}, a) \pi^\phi(d a \mid t_k, X_{t_k}, \alpha_{t_k})\right]=0 .
\end{align*}
Consequently, the expected critic direction reduces to the accumulated discretization remainder:
\begin{equation*}
\mathbb{E}_{0, x, i}^{\phi, \mathcal{P}}[\Delta \theta_{\mathcal{P}}]=\sum_{k=0}^{N-1} \mathbb{E}_{0, x, i}^{\phi, \mathcal{P}}\left[\nabla_\theta J^{\theta^*}(t_k, X_{t_k}, \alpha_{t_k}) \varepsilon_{k, \mathcal{P}}^{\theta^*, \phi}\right] .
\end{equation*}
By the weighted remainder condition \eqref{eq:corvingence-of-critic parameter-gradient-and-policy}, the right-hand side tends to zero as $|\mathcal{P}| \rightarrow 0$. Therefore,
$$
\lim _{|\mathcal{P}| \rightarrow 0} \mathbb{E}_{0, x, i}^{\phi, \mathcal{P}}\left[\Delta \theta_{\mathcal{P}}\right]=0 .
$$
This proves the first assertion of \eqref{eq:convergence-of-finite-grid-critic-actor-directions}. In particular, the result follows from the PE equation and local consistency of the discrete TD increment; it does not require the sampled-action TD increment itself to be a martingale.

\textbf{Step 2. Consistency of the actor direction.}
We next consider $\Delta \phi_{\mathcal{P}}$. At the exact critic $\theta^*$, substituting \eqref{eq:finite-grid-policy-gradient-increment} into \eqref{eq:finite-grid-approximations-of-PE-and-pG-increment} gives
\begin{equation*}
\mathbb{E}_{0, x, i}^{\phi, \mathcal{P}}\left[\Delta \phi_{\mathcal{P}}\right]=\sum_{k=0}^{N-1} \mathbb{E}_{0, x, i}^{\phi, \mathcal{P}}\left[e^{-\beta t_k} \left\{\nabla_\phi \log \pi^\phi(a_k \mid t_k, X_{t_k}, \alpha_{t_k}) \delta_k^{\theta^*, \phi}+\gamma h_k \nabla_\phi p_k^\phi\right\}\right],
\end{equation*}
where $\nabla_\phi p_k^\phi$ denotes the policy-parameter derivative of the regularizer evaluated at $(t_k, X_{t_k}, \alpha_{t_k}, a_k)$.

Using the local consistency relation \eqref{eq:the-local-consistency-relation} and the fact that $\nabla_\phi \log \pi^\phi(a_k \mid t_k, X_{t_k}, \alpha_{t_k})$ is measurable after the action $a_k$ has been sampled, we obtain
\begin{align*}
\mathbb{E}_{0, x, i}^{\phi, \mathcal{P}}\left[\Delta \phi_{\mathcal{P}}\right]
& =\sum_{k=0}^{N-1} \mathbb{E}_{0, x, i}^{\phi, \mathcal{P}}\left[h_k e^{-\beta t_k}\left\{\nabla_\phi \log \pi^\phi(a_k \mid t_k, X_{t_k}, \alpha_{t_k}) H^{\theta^*, \phi}\left(t_k, X_{t_k}, \alpha_{t_k}, a_k\right)+\gamma \nabla_\phi p_k^\phi\right\}\right] \\
&\quad+\sum_{k=0}^{N-1} \mathbb{E}_{0, x, i}^{\phi, \mathcal{P}}\left[e^{-\beta t_k} \nabla_\phi \log \pi^\phi(a_k \mid t_k, X_{t_k}, \alpha_{t_k}) \varepsilon_{k, \mathcal{P}}^{\theta^*, \phi}\right] .
\end{align*}
Recall from the preceding policy-gradient theorem that the auxiliary reward satisfies
\begin{equation*}
\check{r}(t, x, i, a ; \phi)=  \nabla_\phi \log \pi^\phi(a \mid t, x, i) H^{\theta^*, \phi}(t, x, i, a) +\gamma \nabla_\phi p\left(t, x, i, a, \pi^\phi(\cdot \mid t, x, i)\right),
\end{equation*}
where $J^{\theta^*}=J^\phi$ has been used in the definition of $H^{\theta^*, \phi}$.
Therefore, averaging the first term in the preceding expression over $a_k$ yields
\begin{align}\label{eq:averaging-the-first-term-of-critic-direction}
 \mathbb{E}_{0, x, i}^{\phi, \mathcal{P}}\left[\Delta \phi_{\mathcal{P}}\right]
& =\mathbb{E}_{0, x, i}^{\phi, \mathcal{P}} \sum_{k=0}^{N-1} h_k e^{-\beta t_k} \int_{\mathcal{A}} \check{r}(t_k, X_{t_k}, \alpha_{t_k}, a ; \phi) \pi^\phi(d a \mid t_k, X_{t_k}, \alpha_{t_k}) \notag\\
&\quad +\sum_{k=0}^{N-1} \mathbb{E}_{0, x, i}^{\phi, \mathcal{P}}\left[e^{-\beta t_k} \nabla_\phi \log \pi^\phi(a_k \mid t_k, X_{t_k}, \alpha_{t_k}) \varepsilon_{k, \mathcal{P}}^{\theta^*, \phi}\right] .
\end{align}
Equation \eqref{eq:averaging-the-first-term-of-critic-direction} is the key step in the actor consistency argument. It shows that the expected discrete actor direction consists of a finite-grid approximation of the auxiliary reward integral in the policy-gradient theorem, together with an accumulated discretization error.

By \eqref{eq:corvingence-of-critic parameter-gradient-and-policy}, the second term on the right-hand side of \eqref{eq:averaging-the-first-term-of-critic-direction} tends to zero. Moreover, the occupation consistency condition \eqref{eq:convergence-of auxiliary} ensures that the first term converges to
$$
\mathbb{E}_{0, x, i}^{\pi^\phi} \int_0^T e^{-\beta t} \int_{\mathcal{A}} \check{r}(t, \widetilde{X}_t^{\pi^\phi}, \alpha_{t-}, a ; \phi) \pi^\phi(d a \mid t, \widetilde{X}_t^{\pi^\phi}, \alpha_{t-}) d t .
$$
By the continuous-time policy-gradient theorem established earlier, this expression equals $\nabla_\phi J(0, x, i ; \pi^\phi)$. Hence,
$$
\lim _{|\mathcal{P}| \rightarrow 0} \mathbb{E}_{0, x, i}^{\phi, \mathcal{P}}\left[\Delta \phi_{\mathcal{P}}\right]=\nabla_\phi J\left(0, x, i ; \pi^\phi\right) .
$$
This proves the second assertion of \eqref{eq:convergence-of-finite-grid-critic-actor-directions} and completes the proof.

\end{proof}
The proposition establishes that \textbf{the discrete Critic and Actor directions recover their respective continuous-time targets in the limit of a vanishing time mesh.} Notably, the two limits follow from different properties: the Critic limit uses the policy-averaged PE equation \eqref{eq:averaging-PE-equation}, whereas the Actor limit additionally uses the auxiliary reward representation and the policy-gradient theorem. Neither argument requires the sampled-action TD increment itself to be a martingale.

\textit{Batch Actor-Critic Construction.}
We now use the finite-grid directions in \eqref{eq:finite-grid-approximations-of-PE-and-pG-increment} to construct a batch Actor-Critic algorithm. The preceding proposition establishes their consistency with continuous-time PE and PG when the critic is exact and the time mesh vanishes. \textbf{The purpose of the batch construction is to estimate these directions from finitely many independent trajectories and update the critic and actor parameters through stochastic approximation.}

At iteration $m$, let $(\theta_m, \phi_m)$ denote the current critic and actor parameters, and let $\mathscr{F}_m$ contain all information collected before the current batch. Conditional on $\mathscr{F}_m$, we keep $(\theta_m, \phi_m)$ fixed and generate $M$ independent trajectories under the sample-and-hold implementation of $\pi^{\phi_m}$, each starting from the prescribed initial state-regime pair $(x, i)$.

For trajectory $n$ and time interval $[t_k, t_{k+1}]$, define the discrete TD increment, consistently with \eqref{eq:the-discrete-TD-increment}, by
\begin{align}\label{eq:the-betch-discrete-TD-increment}
\delta_{n,k}^{\theta_m,\phi_m} := &J^{\theta_m} (t_{k+1},X_{n,t_{k+1}},\alpha_{n,t_{k+1}}) - J^{\theta_m} (t_k,X_{n,t_k},\alpha_{n,t_k})\notag\\
&\quad+ h_k \left[ \widehat{\ell}_{n,k} - \beta J^{\theta_m} (t_k,X_{n,t_k},\alpha_{n,t_k}) \right].
\end{align}
Here, $\hat{\ell}_{n, k}$ is conditionally unbiased for the regularized instantaneous reward evaluated at $(t_k, X_{n, t_k}, \alpha_{n, t_k}, a_{n, k})$. The increment in \eqref{eq:the-betch-discrete-TD-increment} combines the observed change in the estimated value function with the regularized reward and discount contributions over the current interval.

For the actor, define
\begin{equation*}
P_{n,k}^{\phi_m} := \nabla_\phi p( t_k,X_{n,t_k},\alpha_{n,t_k},a_{n,k}, \pi^\phi(\cdot\mid t_k,X_{n,t_k},\alpha_{n,t_k}) ) |_{\phi=\phi_m}.
\end{equation*}
where the differentiation is taken with respect to the policy parameter while the observed state, regime, and action are held fixed. The corresponding actor increment is
\begin{equation*}
\mathcal{G}_{n,k}^{\theta_m,\phi_m} := e^{-\beta t_k} \bigg[  \nabla_\phi \log \pi^\phi (a_{n,k}\mid t_k,X_{n,t_k},\alpha_{n,t_k}) |_{\phi=\phi_m} \delta_{n,k}^{\theta_m,\phi_m} + \gamma h_k P_{n,k}^{\phi_m} \bigg].
\end{equation*}
Thus, the actor increment consists of the score-weighted TD increment and the explicit policy-parameter derivative of the regularizer, in accordance with the continuous-time policy-gradient representation.

Averaging the discrete directions over the $M$ trajectories gives
\begin{equation}\label{eq:batch-directions}
\Delta\theta_m = \frac{1}{M} \sum_{n=1}^{M} \sum_{k=0}^{N-1} \nabla_\theta J^\theta(t_k,X_{n,t_k},\alpha_{n,t_k})|_{\theta=\theta_m} \delta_{n,k}^{\theta_m,\phi_m},
\quad
\Delta\phi_m = \frac{1}{M} \sum_{n=1}^{M} \sum_{k=0}^{N-1} \mathcal{G}_{n,k}^{\theta_m,\phi_m}.
\end{equation}
Both directions are evaluated at the same pre-update parameter pair $(\theta_m,\phi_m)$. This is important because the Critic and Actor increments must correspond to the same current policy when applying the PE and PG representations.

The next proposition clarifies the effect of batch averaging. In particular, increasing the batch size reduces sampling variability, but does not remove systematic errors caused by time discretization or critic approximation.
\begin{proposition}\label{pro:Conditional-moments-of-AC-directions}
Suppose that, conditional on $\mathscr{F}_m$, the $M$ sampled trajectories are independent and identically distributed under the finite-grid implementation of $\pi^{\phi_m}$. Assume that the corresponding single-trajectory critic and actor directions have finite conditional second moments. Then
\begin{equation}\label{eq:second-moments-of-AC-directions}
\mathbb{E}[\Delta \theta_m \mid \mathscr{F}_m]  =f_{\mathcal{P}}(\theta_m, \phi_m), \quad
\mathbb{E}[\Delta \phi_m \mid \mathscr{F}_m] =g_{\mathcal{P}}(\theta_m, \phi_m),
\end{equation}
where
\begin{equation*}
f_{\mathcal{P}}(\theta, \phi):=\mathbb{E}_{0, x, i}^{\phi, \mathcal{P}}\left[\sum_{k=0}^{N-1} \nabla_\theta J^\theta(t_k, X_{t_k}, \alpha_{t_k}) \delta_k^{\theta, \phi}\right], \quad
g_{\mathcal{P}}(\theta, \phi):=\mathbb{E}_{0, x, i}^{\phi, \mathcal{P}}\left[\sum_{k=0}^{N-1} \mathcal{G}_k^{\theta, \phi}\right] .
\end{equation*}
Moreover, if the single-trajectory conditional second moments are bounded uniformly in $m$, the conditional covariance matrices of the batch directions satisfy
\begin{equation}\label{eq:conditional-covariance- of-batch-directions}
\operatorname{Cov}(\Delta \theta_m \mid \mathscr{F}_m)=O(M^{-1}), \quad \operatorname{Cov}(\Delta \phi_m \mid \mathscr{F}_m)=O(M^{-1}) .
\end{equation}
\end{proposition}

\begin{proof}
 Since the parameter pair $\left(\theta_m, \phi_m\right)$ is fixed conditional on $\mathscr{F}_m$, the batch directions in \eqref{eq:batch-directions} are averages of conditionally independent and identically distributed single-trajectory directions. Linearity of conditional expectation immediately yields \eqref{eq:second-moments-of-AC-directions}. Conditional independence also eliminates the cross-covariance terms between different trajectories, so the covariance of each batch average equals $M^{-1}$ times the corresponding single-trajectory conditional covariance matrix. The uniform second-moment bounds therefore imply \eqref{eq:conditional-covariance- of-batch-directions}.
\end{proof}
The proposition shows that batch averaging preserves the expected finite-grid estimating directions while reducing their Monte Carlo variance. However, the mean fields $f_{\mathcal{P}}$ and $g_{\mathcal{P}}$ generally differ from their continuous-time counterparts when the mesh size is nonzero.

In particular, if the critic is exact at the current policy, $J^{\theta_m}=J^{\phi_m}$, the finite-grid consistency result \eqref{eq:convergence-of-finite-grid-critic-actor-directions} gives, with $\left(\theta_m, \phi_m\right)$ held fixed,
\begin{equation*}
\mathbb{E}[\Delta \theta_m \mid \mathscr{F}_m] \rightarrow 0,  \quad
\mathbb{E}[\Delta \phi_m \mid \mathscr{F}_m] \rightarrow \nabla_\phi J(0, x, i ; \pi^{\phi_m}), \quad |P| \rightarrow 0 .
\end{equation*}
\textbf{Thus, the batch critic direction approaches the continuous-time PE zero-moment condition, whereas the batch actor direction approaches the true policy gradient.} These limits do not require $M \rightarrow \infty$, since they concern conditional expectations; increasing $M$ instead controls the variability of the corresponding sample estimates.

When the critic is approximate, the expected actor direction contains an additional error that is distinct from finite-grid approximation error. To make this distinction precise, let $\theta^*(\phi_m)$ denote an exact critic parameter for the current policy and write
$J^{\theta^*(\phi_m)}=J^{\phi_m} $.
Define the critic approximation error by
$$
E_m(t, x, i):=J^{\theta_m}(t, x, i)-J^{\phi_m}(t, x, i) .
$$
The discrepancy between the expected batch actor direction and the true policy gradient can then be decomposed as
\begin{equation*}
g_{\mathcal{P}} (\theta_m, \phi_m)-\nabla_\phi J(0, x, i ; \pi^{\phi_m})
=\underbrace{g_{\mathcal{P}}(\theta^*(\phi_m), \phi_m)-\nabla_\phi J(0, x, i ; \pi^{\phi_m})}_{\text {Finite-grid approximation bias }} +\underbrace{g_{\mathcal{P}}(\theta_m, \phi_m)-g_{\mathcal{P}}(\theta^*(\phi_m), \phi_m)}_{\text {Critic approximation bias }} .
\end{equation*}
This decomposition separates the error introduced by the finite time grid from that caused by replacing the exact value function with an approximate critic.

The second term can be expressed explicitly using the discrete TD increments. Indeed, subtracting the exact-critic increment from \eqref{eq:the-discrete-TD-increment} gives
\begin{equation*}
\delta_k^{\theta_m, \phi_m}-\delta_k^{\theta^*(\phi_m), \phi_m}= E_m(t_{k+1}, X_{t_{k+1}}, \alpha_{t_{k+1}}) -E_m(t_k, X_{t_k}, \alpha_{t_k}) -\beta h_k E_m(t_k, X_{t_k}, \alpha_{t_k}) .
\end{equation*}
Since the regularizer derivative is the same for both critics, it cancels when taking the difference between their actor increments. Hence,
\begin{align*}
&g_{\mathcal{P}}(\theta_m, \phi_m)-g_{\mathcal{P}}(\theta^*(\phi_m), \phi_m)\\
& =\mathbb{E}_{0, x, i}^{\phi_m, \mathcal{P}} \sum_{k=0}^{N-1} e^{-\beta t_k} \nabla_\phi \log \pi^{\phi_m}(a_k \mid t_k, X_{t_k}, \alpha_{t_k}) \\
& \quad \times\left[E_m(t_{k+1}, X_{t_{k+1}}, \alpha_{t_{k+1}})-E_m(t_k, X_{t_k}, \alpha_{t_k})-\beta h_k E_m(t_k, X_{t_k}, \alpha_{t_k})\right] .
\end{align*}
Therefore, the critic approximation affects the expected actor direction through the score-weighted TD increment of the value-function approximation error. In particular, this bias vanishes when the critic is exact, although exact representation is not necessary if the corresponding score-weighted error has zero expectation.

Combining this observation with Proposition \ref{pro:Conditional-moments-of-AC-directions} shows that the batch actor direction is subject to three distinct effects: finite-grid approximation bias, critic approximation bias, and Monte Carlo sampling variability. \textbf{The first two affect its conditional expectation, whereas the last affects its conditional variance. Consequently, increasing the batch size cannot, by itself, eliminate the errors caused by a fixed time grid or an inaccurate critic.}

We now use the directions in \eqref{eq:batch-directions} to update the critic and actor parameters. Under the reward-maximization convention, consider the projected updates
\begin{equation*}
\theta_{m+1}  =\Pi_{\Theta}\left[\theta_m+\eta_{\theta, m} \Delta \theta_m\right], \quad
\phi_{m+1} =\Pi_{\Phi}\left[\phi_m+\eta_{\phi, m} \Delta \phi_m\right],
\end{equation*}
where $\Pi_{\Theta}$ and $\Pi_{\Phi}$ denote projections onto the prescribed parameter domains, and $\eta_{\theta, m}, \eta_{\phi, m}>0$ are the critic and actor learning rates.

The two updates have different interpretations. The critic update seeks a root of the finite-grid PE moment equation $f_{\mathcal{P}}(\theta, \phi)=0$, whereas the actor update seeks to improve the policy using the estimated policy-gradient direction $g_{\mathcal{P}}(\theta, \phi)$.

\begin{remark}[Interpretation and implementation]
The preceding results establish the connection between the continuous-time Actor-Critic framework and its finite-grid batch implementation. Their implications are threefold.

First, the discrete TD increment in \eqref{eq:the-discrete-TD-increment} retains the observed value-function changes, including those arising from regime transitions. As a result, the critic and actor directions can be computed from sampled trajectories without explicit evaluation of the environmental drift, diffusion coefficients, or switching generator. The operator $\mathcal{L}^a$ appears in the consistency analysis through the PE residual \eqref{eq:the-action-dependent-PE-residual}, but it does not need to be evaluated in the algorithm.

Second, the occupation and time-discretization consistency conditions are essential for connecting finite-grid sampling to continuous-time theory. A sample-and-hold policy generally does not satisfy exact Assumption \ref{assum:occupation-matching} at a fixed nonzero mesh, since actions remain unchanged between decision times even when the state or regime changes. Therefore, the continuous-time PE and PG identities are recovered through limiting arguments rather than being assumed to hold exactly for the finite-grid implementation.

Third, the batch and two-time-scale constructions address different aspects of the learning problem. Batch averaging reduces the variability of the estimating directions, while time-grid refinement reduces discretization bias. Two-time-scale stochastic approximation is intended to allow the critic to track the evolving policy, thereby controlling the approximation error entering the actor direction. A complete convergence analysis must account for all three effects, together with the stability and identification properties of the underlying mean-field equations.

Consequently, the proposed finite-grid batch Actor-Critic scheme provides an implementable approximation of the continuous-time martingale-based PE and PG framework. Under the stated consistency conditions, the exact-critic batch directions recover the continuous-time PE zero-moment condition and policy-gradient representation in the vanishing-mesh limit, while the additional stochastic approximation conditions govern the behavior of the coupled parameter updates.
\end{remark}

\begin{algorithm}[htbp]
\caption{Offline--Episodic Regime-Switching Actor--Critic Algorithm}
\textbf{Inputs:}
Initial state-regime pair $(x,i)$; partition $\mathcal{P}$, $h_k$; initial parameters $(\theta_0,\phi_0)$; batch size $M$; number of actor-critic iterations $K_{\mathrm{AC}}$; learning-rate sequences $\{\eta_{\theta,m}\}$ and $\{\eta_{\phi,m}\}$; functional form of the critic $J^\theta(t,x,i)$; functional form of the policy $\pi^\phi(\cdot\mid t,x,i)$; functional form of the regularizer $p(t,x,i,a,\pi^\phi(\cdot\mid t,x,i))$; temperature parameter $\gamma$; discount factor $\beta$; projection sets $\Theta$ and $\Phi$. When no parameter constraints are imposed, $\Pi_\Theta$ and $\Pi_\Phi$ are interpreted as the identity mappings.\\[1ex]
\textbf{Required program:}
A regime-switching environment simulator
$$
(X',I',\widehat{\ell}) = \operatorname{Environment}_{h} (t,X,I,a),
$$
which, given the current time-state-regime-action tuple $(t,X,I,a)$, generates the next state $X'$, the next regime $I'$, and an observed regularized reward rate $\widehat{\ell}$, satisfying
$$
\mathbb{E}[\widehat{\ell}\mid\mathcal{H},a] = r(t,X,I,a) + \gamma p( t,X,I,a,\pi^\phi(\cdot\mid t,X,I) ).
$$
\textbf{Learning procedure:}
\begin{algorithmic}[1]
\State Initialize $(\theta,\phi)\leftarrow(\theta_0,\phi_0)$.
\For{$m=0,\ldots,K_{\mathrm{AC}}-1$}
    \State Freeze the current parameter pair $(\theta_m,\phi_m)$.
    \For{trajectory $n=1,\ldots,M$}
        \State Initialize $X_{n,t_0}\leftarrow x, \quad \alpha_{n,t_0}\leftarrow i$.
        \For{$k=0,\ldots,N-1$}
            \State Sample $a_{n,k} \sim \pi^{\phi_m} ( \cdot\mid t_k,X_{n,t_k},\alpha_{n,t_k})$.
            \State Apply $a_{n,k}$ over $[t_k,t_{k+1})$ and observe $(X_{n,t_{k+1}}, \alpha_{n,t_{k+1}}, \widehat{\ell}_{n,k})$.
            \State Store the state-regime-action-reward transition and continue until $t_N=T$.
        \EndFor
    \EndFor
    \State After all $M$ complete trajectories have been generated, compute, for $n=1,\ldots,M$ and $k=0,\ldots,N-1$,
    \begin{align*}
    \delta_{n,k}^{\theta_m,\phi_m} = J^{\theta_m} ( t_{k+1}, X_{n,t_{k+1}}, \alpha_{n,t_{k+1}} ) - J^{\theta_m} ( t_k, X_{n,t_k}, \alpha_{n,t_k} )+ h_k \left[ \widehat{\ell}_{n,k} - \beta J^{\theta_m} ( t_k, X_{n,t_k}, \alpha_{n,t_k} ) \right].
    \end{align*}
    \State Compute the critic direction
    $$
    \Delta\theta_m = \frac1M \sum_{n=1}^{M} \sum_{k=0}^{N-1} \nabla_\theta J^{\theta_m} ( t_k, X_{n,t_k}, \alpha_{n,t_k} ) \delta_{n,k}^{\theta_m,\phi_m}.
    $$
    \State For the actor, compute
    $$
    P_{n,k}^{\phi_m} = \nabla_\phi p( t_k, X_{n,t_k}, \alpha_{n,t_k}, a_{n,k}, \pi^\phi(\cdot\mid t_k,X_{n,t_k},\alpha_{n,t_k}) )|_{\phi=\phi_m},
    $$
    \State and
    \begin{align*}
    \mathcal G_{n,k}^{\theta_m,\phi_m} = e^{-\beta t_k} \Bigg[ \nabla_\phi \log \pi^{\phi_m}(a_{n,k}\mid t_k,X_{n,t_k},\alpha_{n,t_k}) \delta_{n,k}^{\theta_m,\phi_m}+ \gamma h_k P_{n,k}^{\phi_m} \Bigg].
    \end{align*}
    \State Then set
    $$
    \Delta\phi_m = \frac1M \sum_{n=1}^{M} \sum_{k=0}^{N-1} \mathcal G_{n,k}^{\theta_m,\phi_m}.
    $$
    \State Update the critic and actor by
    $$
    \theta_{m+1} = \Pi_\Theta \left[ \theta_m + \eta_{\theta,m}\Delta\theta_m \right], \qquad \phi_{m+1} = \Pi_\Phi \left[ \phi_m + \eta_{\phi,m}\Delta\phi_m \right].
    $$
\EndFor
\State \textbf{Return} $(\theta_{K_{\mathrm{AC}}},\phi_{K_{\mathrm{AC}}})$.
\end{algorithmic}
\end{algorithm}

\subsection{Online regime-switching Actor-Critic}

\noindent\textit{Continuous-Time Representation.} The Proposition \ref{pro:Trajectory-based-policy-gradient-representation} established in the preceding offline actor-critic analysis expresses the policy gradient as an expected integral over the remaining time horizon. Although this representation provides a basis for episodic learning, it does not directly yield an actor update after each observed transition. To develop an online formulation, we combine the first-order stationarity condition with the policy-evaluation moment condition, thereby obtaining an additional class of actor estimating directions.

Recall the regularized instantaneous reward $\ell^\phi$ defined in
\eqref{eq:the-regularized-instantaneous-reward}, and let $J^\phi(t, x, i)=J\left(t, x, i ; \pi^\phi\right)$ be the value function of the current policy. We write $d \delta_t^\phi$ for the sampled-action TD increment $d \delta_t^{\theta^*, \phi}$ associated with the exact critic $J^{\theta^*}(\cdot,\cdot,\cdot)=J^\phi(\cdot,\cdot,\cdot)$, as defined in the preceding offline analysis.

The key observation is that the exact critic satisfies a PE moment condition even though $d \delta_t^\phi$ is not necessarily a martingale increment along the sampled-action trajectory. This allows us to modify the score-weighted actor direction without changing its expectation.
\begin{lemma}\label{lem:PE-orthogonality-condtion}
Suppose that the Proposition \ref{pro:Trajectory-based-critic moment-condition} hold. Let $\zeta_t=\zeta(t, X_t, \alpha_{t-})$
be a predictable, action-independent, $\mathbb{R}^{d_\phi}$-valued test process satisfying the required integrability conditions. Assume that Assumption
\ref{assum:occupation-matching} holds for the corresponding weighted PE residual. Then
\begin{equation}\label{eq:PE-orthogonality-condtion}
\mathbb{E}_{0, x, i}^\phi\left[\int_0^T e^{-\beta t} \zeta_t d \delta_t^\phi\right]=0.
\end{equation}
\end{lemma}

\begin{proof}
By \eqref{eq:continuous-time-semimartingale-increment}, the TD increment associated with the exact critic $J^{\theta^*}(\cdot,\cdot,\cdot)=J^\phi(\cdot,\cdot,\cdot)$ admits
\begin{equation*}
d\delta_t^\phi=\left[(\mathcal{L}^{a_t} J^\phi)(t, X_t, \alpha_{t-})+\ell^\phi(t, X_t, \alpha_{t-}, a_t)-\beta J^\phi(t, X_t, \alpha_{t-})\right] d t+d \mathcal{N}_t^{\theta^*, \phi}.
\end{equation*}
Since $e^{-\beta t} \zeta_t$ is predictable and satisfies the required integrability conditions, the stochastic integral has zero expectation:
$$
\mathbb{E}_{0, x, i}^\phi\left[\int_0^T e^{-\beta t} \zeta_t d \mathcal{N}_t^{\theta^*, \phi}\right]=0.
$$
Consequently, applying the occupation-matching condition to the remaining drift term and using the policy-evaluation equation \eqref{eq:PE-for-J-theta(phi)}, we obtain
\begin{align*}
& \mathbb{E}_{0, x, i}^\phi\left[\int_0^T e^{-\beta t} \zeta_t d \delta_t^\phi\right] \\
& =\mathbb{E}_{0, x, i}^\phi \int_0^T e^{-\beta t} \zeta_t\left[(\mathcal{L}^{a_t} J^\phi)(t, X_t, \alpha_{t-})+\ell^\phi(t, X_t, \alpha_{t-}, a_t)-\beta J^\phi(t, X_t, \alpha_{t-})\right] d t \\
& =\mathbb{E}_{0, x, i}^{\pi^\phi} \int_0^T e^{-\beta t} \zeta(t, \widetilde{X}_t^{\pi^\phi}, \alpha_{t-}) \\
& \quad \times \int_{\mathcal{A}}\left[(\mathcal{L}^a J^\phi)(t, \widetilde{X}_t^{\pi^\phi}, \alpha_{t-})+\ell^\phi(t, \widetilde{X}_t^{\pi^\phi}, \alpha_{t-}, a)-\beta J^\phi(t, \widetilde{X}_t^{\pi^\phi}, \alpha_{t-})\right] \pi^\phi(a \mid t, \widetilde{X}_t^{\pi^\phi}, \alpha_{t-}) \lambda(d a) d t \\
& =0.
\end{align*}
The last equality follows from \eqref{eq:PE-for-J-theta(phi)}, which holds for every $(t,x,i)$. This proves \eqref{eq:PE-orthogonality-condtion}.
\end{proof}

The Lemma \eqref{lem:PE-orthogonality-condtion} shows that the PE moment condition provides an auxiliary zero-mean term for the actor. Adding this term to the Proposition \ref{pro:Trajectory-based-policy-gradient-representation} in the preceding offline actor-critic subsection preserves the expected actor direction when the critic is exact. The resulting characterization is given below.

\begin{theorem}\label{the:orthogonality-condition}
Suppose that the assumptions required by the preceding Proposition \ref{pro:Trajectory-based-policy-gradient-representation} hold.
Let $\zeta_t$ satisfy the conditions of the Lemma \ref{lem:PE-orthogonality-condtion}. Then, for every admissible policy parameter $\phi$,
\begin{align}\label{eq:orthogonality-condition}
\nabla_\phi J(0, x, i ; \pi^\phi)=\mathbb{E}_{0, x, i}^\phi & \bigg[\int _ { 0 } ^ { T } e ^ { - \beta t } \Big\{[\nabla_\phi \log \pi^\phi(a_t \mid t, X_t, \alpha_{t-})+\zeta_t] d \delta_t^\phi\notag\\
& +\gamma \nabla_\phi p(t, X_t, \alpha_{t-}, a_t, \pi^\phi(\cdot \mid t, X_t, \alpha_{t-})) d t\Big\}\bigg].
\end{align}
In particular, if $\phi^*$ is an interior stationary point satisfying
$\nabla_\phi J(0, x, i ; \pi^{\phi^*})=0$,
then
\begin{align*}
0=\mathbb{E}_{0, x, i}^{\phi^*}\bigg[\int_0^T e^{-\beta t} & \Big\{\big[\nabla_\phi \log \pi^{\phi^*}\left(a_t \mid t, X_t, \alpha_{t-}\right)+\zeta_t\big] d \delta_t^{\phi^*} \\
& +\gamma \nabla_\phi p(t, X_t, \alpha_{t-}, a_t, \pi^\phi(\cdot \mid t, X_t, \alpha_{t-}))|_{\phi=\phi^*} d t\Big\}\bigg].
\end{align*}
\end{theorem}

\begin{proof}
The preceding Proposition \ref{pro:Trajectory-based-policy-gradient-representation} establishes \eqref{eq:orthogonality-condition} with $\zeta_t=0$. By Lemma \ref{lem:PE-orthogonality-condtion}, the additional term satisfies
$$
\mathbb E_{0,x,i}^\phi
\int_0^T e^{-\beta t}\zeta_t\,d\delta_t^\phi=0.
$$
Adding this identity to the Proposition \ref{pro:Trajectory-based-policy-gradient-representation} representation proves \eqref{eq:orthogonality-condition}. The stationary-policy assertion follows immediately by setting $\phi=\phi^*$.
\end{proof}
This theorem establishes a family of actor orthogonality conditions through the choice of $\zeta_t$. Unlike the critic test process used in CTD(0), $\zeta_t$ is introduced into the actor direction to preserve the policy-gradient expectation while allowing alternative trajectory-based estimating directions.
Two distinctions are important. First, \eqref{eq:orthogonality-condition} characterizes the gradient of the policy objective and hence its first-order stationarity condition; without concavity or another suitable sufficiency condition, stationarity does not imply global optimality. Second, the zero-moment condition at $\phi^*$ holds for the accumulated actor increment over the horizon. It does not imply that the actor increment has zero conditional expectation at each time, nor does it justify multiplying the entire actor increment by an arbitrary predictable process.
The role of $\zeta_t$ is therefore different from imposing an unrestricted family of actor orthogonality conditions: its zero-mean contribution is justified specifically by the PE moment identity \eqref{eq:PE-orthogonality-condtion}.

We now replace the exact value function $J^\phi(\cdot,\cdot,\cdot)$ by a parameterized critic $J^\theta(\cdot,\cdot,\cdot)$, satisfying the terminal condition $J^\theta(T, x, i)=h(x)$. The sampled-action TD increment $d \delta_t^{\theta, \phi}$ is the same as that used in the preceding offline actor-critic formulation.

The continuous-time critic and actor estimating increments are defined by
\begin{equation}\label{eq:the-PE-direction}
\begin{aligned}
d U_t^\theta= & \nabla_\theta J^\theta\left(t, X_t, \alpha_{t-}\right) d \delta_t^{\theta, \phi} \\
d U_t^\phi= & e^{-\beta t}\left[\left(\nabla_\phi \log \pi^\phi\left(a_t \mid t, X_t, \alpha_{t-}\right)+\zeta_t\right) d \delta_t^{\theta, \phi}+\gamma \nabla_\phi p\left(t, X_t, \alpha_{t-}, a_t, \pi^\phi(\cdot \mid t, X_t, \alpha_{t-})\right) d t\right] .
\end{aligned}
\end{equation}
These directions share the same observed TD increment, but they have different statistical interpretations. The critic seeks a root of the PE moment equation, whereas the actor uses a score-weighted TD increment together with the explicit policy derivative of the regularizer to estimate the policy gradient.

When the critic is exact, Theorem \ref{the:orthogonality-condition} gives
\begin{equation}\label{eq:expectation-of-Ut-phi}
\mathbb{E}_{0, x, i}^\phi\left[\int_0^T d U_t^\phi\right]=\nabla_\phi J(0, x, i ; \pi^\phi) .
\end{equation}
Consequently, at an interior stationary policy $\phi^*$,
$$
\mathbb{E}_{0, x, i}^{\phi^*}\left[\int_0^T d U_t^{\phi^*}\right]=0, \quad J^\theta(\cdot,\cdot,\cdot)=J^{\phi^*}(\cdot,\cdot,\cdot)
$$
For an approximate critic, however, the additional term involving $\zeta_t$ need not have zero expectation. To make this dependence explicit, let
$$
E^\theta(t, x, i):=J^\theta(t, x, i)-J^\phi(t, x, i)
$$
denote the critic approximation error. Subtracting the exact-critic actor increment from \eqref{eq:the-PE-direction}, and using \eqref{eq:expectation-of-Ut-phi}, yields
\begin{align}\label{eq:approximation-bias-of-online-actor-direction}
& \mathbb{E}_{0, x, i}^\phi\left[\int_0^T d U_t^\phi\right]-\nabla_\phi J(0, x, i ; \pi^\phi) \notag\\
& =\mathbb{E}_{0, x, i}^\phi \int_0^T e^{-\beta t}\left[\nabla_\phi \log \pi^\phi(a_t \mid t, X_t, \alpha_{t-})+\zeta_t\right]  \times\left[d E^\theta(t, X_t, \alpha_t)-\beta E^\theta(t, X_t, \alpha_{t-}) d t\right] .
\end{align}
Equation \eqref{eq:approximation-bias-of-online-actor-direction} shows that the approximation bias of the online actor direction depends on the score- and $\zeta_t$-weighted TD increment of the critic approximation error. Thus, although $\zeta_t$ has zero expected contribution at the exact critic, it may affect both the bias and variance of the actor estimate when the critic is approximate. In particular, choosing $\zeta_t \neq 0$ does not automatically improve estimation accuracy. The choice $\zeta_t \equiv 0$ recovers the basic score-weighted actor direction established in the offline formulation.

The representation \eqref{eq:the-PE-direction} provides continuous-time increments that can be evaluated from newly observed state-regime-action data. Nevertheless, its exact population interpretation remains an identity for the accumulated increments under a fixed policy. We next construct the corresponding single-transition updates and examine their relationship with the continuous-time formulation.

{\it Time Discretization and Online Updates .} We now discretize the continuous-time directions in \eqref{eq:the-PE-direction} to obtain an online actor-critic algorithm that updates its parameters after each observed transition. \textbf{Unlike the preceding batch construction, which accumulates the estimating directions over complete trajectories before updating the parameters, the online scheme uses the information obtained from a single transition.}

We adopt the time partition $\mathcal{P}$ and pre-action information sets $\mathcal{H}_k$ introduced in the preceding offline formulation. At each decision time $t_k$, the current parameter pair $(\theta_m, \phi_m)$ is $\mathcal{H}_k$ -measurable and held fixed over the subsequent transition.
Given $\phi_m$, sample
$$
a_k \mid \mathcal{H}_k \sim \pi^{\phi_m}(\cdot \mid t_k, X_{t_k}, \alpha_{t_k}),
$$
and hold this action fixed until $t_{k+1}$, as in the preceding finite-grid construction. Let the regularized reward estimator $\hat{\ell}_k$ satisfy the same conditional unbiasedness and moment conditions as those established in the offline analysis.

After observing the next state-regime pair, we compute the finite-grid TD increment
\begin{equation}\label{eq:delta_k-m}
\delta_k^{\theta_m, \phi_m}:= J^{\theta_m}(t_{k+1}, X_{t_{k+1}}, \alpha_{t_{k+1}})-J^{\theta_m}(t_k, X_{t_k}, \alpha_{t_k})  +h_k\left[\hat{\ell}_k-\beta J^{\theta_m}(t_k, X_{t_k}, \alpha_{t_k})\right].
\end{equation}

For CTD(0), the single-transition critic direction is
\begin{equation}\label{eq:single-transition-critic-direction} \Delta \theta_{m, k} = \nabla_\theta J^\theta(t_k, X_{t_k}, \alpha_{t_k})\big|_{\theta=\theta_m} \, \delta_k^{\theta_m, \phi_m}.
\end{equation}
The corresponding actor direction is defined by
\begin{align}\label{eq:actor-direction}
\Delta \phi_{m, k}:=e^{-\beta t_k} & \bigg[\Big(\nabla_\phi \log \pi^\phi(a_k \mid t_k, X_{t_k}, \alpha_{t_k})|_{\phi=\phi_m}+\zeta_{t_k}\Big) \delta_k^{\theta_m, \phi_m}  \notag\\
& \quad+\gamma h_k \nabla_\phi p(t_k, X_{t_k}, \alpha_{t_k}, a_k, \pi^\phi(\cdot \mid t_k, X_{t_k}, \alpha_{t_k}))\big|_{\phi=\phi_m}\bigg],
\end{align}
where $\zeta_{t_k}=\zeta(t_k, X_{t_k}, \alpha_{t_k})$ is measurable with respect to $\mathcal{H}_k$ and satisfies the conditions of the preceding actor orthogonality theorem.
The choice $\zeta_{t_k}=0$ recovers the basic actor direction in the offline finite-grid formulation. The factor $e^{-\beta t_k}$ is retained to approximate the discounted policy-gradient representation.
Equations \eqref{eq:delta_k-m}-\eqref{eq:actor-direction} show that both learning directions are obtained from the same newly observed transition. The principal difference from batch learning is therefore the frequency of parameter updates, rather than the definition of the underlying TD increment.

We next clarify how the single-transition directions relate to the continuous-time PE and PG equations. \textbf{The central issue is that the online algorithm uses a sequence of changing policy parameters, whereas the continuous-time policy-gradient theorem and the preceding finite-grid consistency proposition are formulated for a fixed policy.}
\textbf{For this reason, we first consider the local conditional expectations with the pre-update parameters held fixed over the current transition.}

Recall the action-dependent PE residual $H^{\theta, \phi}$ introduced in the offline finite-grid analysis. The local consistency lemma established there gives
\begin{equation}\label{eq:expection-of-delta_k-m}
\mathbb{E}\left[\delta_k^{\theta_m, \phi_m} \mid \mathcal{H}_k, a_k\right]=h_k H^{\theta_m, \phi_m}(t_k, X_{t_k}, \alpha_{t_k}, a_k)+R_{m, k}^{\theta_m, \phi_m},
\end{equation}
where $R_{m, k}^{\theta_m, \phi_m}$ denotes the local consistency remainder, including any additional numerical simulation error. This relation requires the same regularity, filtration, and integrability conditions as the offline local consistency lemma.

The next proposition shows how \eqref{eq:expection-of-delta_k-m} determines the conditional means of the online Critic and Actor directions.

\begin{proposition}\label{pro:Local consistency-of-online-AC-directions}
Suppose that the local consistency relation \eqref{eq:expection-of-delta_k-m} holds and that the required weighted remainder bounds are satisfied. Let $(\theta_m,\phi_m)$ be fixed over the current transition. If the critic is exact for the current policy, namely,
$J^{\theta_m}(t, x, i)=J(t, x, i ; \pi^{\phi_m})$, then
\begin{equation}\label{eq:Delta-theta_mk}
\mathbb{E}\left[\Delta \theta_{m, k} \mid \mathcal{H}_k\right]= \nabla_\theta J^\theta(t_k, X_{t_k}, \alpha_{t_k})|_{\theta=\theta_m} \times \mathbb{E}\left[R_{m, k}^{\theta_m, \phi_m} \mid \mathcal{H}_k\right],
\end{equation}
and
\begin{align}\label{eq:Delta-phi_mk}
\mathbb{E}\left[\Delta \phi_{m, k} \mid \mathcal{H}_k\right]= & h_k e^{-\beta t_k} \int_{\mathcal{A}} \check{r}(t_k, X_{t_k}, \alpha_{t_k}, a ; \phi_m) \pi^{\phi_m}(d a \mid t_k, X_{t_k}, \alpha_{t_k}) \notag\\
& +e^{-\beta t_k} \mathbb{E}\left[\left(\nabla_\phi \log \pi^{\phi_m}(a_k \mid t_k, X_{t_k}, \alpha_{t_k})+\zeta_{t_k}\right) R_{m, k}^{\theta_m, \phi_m} \mid \mathcal{H}_k\right] .
\end{align}
\end{proposition}

\begin{proof}
We first prove \eqref{eq:Delta-theta_mk}.
Since the critic parameter gradient is $\mathcal{H}_k$-measurable, applying the tower property and the local consistency relation \eqref{eq:expection-of-delta_k-m} yields
\begin{align*}
 \mathbb{E}[\Delta \theta_{m, k} \mid \mathcal{H}_k]
& =\nabla_\theta J^\theta(t_k, X_{t_k}, \alpha_{t_k})\Big|_{\theta=\theta_m} \mathbb{E}\left[\delta_k^{\theta_m, \phi_m} \Big| \mathcal{H}_k\right] \\
& =\nabla_\theta J^\theta(t_k, X_{t_k}, \alpha_{t_k})\Big|_{\theta=\theta_m}\bigg\{h_{k} \int_{\mathcal{ A } } \Big[(\mathcal{L}^a J^{\theta_m})(t_k, X_{t_k}, \alpha_{t_k})+\ell^{\phi_m}(t_k, X_{t_k}, \alpha_{t_k}, a) \\
& \quad \quad-\beta J^{\theta_m}(t_k, X_{t_k}, \alpha_{t_k})\Big] \pi^{\phi_m}(d a \mid t_k, X_{t_k}, \alpha_{t_k})+\mathbb{E}\left[R_{m, k}^{\theta_m, \phi_m} \mid \mathcal{H}_k\right]\bigg\}.
\end{align*}
Since $J^{\theta_m}(\cdot, \cdot, \cdot)=J(\cdot, \cdot, \cdot ; \pi^{\phi_m})$, the policy-averaged term vanishes by the PE equation \eqref{eq:PE-for-J-theta(phi)}. Therefore,
$$
\mathbb{E}[\Delta \theta_{m, k} \mid \mathcal{H}_k]=\nabla_\theta J^\theta(t_k, X_{t_k}, \alpha_{t_k})\Big|_{\theta=\theta_m} \mathbb{E}\left[R_{m, k}^{\theta_m, \phi_m} \mid \mathcal{H}_k\right],
$$
which proves \eqref{eq:Delta-theta_mk}.

We next establish \eqref{eq:Delta-phi_mk}. Unlike the critic parameter gradient, the policy score depends on the sampled action $a_k$. We therefore apply \eqref{eq:expection-of-delta_k-m} before averaging over the conditional action distribution. Using the definition of the actor direction in \eqref{eq:actor-direction}, we obtain
\begin{align*}
\mathbb{E}[\Delta \phi_{m, k} \mid \mathcal{H}_k]
= & h_k e^{-\beta t_k} \int_{\mathcal{A}}\bigg\{\Big[\nabla_\phi \log \pi^{\phi_m}(a \mid t_k, X_{t_k}, \alpha_{t_k})+\zeta_{t_k}\Big] \\
& \times\Big[(\mathcal{L}^a J^{\theta_m})(t_k, X_{t_k}, \alpha_{t_k})+\ell^{\phi_m}(t_k, X_{t_k}, \alpha_{t_k}, a)-\beta J^{\theta_m}(t_k, X_{t_k}, \alpha_{t_k})\Big] \\
& +\gamma \nabla_\phi p(t_k, X_{t_k}, \alpha_{t_k}, a, \pi^\phi(\cdot \mid t_k, X_{t_k}, \alpha_{t_k}))\Big|_{\phi=\phi_m}\bigg\} \pi^{\phi_m}(d a \mid t_k, X_{t_k}, \alpha_{t_k}) \\
& +e^{-\beta t_k} \mathbb{E}\Big[(\nabla_\phi \log \pi^{\phi_m}(a_k \mid t_k, X_{t_k}, \alpha_{t_k})+\zeta_{t_k}) R_{m, k}^{\theta_m, \phi_m} \mid \mathcal{H}_k\Big] .
\end{align*}
Since $\zeta_{t_k}$ is $\mathcal{H}_k$-measurable and independent of the current action, the PE equation \eqref{eq:PE-for-J-theta(phi)} implies
\begin{align*}
\int_{\mathcal{A}} \zeta_{t_k}\Big[ (\mathcal{L}^a J^{\theta_m})(t_k, X_{t_k}, \alpha_{t_k})+\ell^{\phi_m}(t_k, X_{t_k}, \alpha_{t_k}, a) -\beta J^{\theta_m}(t_k, X_{t_k}, \alpha_{t_k})\Big] \pi^{\phi_m}(d a \mid t_k, X_{t_k}, \alpha_{t_k})=0 .
\end{align*}
Thus, the contribution of $\zeta_{t_k}$ to the leading-order term vanishes. By the definition of the auxiliary reward $\check{r}$ in the preceding policy-gradient theorem, the remaining score-weighted PE residual and regularizer derivative combine to give
\begin{align*}
\mathbb{E}[\Delta \phi_{m, k} \mid \mathcal{H}_k]= & h_k e^{-\beta t_k} \int_{\mathcal{A}} \check{r}(t_k, X_{t_k}, \alpha_{t_k}, a ; \phi_m) \pi^{\phi_m}(d a | t_k, X_{t_k}, \alpha_{t_k}) \\
& +e^{-\beta t_k} \mathbb{E}\Big[(\nabla_\phi \log \pi^{\phi_m}(a_k | t_k, X_{t_k}, \alpha_{t_k})+\zeta_{t_k}) R_{m, k}^{\theta_m, \phi_m} | \mathcal{H}_k\Big],
\end{align*}
which is precisely \eqref{eq:Delta-phi_mk}.
\end{proof}

The proposition identifies the difference between the local Critic and Actor conditions. At the exact critic, the leading term in the conditional mean of the Critic direction vanishes, whereas the corresponding Actor direction generally has a nonzero leading term proportional to the policy-averaged auxiliary reward.
In particular, even if $\phi^*$ is an interior stationary policy satisfying
$\nabla_\phi J(0, x, i ; \pi^{\phi^*})=0$,
it does not follow that
$$
\int_{\mathcal{A}} \check{r}(t, x, i, a ; \phi^*) \pi^{\phi^*}(d a \mid t, x, i)=0,
$$
at every $(t, x, i)$. The stationary-policy condition concerns the accumulated expected policy gradient over the horizon, rather than the conditional expectation of each individual actor increment.
This distinction is essential for online learning: \textbf{a single actor increment represents a local contribution to the policy-gradient estimate, not an unbiased estimate of the entire policy gradient.}

For a fixed parameter pair $(\theta,\phi)$, summing the discrete directions over the complete horizon $T$ and applying the occupation and discretization consistency results from the offline analysis gives, at the exact critic,
\begin{equation}\label{eq:convergence-of-Delta theta-k-and-phi-k}
 \mathbb{E}_{0, x, i}^{\phi, \mathcal{P}}\left[\sum_{k=0}^{N-1} \Delta \theta_k\right] \rightarrow 0, \quad
 \mathbb{E}_{0, x, i}^{\phi, \mathcal{P}}\left[\sum_{k=0}^{N-1} \Delta \phi_k\right] \rightarrow \nabla_\phi J\left(0, x, i ; \pi^\phi\right), \quad
|\mathcal{P}| \rightarrow 0 .
\end{equation}
The auxiliary terms involving $\zeta_{t_k}$ make no contribution to the limiting expected actor direction, provided their weighted local consistency errors vanish after summation. Thus, \eqref{eq:convergence-of-Delta theta-k-and-phi-k} follows from the existing offline consistency theory together with the additional PE orthogonality established in \eqref{eq:PE-orthogonality-condtion}; it does not require a separate proof of the full policy-gradient representation.
However, \eqref{eq:convergence-of-Delta theta-k-and-phi-k} assumes that the same policy parameter is used throughout the trajectory. It therefore cannot be applied directly to an online recursion in which the parameters are updated after every transition. For such a recursion, the local directions remain well defined, but their accumulated mean must also account for the variation of the critic and actor parameters during learning.

We now use \eqref{eq:single-transition-critic-direction} and \eqref{eq:actor-direction} to construct an online Actor-Critic recursion. Since the actor direction depends on the critic approximation, we adopt a two-time-scale stochastic approximation scheme in which the critic evolves faster than the actor.

Let $m$ denote the global update index across episodes and transitions. After observing the $m$-th transition and computing both estimating directions using the same pre-update parameter pair $(\theta_m, \phi_m)$, update
\begin{equation}\label{eq:online-update}
\theta_{m+1}  =\Pi_{\Theta}\left[\theta_m+\eta_{\theta, m} \Delta \theta_{m, k}\right],  \quad
\phi_{m+1}  =\Pi_{\Phi}\left[\phi_m+\eta_{\phi, m} \Delta \phi_{m, k}\right],
\end{equation}
where $\eta_{\theta, m}, \eta_{\phi, m}>0$ are the learning rates, and $\Pi_{\Theta}, \Pi_{\Phi}$ denote projections onto the prescribed parameter sets.

\begin{algorithm}[htbp]
\caption{Online--Incremental Regime-Switching Actor--Critic Algorithm}
\textbf{Inputs:}
Initial state-regime pair $(x,i)$; partition $\mathcal{P}$, $h_k$; initial parameters $(\theta_0,\phi_0)$; number of episodes $K_{\mathrm{ep}}$; learning-rate sequences $\{\eta_{\theta,m}\}$ and $\{\eta_{\phi,m}\}$; functional form of the critic $J^\theta(t,x,i)$; functional form of the policy $\pi^\phi(\cdot\mid t,x,i)$; regularizer $p(t,x,i,a,\pi^\phi(\cdot\mid t,x,i))$; predictable auxiliary process $\zeta_{t_k}$, with the default choice $\zeta_{t_k}=0$; temperature parameter $\gamma$; discount factor $\beta$; projection sets $\Theta$ and $\Phi$.\\[1ex]
\textbf{Required program:}
A regime-switching environment simulator
$(X',I',\widehat{\ell}) = \operatorname{Environment}_{h} (t,X,I,a),$
which generates the next state, next regime, and observed regularized reward rate following the current action.\\[1ex]
\textbf{Learning procedure:}
\begin{algorithmic}[1]
\State Initialize $(\theta,\phi)\leftarrow(\theta_0,\phi_0), \quad m\leftarrow0$.
\For{episode $n=1,\ldots,K_{\mathrm{ep}}$}
    \State Reset $X_{t_0}\leftarrow x, \quad \alpha_{t_0}\leftarrow i$.
    \For{$k=0,\ldots,N-1$}
        \State Freeze the pre-transition parameters $(\theta_m,\phi_m)$.
        \State Compute the predictable critic test function
        $$
        \xi_{t_k} = \nabla_\theta J^{\theta_m} (t_k,X_{t_k},\alpha_{t_k}),
        $$
        \State and, if used, compute $\zeta_{t_k}$ from information available before sampling the current action.
        \State Sample $a_k \sim \pi^{\phi_m} (\cdot\mid t_k,X_{t_k},\alpha_{t_k})$.
        \State Apply $a_k$ over $[t_k,t_{k+1})$ and observe $( X_{t_{k+1}}, \alpha_{t_{k+1}}, \widehat{\ell}_k)$.
        \State Compute the first-order PE increment
        \begin{align*}
        \delta_k^{\theta_m,\phi_m} = J^{\theta_m} (t_{k+1},X_{t_{k+1}},\alpha_{t_{k+1}}) - J^{\theta_m} (t_k,X_{t_k},\alpha_{t_k}) + h_k \left[ \widehat{\ell}_k - \beta J^{\theta_m} (t_k,X_{t_k},\alpha_{t_k}) \right].
        \end{align*}
        \State Compute the critic direction
        $$
        \Delta\theta_{m,k} = \xi_{t_k} \delta_k^{\theta_m,\phi_m}.
        $$
        \State Compute
        $$
        P_k^{\phi_m} =\nabla_\phi p( t_k, X_{t_k}, \alpha_{t_k}, a_k, \pi^\phi(\cdot\mid t_k,X_{t_k},\alpha_{t_k}) )|_{\phi=\phi_m},
        $$
        \State and the online actor direction
        \begin{align*}
        \Delta\phi_{m,k} = e^{-\beta t_k} \bigg[( \nabla_\phi \log \pi^{\phi_m} (a_k\mid t_k,X_{t_k},\alpha_{t_k}) + \zeta_{t_k} ) \delta_k^{\theta_m,\phi_m} + \gamma h_k P_k^{\phi_m} \bigg].
        \end{align*}
        \State Using the same pre-transition pair $(\theta_m,\phi_m)$ in both directions, update
        $$
        \theta_{m+1} = \Pi_\Theta \left[ \theta_m + \eta_{\theta,m} \Delta\theta_{m,k} \right], \qquad \phi_{m+1} = \Pi_\Phi \left[ \phi_m + \eta_{\phi,m} \Delta\phi_{m,k} \right].
        $$
        \State Set $m \leftarrow m+1$ and proceed to the next transition.
    \EndFor
\EndFor
\State \textbf{Return} the final actor and critic parameters $(\theta_m,\phi_m)$.
\end{algorithmic}
\end{algorithm}

\section{Applications and Numerical Experiments: LQ problem with regime-switching}
\label{sec:Applications-and-Numerical-Experiments}

We conduct two numerical experiments to examine the theoretical results established in the preceding sections. The first experiment considers an analytically tractable two-regime linear-quadratic (LQ) problem to assess the accuracy of martingale-based policy evaluation, trajectory-based policy-gradient estimation, and actor learning. The second experiment introduces a regime-dependent critic perturbation to verify the additional regime-switching contribution to the MSTDE decomposition. Throughout this section, the discount factor is set to $\beta=0$.

\subsection{A two-regime LQ benchmark}

Consider a finite-horizon LQ problem with $T=1$, governed by
\begin{equation*}
d X_t=\left(A_i X_t+B_i U_t\right) d t+\sigma_i d W_t, \quad
X_0=x_0,\quad \alpha_0=i_0,
\end{equation*}
where the regime process $\alpha=\{\alpha_t\}_{t \in[0, T]}$ takes values in $\mathcal{M}=\{1,2\}$ and has generator
$Q=\left(\begin{array}{cc}
-0.8 & 0.8 \\
0.5 & -0.5
\end{array}\right) .$
The model parameters are specified as follows:
 $(A_1,A_2)=(0.2,-0.1),
(B_1,B_2)=(1.0,0.8)$ and $(\sigma_1,\sigma_2)=(0.35,0.50).$
The running and terminal cost parameters are
$(Q_1, Q_2)=(1.0,1.5), (R_1, R_2)=(0.5,0.8),$
and
$(M_1, M_2)=(1.0,1.2) .$
We consider the entropy-regularized quadratic cost minimization problem with regularization parameter $\lambda=0.2$.
The optimal value function and exploratory Gaussian policy admit the representations
$$
J_i^*(t, x)=P_i(t) x^2+C_i(t),  \quad \pi_i^*(\cdot \mid t, x)=\mathcal{N}\left(-\frac{B_i P_i(t)}{R_i} x, \frac{\lambda}{2 R_i}\right),
$$
respectively, where $P_i$ and $C_i$ are determined by the corresponding coupled Riccati equations. Numerical solution yields $J^*(0,1,1) \approx 1.04727$.
Hence the optimal mean-feedback gains and exploration variances are
$K_1^*(t)=2 P_1(t), K_2^*(t)=P_2(t),$
and
$v_1^*=0.2, v_2^*=0.125$.

{\it Policy Evaluation.} These analytical results provide independent benchmarks for the numerical experiments. We first fix a nonoptimal Gaussian policy,
$$
\pi^\phi(\cdot \mid x, i)=\mathcal{N}\left(-K_i x, v_i\right),
$$
with $K_1=0.6, K_2=0.4,$
and $v_1=0.30, v_2=0.20 .$
Equivalently,
$\phi_{\mathrm{PE}}=(0.6,0.4, \log 0.30, \log 0.20)$.

Under this policy, the exact value function remains quadratic,
$$
J_i^\phi(t, x)=P_i^\phi(t) x^2+C_i^\phi(t).
$$
The corresponding policy-evaluation equations become
$$
\left\{\begin{array}{l}
\dot{P}_i^\phi(t)+2(A_i-B_i K_i) P_i^\phi(t)+Q_i+R_i K_i^2+\sum_{j=1}^2 q_{i j} P_j^\phi(t)=0, \\
P_1^\phi(T)=M_i,
\end{array}\right.
$$
and
$$
\left\{\begin{array}{l}
\dot{C}_i^\phi(t)+\sigma_i^2 P_i^\phi(t)+R_i v_i-\frac{\lambda}{2}\left[\log \left(2 \pi v_i\right)+1\right]+\sum_{j=1}^2 q_{i j} C_j^\phi(t)=0, \\
C_i^\phi(T)=0.
\end{array}\right.
$$
Solving these coupled equations gives, $J^{\phi_{\mathrm{PE}}}(0,1,1) \approx 1.56887$. This value serves as the theoretical reference for policy evaluation under the prescribed policy.

State-regime-action trajectories are generated using the finite-grid scheme developed earlier. The baseline configuration uses $N=200$ time steps, corresponding to $h=0.005$, and $M_{\text {critic }}=5000$ independent trajectories for critic training. To examine sampling and discretization effects, we additionally consider
$N \in\{50,100,200,400\}$,
and $M_{\text {critic }} \in\{500,1000,2000,5000,10000\}$.
The critic accuracy is assessed on a fixed time-state-regime evaluation grid using the relative error
$$
E_J=\left[\frac{\sum_{i, k, l}\left|J_i^\theta\left(t_k, x_l\right)-J_i^\phi\left(t_k, x_l\right)\right|^2}{\sum_{i, k, l}\left|J_i^\phi\left(t_k, x_l\right)\right|^2}\right]^{1 / 2} .
$$
The evaluation grid consists of 51 time points on $[0,1]$ and 101 equally spaced state points on $[-2.5,2.5]$.

Figure~\ref{fig:exp1-critic-convergence} compares the learned critic with the analytical value function of the fixed policy, examining whether martingale-based policy evaluation accurately recovers the corresponding value function. In particular, the consistency target is
$$
J^\theta \longrightarrow J^{\phi_{\mathrm{PE}}},
$$
rather than convergence to the optimal value function $J^*$.

\begin{figure}[!htbp]
    \centering
    \includegraphics[width=1.0\textwidth]{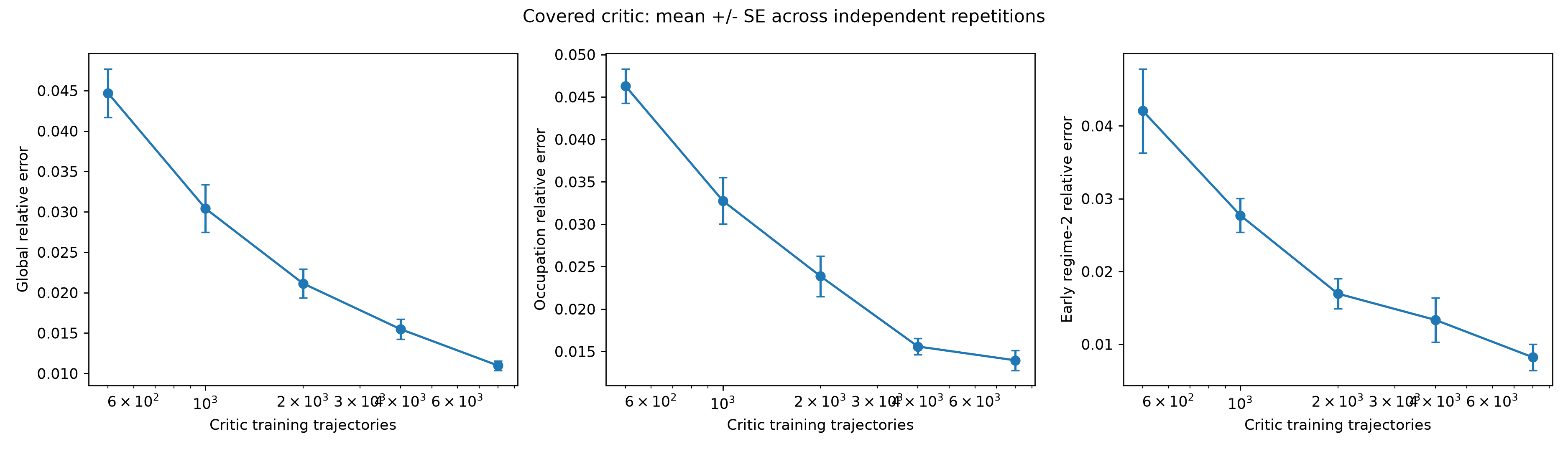}
    \caption{Convergence of the critic toward the analytical value function.}
    \label{fig:exp1-critic-convergence}
\end{figure}

{\it Policy-Gradient Estimation.} For the same prescribed policy $\phi_{\mathrm{PE}}$, the exact policy gradient is computed independently by differentiating the coupled policy-evaluation equations with respect to the actor parameters:
$$
\nabla_\phi J^{\text {true }}=\nabla_\phi J^{\phi_{\mathrm{PE}}}(0,1,1) .
$$
This reference is compared with the trajectory-based policy-gradient estimator. The relative error is defined as
$$
E_{\mathrm{PG}}=\frac{\left\|\widehat{\nabla_\phi J}-\nabla_\phi J^{\text {true }}\right\|}{1+\left\|\nabla_\phi J^{\text {true }}\right\|} .
$$
The baseline experiment uses $M_{\text {actor }}=5000$ independent trajectories. We further consider
$$
M_{\text {actor }} \in\{500,1000,2000,5000,10000\}
$$
to examine Monte Carlo accuracy.
To distinguish the sources of approximation error, the experiment is conducted in two stages. First, the exact value function $J^\phi$ is used in the trajectory-based PG estimator to assess Monte Carlo and time-discretization errors. Second, the exact value function is replaced by the learned critic $J^\theta$, allowing us to examine the additional error introduced by critic approximation. Independent trajectory samples are used for critic fitting and actor evaluation.

Figure~\ref{fig:exp1-policy-gradient} compares the trajectory-based estimate with the theoretical gradient, with the objective of verifying $\widehat{\nabla_\phi J} \rightarrow \nabla_\phi J^{\phi_{\mathrm{PE}}}$.

\begin{figure}[!htbp]
    \centering
    \includegraphics[width=1.0\textwidth]{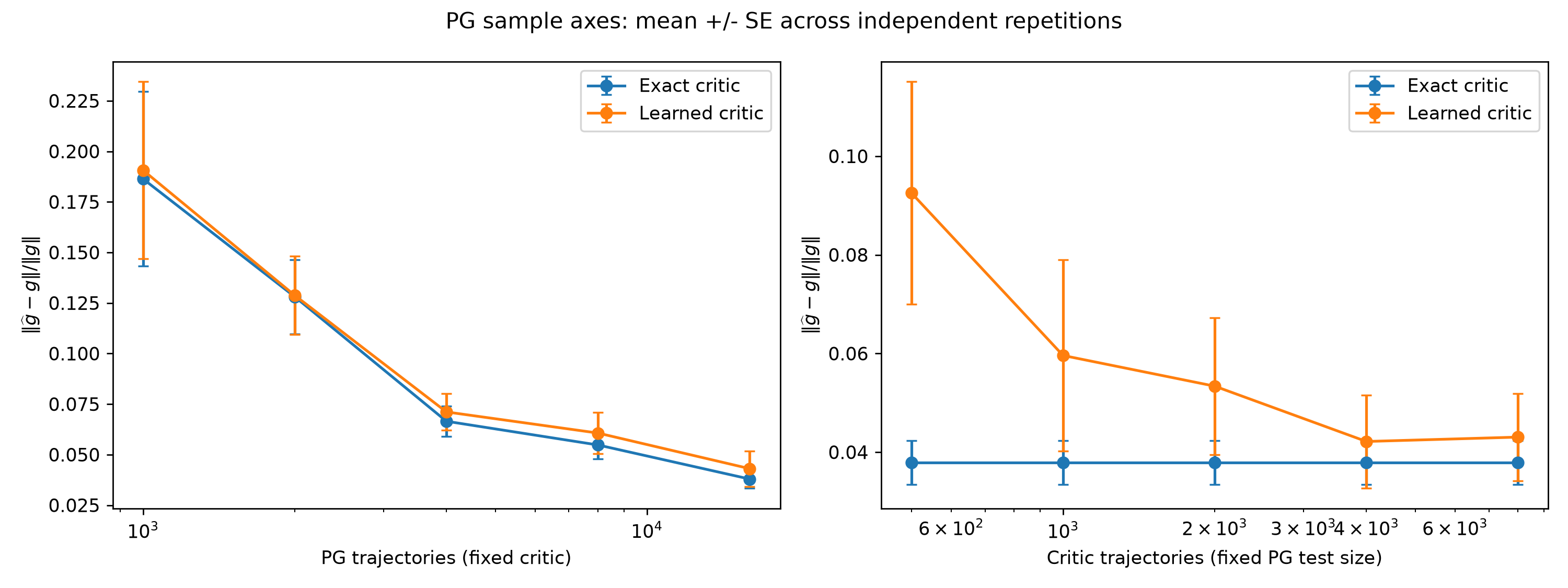}
    \caption{Comparison between the estimated and analytical policy gradients.}
    \label{fig:exp1-policy-gradient}
\end{figure}

{\it Actor Convergence.} The third test examines policy improvement through actor learning.
Under the cost-minimization convention, the actor parameters are updated according to
$$
\phi_{n+1}=\Pi_{\Phi}\left[\phi_n-\eta_{\phi, n} \widehat{\nabla_\phi J}\left(\phi_n\right)\right],
$$
where
$\phi_0 =(0.4,0.3, \log 0.35, \log 0.30), \Phi=[0,3]^2 \times[\log 0.02, \log 1]^2
$,
and the learning rate is $\eta_{\phi, n}=\frac{0.05}{(1+n / 20)^{0.6}}$.

The experiment is conducted for 150 actor iterations. At each iteration, 2000 independent trajectories are used for critic fitting, and another independent batch of 2000 trajectories is used for PG estimation.
Since the Gaussian actor parameters are time-independent, whereas the unrestricted optimal feedback gains are generally time-dependent, the appropriate benchmark is the optimizer within the static Gaussian policy class:
$$
\phi_{\text {stat }}^*=\arg \min _{\phi \in \Phi} J^\phi(0,1,1) .
$$
Independent numerical optimization yields
$$
\phi_{\text {stat }}^* \approx(1.80676,1.19329,-1.60944,-2.07944), \quad
J^{\phi_{\text {stat }}^*}(0,1,1) \approx 1.04741 .
$$
Accordingly, the actor experiment focuses on parameter convergence and the objective gap:
$$
\phi_n \longrightarrow \phi_{\text {stat }}^*, \quad J^{\phi_n}(0,1,1)-J^{\phi_{\text {stat }}^*}(0,1,1) \rightarrow 0.
$$
Figure~\ref{fig:exp1-actor-gain}, ~\ref{fig:exp1-actor-variance}, and ~\ref{fig:exp1-actor-objective-gap} are combined into a single figure with three panels, showing the evolution of the feedback gains, exploration variances, and objective gap, respectively. These comparisons jointly assess the mean and variance components of the Gaussian policy.

Furthermore, Figure~\ref{fig:exp1-mesh-refinement} examines whether the discrepancy between the finite-grid PE–PG estimates and their continuous-time theoretical references decreases as the time mesh is refined.

Taken together, Experiment 1 examines the accuracy of policy evaluation, policy-gradient estimation, and actor improvement, while the mesh-refinement analysis assesses the consistency of the finite-grid implementation with the continuous-time framework.

\begin{figure}[!htbp]
    \centering
    \includegraphics[width=0.6\textwidth]{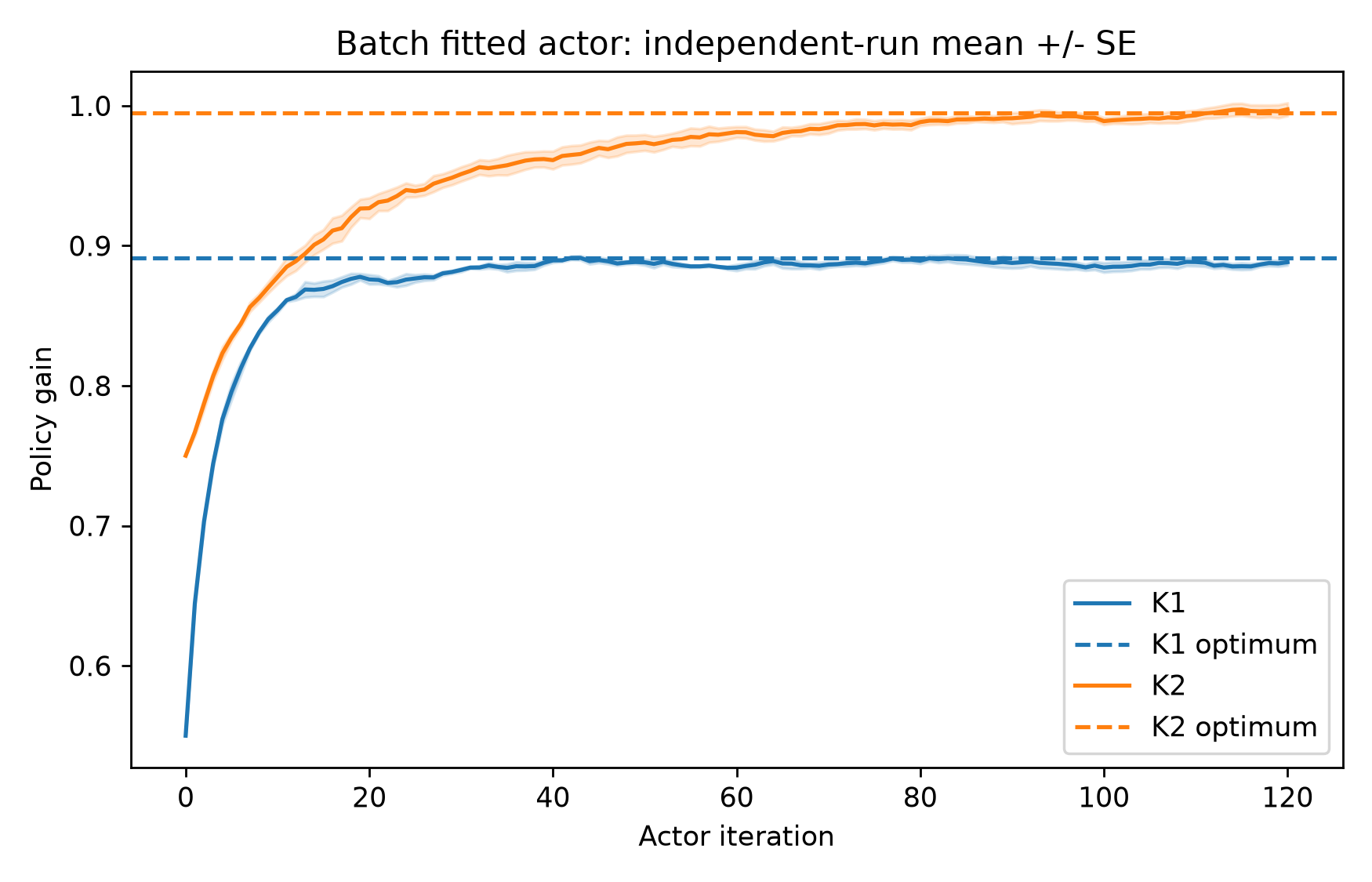}
    \caption{Convergence of the estimated actor gain toward its theoretical benchmark.}
    \label{fig:exp1-actor-gain}
\end{figure}

\begin{figure}[!htbp]
    \centering
    \includegraphics[width=0.6\textwidth]{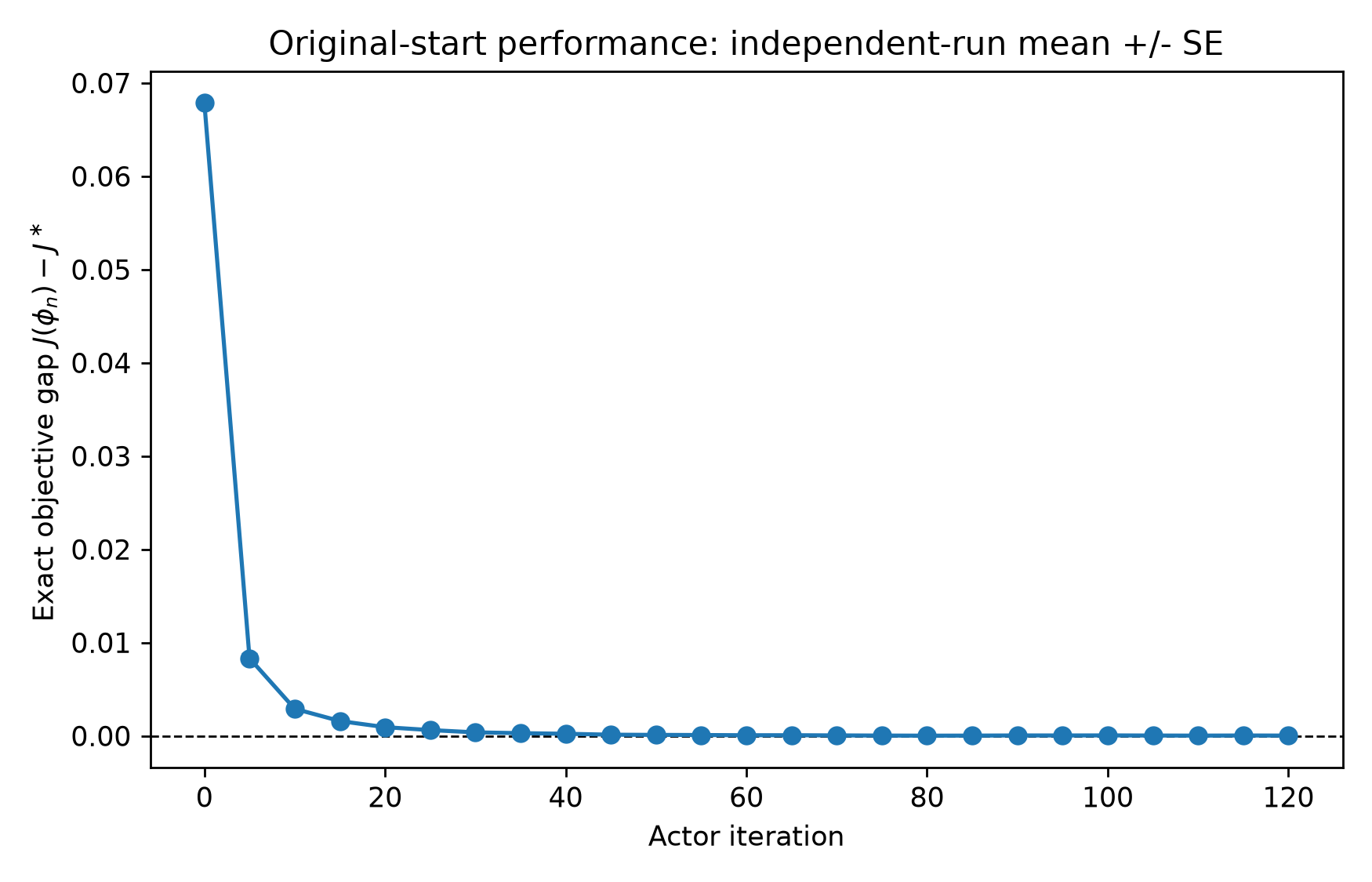}
    \caption{Evolution of the actor objective gap during learning.}
    \label{fig:exp1-actor-objective-gap}
\end{figure}

\begin{figure}[!htbp]
    \centering
    \includegraphics[
        width=0.60\textwidth,
        height=0.27\textheight,
        keepaspectratio
    ]{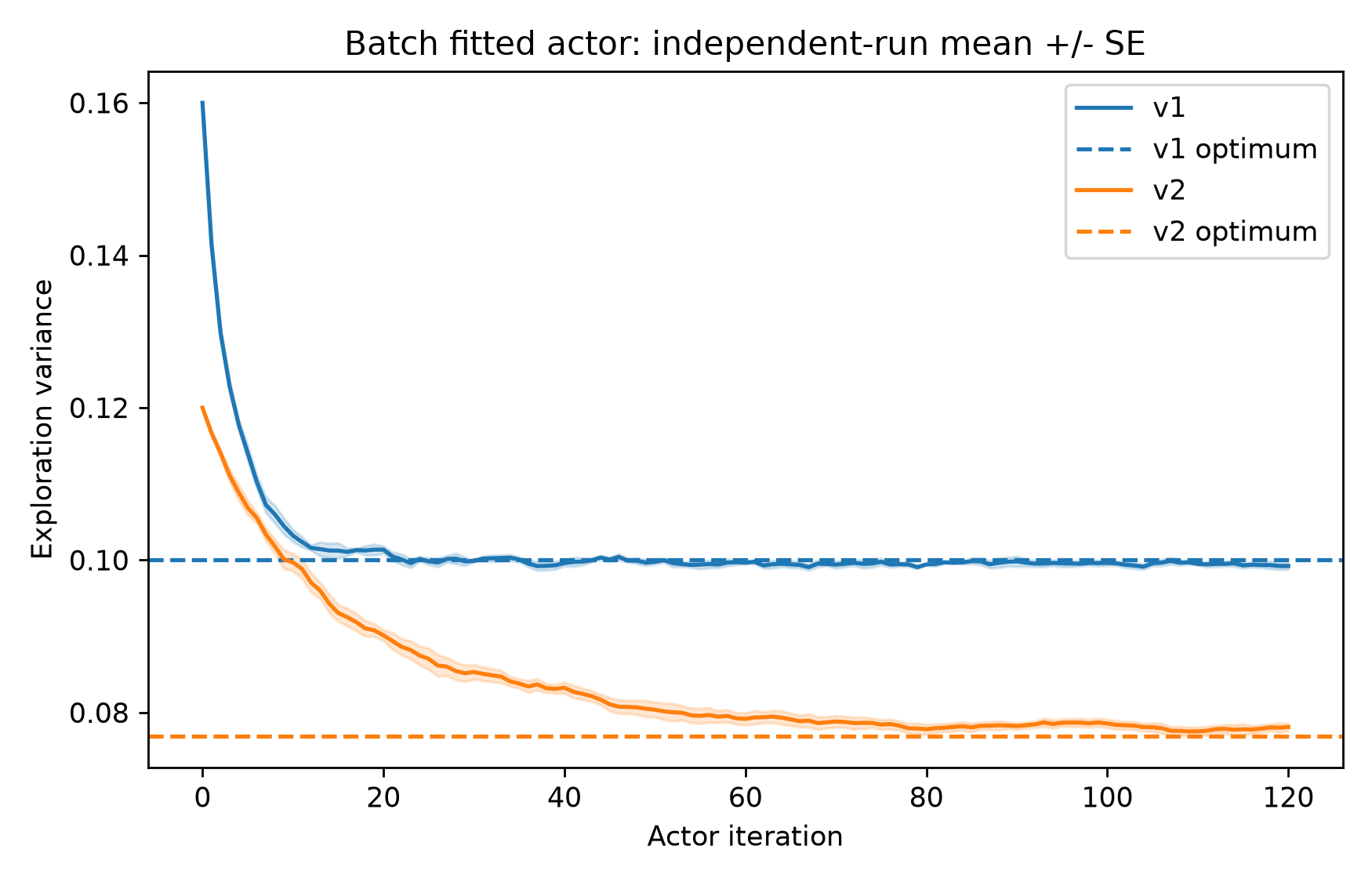}
    \caption{Evolution of the Gaussian exploration variance during actor learning.}
    \label{fig:exp1-actor-variance}

    \vspace{4pt}

    \includegraphics[
        width=0.85\textwidth,
        height=0.32\textheight,
        keepaspectratio
    ]{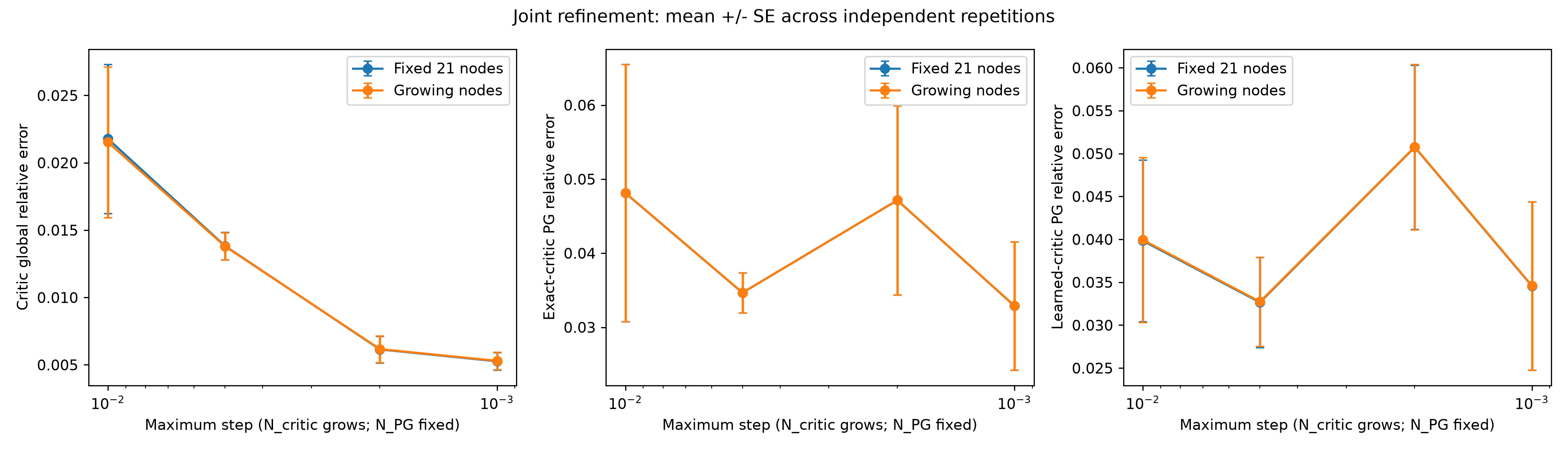}
    \caption{Mesh-refinement analysis of the finite-grid actor-critic approximation.}
    \label{fig:exp1-mesh-refinement}
\end{figure}

\subsection{Verification of the MSTDE decomposition}

The second experiment examines the additional regime-switching contribution to the MSTDE decomposition, given by
$\sum_{j \neq i} q_{i j}(J_j-J_i)^2$.
To isolate the effect of regime switching, we consider a regime-independent state process
\begin{equation}\label{eq:state-dynamic-of-MSTDE}
d X_t=-0.2 X_t d t+0.4 d W_t, \quad X_0=1, \quad t \in[0,1],
\end{equation}
together with a symmetric two-state Markov chain with generator
$Q=\left(\begin{array}{cc}
-\kappa & \kappa \\
\kappa & -\kappa
\end{array}\right) .$
The corresponding true value function is
$$
J(t, x)=\mathbb{E}_{t, x}\left[\int_t^1 X_s^2 d s+X_1^2\right]=P(t)x^2+C(t),
$$
where
$P(t)=2.5-1.5 e^{-0.4(1-t)}, C(t)=0.4(1-t)-0.6[1-e^{-0.4(1-t)}].$
Since the state dynamics and cost functions are independent of the regime, the true value function satisfies $J_1(t, x)=J_2(t, x)=J(t, x)$.

To isolate the regime-switching contribution to the MSTDE while preserving the terminal condition, we introduce the perturbed critic
$$
J_1^d(t, x)=J(t, x)+\frac{1-t}{2}, \quad J_2^d(t, x)=J(t, x)-\frac{1-t}{2} .
$$
Consequently,
$$
J_1^d(t, x)-J_2^d(t, x)=1-t .
$$
By the quadratic-variation decomposition established in the preceding MSTDE analysis, the limiting mean-square TD error is determined by the expected predictable quadratic variation of the martingale component of the corresponding semimartingale:
$$
\operatorname{MSTDE}(\kappa)=\mathcal{V}_W+\mathcal{V}_Q .
$$
Since the perturbation is independent of the state variable, the Brownian contribution is
\begin{align*}
\mathcal{V}_W&=0.64 \int_0^1 P^2(t) \mathbb{E}\left[X_t^2\right] d t\\
&=0.64 \int_0^1\left[2.5-1.5 e^{-0.4(1-t)}\right]^2\left[0.4+0.6 e^{-0.4 t}\right] d t \\
&\approx 0.93892.
\end{align*}
wheres the regime-switching contribution is
\begin{align*}
\mathcal{V}_Q =\kappa \int_0^1\left|J_1^d(t, x)-J_2^d(t, x)\right|^2 d t =\kappa \int_0^1(1-t)^2 d t=\frac{\kappa}{3}
\end{align*}
Therefore, the theoretical MSTDE benchmark is
$$
\operatorname{MSTDE}(\kappa)=0.93892+\frac{\kappa}{3} .
$$
This expression provides the principal prediction of Experiment 2. As the switching intensity $\kappa$ increases, the Brownian contribution remains unchanged, whereas the additional jump contribution increases linearly with slope $1 / 3$.

\textbf{Numerical implementation.}
We consider the switching intensities
$\kappa \in\{0,0.25,0.5,1,2,4,8\}$.
The baseline experiment uses $N=200$ time steps with $h=0.005$, and $M=20000$ independent trajectories are generated for each value of $\kappa$. To assess numerical robustness, we additionally consider
$N \in\{50,100,200,400\}, M \in\{2000,5000,10000,20000\}$.
The state process is simulated using the Euler-Maruyama scheme, while the regime process is generated using the exact transition matrix $e^{\mathcal{Q}_\kappa h}$, consistent with the original experimental configuration.

Using the simulated trajectories, the Brownian and predictable regime-switching contributions are estimated by
$$
\widehat{\mathcal{V}}_W  =\frac{1}{M} \sum_{m=1}^M \sum_{k=0}^{N-1} 0.64 P^2(t_k) X_{m, k}^2 h, \quad
\widehat{\mathcal{V}}_Q^{\mathrm{pred}} =\sum_{k=0}^{N-1} \kappa(1-t_k)^2 h .
$$
To further examine the MSTDE decomposition, we also compute the finite-grid TD increment associated with the perturbed critic from the same simulated trajectories:
$$
\delta_{m, k}^d= J_{\alpha_{m, t_{k+1}}^d}^d(t_{k+1}, X_{m, t_{k+1}})-J_{\alpha_{m, t_k}}^d(t_k, X_{m, t_k}) +h X_{m, t_k}^2 .
$$
The empirical squared TD statistic is then defined by
$$
\overline{\operatorname{MSTDE}}(\kappa)=\frac{1}{M} \sum_{m=1}^M \sum_{k=0}^{N-1}\left|\delta_{m, k}^d\right|^2 .
$$
This statistic can be computed directly from the simulated data without modifying the model or sampling parameters.
The numerical experiment primarily examines whether
$$
\overline{\operatorname{MSTDE}}(\kappa) \approx \widehat{\mathcal{V}}_W+\widehat{\mathcal{V}}_Q^{\text {pred }} \approx 0.93892+\frac{\kappa}{3} .
$$
Figure~\ref{fig:exp2-mstde-decomposition} illustrates the Brownian and jump contributions under different switching intensities. In particular, it examines whether the Brownian contribution remains approximately constant and whether the additional regime-switching contribution follows the predicted linear relationship with slope $1 / 3$.

Figure~\ref{fig:exp2-paired-mesh-refinement} examines whether the discrepancy between the empirical MSTDE and its theoretical quadratic-variation decomposition decreases as the time grid is refined.

By keeping the state dynamics and the true value function unchanged while varying only the switching intensity, this experiment separates the additional jump quadratic variation from the Brownian contribution. It therefore provides a direct numerical examination of the regime-switching MSTDE decomposition established in the preceding theoretical analysis.

\begin{figure}[!htbp]
    \centering
    \includegraphics[
        width=0.88\textwidth,
        height=0.28\textheight,
        keepaspectratio
    ]{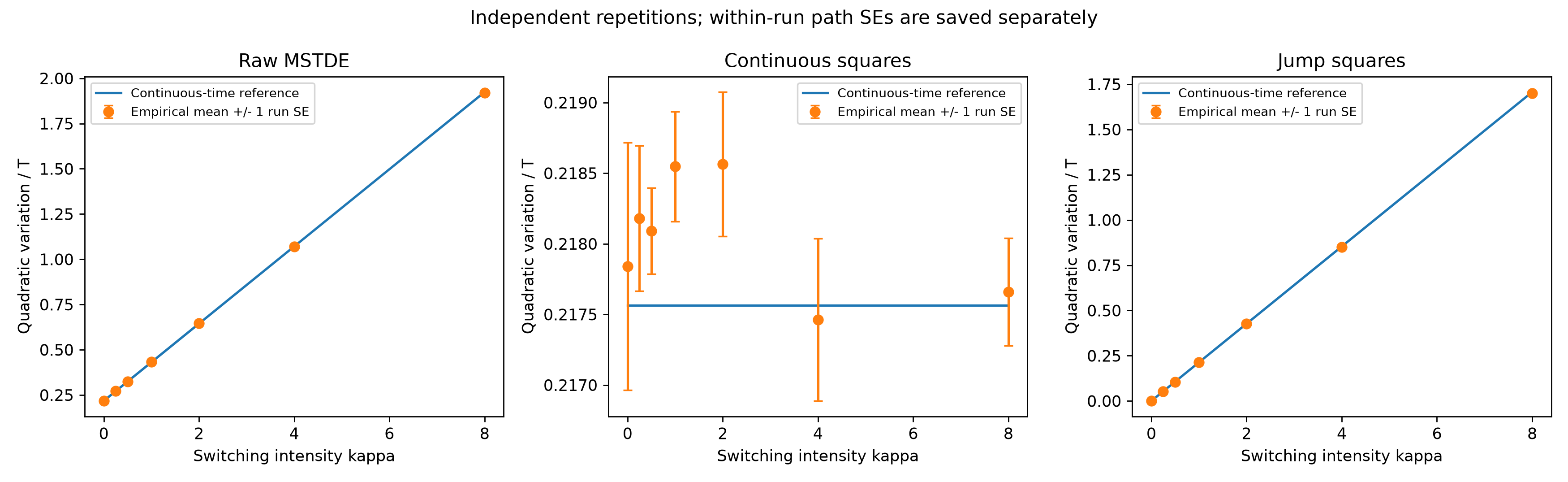}
    \caption{Decomposition of the MSTDE into the Brownian and regime-switching contributions for different switching intensities.}
    \label{fig:exp2-mstde-decomposition}

    \vspace{3pt}

    \includegraphics[
        width=0.88\textwidth,
        height=0.28\textheight,
        keepaspectratio
    ]{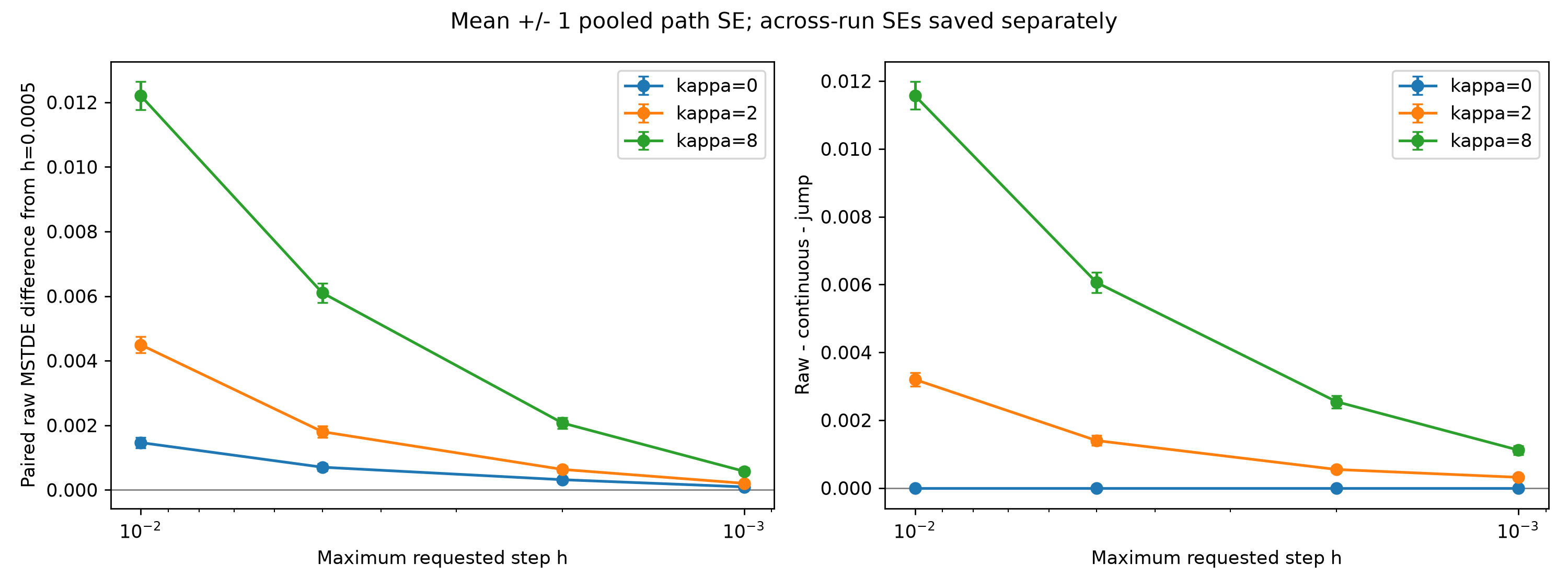}
    \caption{Mesh-refinement analysis for the MSTDE decomposition under different time discretizations.}
    \label{fig:exp2-paired-mesh-refinement}
\end{figure}

\section{Conclusions}
\label{sec:Conclusions}

This paper develops a unified martingale framework for policy evaluation, policy gradient, and Actor-Critic learning in continuous-time stochastic control with Markov regime-switching. The key feature of the framework is that the diffusion dynamics and regime transitions are treated within the same semimartingale structure. For policy evaluation, the value process is characterized through martingale conditions that account for both Brownian and regime-transition martingales. This leads to regime-switching extensions of martingale loss and martingale orthogonality methods. For policy improvement, the regime-dependent value functions remain coupled through the switching mechanism, but the resulting policy-gradient representation can be reformulated in terms of observable trajectory increments. Under an action-independent switching generator, this representation avoids explicit evaluation of the drift, diffusion coefficients, and regime transition rates. These results provide a common theoretical basis for the offline and online Actor-Critic algorithms developed in the paper.
The analysis also clarifies the role of regime-switching in data-driven continuous-time learning. Regime transitions affect the theoretical generator, the semimartingale decomposition of the value process, and its predictable quadratic variation, but their effects can be incorporated through observed value increments in the trajectory-based learning rules. This distinction is important for preserving the model-free character of the resulting algorithms. Together with the finite-grid consistency analysis, the proposed framework connects continuous-time PE and PG theory with implementable regime-switching Actor-Critic methods.

% \bibliographystyle{plainnat-doi}
% \bibliography{referencelist}

\end{document}